\documentclass{TPmod}
\usepackage{url}
\usepackage{setspace}
\usepackage{scrextend}
\usepackage{mathrsfs}
\usepackage{tocloft}

\usepackage{amsthm}
\usepackage{amssymb}
\usepackage[capitalize]{cleveref}
\usepackage{hyperref}
\usepackage{mathtools}
\usepackage{xy}
\usepackage{tikz}
\usepackage{tikz-cd}
\usepackage{tqft}
\usetikzlibrary{backgrounds}

\usepackage[normalem]{ulem}
\usetikzlibrary{shapes.geometric,arrows}

\newcommand{\forget}{\mathrm{forget}}

\newcommand{\Ind}{\mathrm{Ind}}

\DeclareMathOperator{\Symp}{Symp}

\def\P{\mathcal{P}}

\def\bC{\mathbb{C}}
\def\bD{\mathbb{D}}
\def\bE{\mathbb{E}}
\def\bF{\mathbb{F}}

\def\bM{\mathbb{M}}

\def\bP{\mathbb{P}}
\def\bQ{\mathbb{Q}}
\def\bR{\mathbb{R}}
\def\bS{\mathbb{S}}
\def\bT{\mathbb{T}}

\def\bZ{\mathbb{Z}}

\newcommand{\eps}{\varepsilon}

\newcommand{\cA}{\mathcal{A}}
\newcommand{\cB}{\mathcal{B}}
\newcommand{\cC}{\mathcal{C}}
\newcommand{\cD}{\mathcal{D}}
\newcommand{\cE}{\mathcal{E}}
\newcommand{\cF}{\mathcal{F}}
\newcommand{\cG}{\mathcal{G}}
\newcommand{\cH}{\mathcal{H}}
\newcommand{\cI}{\mathcal{I}}
\newcommand{\cJ}{\mathcal{J}}
\newcommand{\cK}{\mathcal{K}}
\newcommand{\cL}{\mathcal{L}}
\newcommand{\cM}{\mathcal{M}}
\newcommand{\cN}{\mathcal{N}}
\newcommand{\cO}{\mathcal{O}}
\newcommand{\cP}{\mathcal{P}}
\newcommand{\cQ}{\mathcal{Q}}
\newcommand{\cR}{\mathcal{R}}
\newcommand{\cS}{\mathcal{S}}
\newcommand{\cT}{\mathcal{T}}
\newcommand{\cU}{\mathcal{U}}
\newcommand{\cV}{\mathcal{V}}
\newcommand{\cW}{\mathcal{W}}
\newcommand{\cX}{\mathcal{X}}
\newcommand{\cY}{\mathcal{Y}}
\newcommand{\cZ}{\mathcal{Z}}

\newcommand{\x}{\times}

\def\bZ{\mathbb{Z}}

\renewcommand{\cong}{\simeq}

\begin{document}
\title{Bordism of flow modules and exact Lagrangians}

\author[]{Noah Porcelli\footnote{Imperial College Department of Mathematics, Huxley Building, 180 Queen's Gate, London SW7 2RH, U.K.\\
E-mail: \textsc{n.porcelli@imperial.ac.uk}\\
Supported by the Engineering and Physical Sciences Research Council through Standard Grant EP/W015889/1 and PhD Studentship 2261120}, Ivan Smith\footnote{Centre for Mathematical Sciences, University of Cambridge, Wilberforce Road, Cambridge CB3 0WB, U.K.\\
E-mail: \textsc{is200@cam.ac.uk}\\
Supported by the Clay Foundation and the Engineering and Physical Sciences Research Council through Frontier Reseach Grantnt EP/X030660/1}}
% \thanks{The second author is  mostly supported by single malt.}
\date{}

\address{Noah Porcelli,  Imperial College Department of Mathematics, Huxley Building, 180 Queen's Gate, London SW7 2RH, U.K.}
\address{Ivan Smith, Centre for Mathematical Sciences, University of Cambridge, Wilberforce Road, Cambridge CB3 0WB, U.K.}

\begin{abstract} {\sc Abstract:} 
For a stably framed Liouville manifold $X$, we construct a `Donaldson-Fukaya category over the sphere spectrum' $\scrF(X;\bS)$. The objects are closed exact Lagrangians whose stable Gauss maps are nullhomotopic compatibly with the ambient stable framing, and the morphisms are bordism classes of framed flow modules over Lagrangian Floer flow categories; this is enriched in modules over the framed bordism ring. We develop an obstruction theory for lifting quasi-isomorphisms in the usual Fukaya category to quasi-isomorphisms in appropriate truncations or quotients of $\scrF(X;\bS)$.  Applications include constraints on the smooth structure of exact Lagrangians in certain cotangent bundles and plumbings, constructing of non-trivial symplectic mapping classes  which act trivially on the integral Fukaya category for a wide class of affine varieties, and ruling out some autoequivalences of the Fukaya categories of certain plumbings from being induced by Liouville automorphisms.  
%Here we collect some lemurs to use to farm bordisms. Lemurs are fairly intelligent and have been known to use tools so this problem is probably tractable with current methods. Hopefully we show something. 
\end{abstract}

\maketitle
\thispagestyle{empty}
\setlength{\cftbeforesubsecskip}{-2pt}
% \newpage
\tableofcontents
\section{Introduction}

\subsection{Context}  
This paper belongs to the developing subject of `Floer homotopy theory' in the sense of \cite{CJS,AB}. 
Constructing homotopical refinements of Floer theory has been a longstanding goal in the subject. 
The  recent advances of Abouzaid-Blumberg \cite{AB}, Bai-Xu \cite{BX}, Rezchikov \cite{Rezchikov} and others have all focussed on the Hamiltonian case: building on work of Large \cite{Large}, here we use Floer homotopy theory in the Lagrangian setting, away from the case of cotangent bundles where methods of generating functions are available. It would be interesting to develop an alternative approach to these questions using microlocal sheaves \cite{nadler-shende,Jin}.

Our purpose is two-fold:
\begin{enumerate}
    \item We give a  self-contained construction of an associative and unital  `spectral Donaldson-Fukaya category' $\scrF(X;\bS)$ for a Liouville manifold $X$ with $TX$ stably trivial, whose objects are Lagrangians $L$ whose stable Gauss map is trivial compatibly with the ambient stable framing (`spectral Lagrangian branes'). The morphism groups of $\scrF(X;\bS)$ are graded modules over the framed bordism ring $\Omega_*^{fr}(\ast)$ of a point, and quasi-isomorphism in $\scrF(X;\bS)$ retains information about the framed bordism class of $L$ in $\Omega_n^{fr}(X)$.
\item We develop an obstruction theory for lifting quasi-isomorphisms in $\scrF(X;\bZ)$ of the  graded Lagrangians underlying spectral Lagrangian branes, to quasi-isomorphisms in $\scrF(X;\bS)$, or more generally in certain `Baas-Sullivan truncations' of $\scrF(X;\bS)$.  This has a number of concrete applications to symplectic topology, detailed in Section \ref{Sec:results}.
    \end{enumerate}

Our work has at least two very substantial intellectual debts.

\begin{enumerate}
    \item On the philosophical side, we follow the mantra of Abouzaid and Blumberg to do `Floer homotopy without homotopy' and define the morphism groups of $\scrF(X;\bS)$ via bordism classes of `flow modules', close cousins of the  flow categories defined by  compactified moduli spaces of holomorphic strips. 
    \item On the technical side, we appeal to work of Fukaya, Oh, Ohta, Ono \cite{FO3:smoothness}, cf. also \cite{Li-Sheng},  and Large\footnote{Large's thesis also constructs a spectral Donaldson-Fukaya category, not using the language of flow modules but  staying closer to spectra. We have been heavily influenced by his work, but have found it easier to pin down some details in our chosen setting.} \cite{Large} asserting  that compactified moduli spaces of Floer strips (and some more general holomorphic curves) in an exact manifold are smooth manifolds with corners; moreover, for spectral branes, those manifolds with corners are themselves stably framed compatibly with breaking. We summarise these results in Section \ref{Sec:technical}. 
\end{enumerate}

The first point above means that the  interpretation of our `spectral Fukaya category' in terms of spectra is strictly conjectural (but is expected to follow from ongoing work of Abouzaid and Blumberg).

Theorem \ref{thm:main} illustrates how one can apply this technology to concrete problems.  It seems reasonable to expect further development in these directions.

\begin{rmk}
    Because of our intended applications, we focus on constructing a Donaldson-Fukaya category ``over the sphere spectrum" assuming a stable framing on $X$. For any Liouville symplectic manifold with $2c_1(X)=0$, our technology also yields a category $\scrF(X;MO)$ whose objects are arbitrary compact exact Lagrangians, and which is enriched in $\bZ$-graded $\bZ/2$-vector spaces which are modules over the unoriented bordism ring $\Omega_*(\ast)$.
\end{rmk}

Abouzaid and Blumberg are developing a quasi-category (stable $\infty$-category) of Flow categories and flow bimodules. A key point of this paper is that it is already interesting to construct and investigate the spectral Donaldson-Fukaya category (i.e. without the $A_{\infty}$-structure). Crucially, it is possible to perform at least some computations, by leveraging information from the `ordinary' Fukaya category. That leveraging relies on the fact that spectra, or flow modules, admit truncation functors, in which one retains the information from moduli spaces of flow lines only up to a certain dimension $k$.  When $k=1$ the theory is equivalent to the classical  theory of Floer cochain complexes, and there is an obstruction theory which governs when one can pass information `up the chain' between successive truncations. Under suitable degree hypotheses, the obstructions vanish for elementary reasons.  

In our applications, most obstructions vanish for degree reasons, but not all: one must be killed `geometrically'. For this, we develop an analogue in flow modules of working over a quotient of the sphere spectrum, namely we consider bordism classes of `flow modules with Baas-Sullivan singularities', inspired by the classical Baas-Sullivan theory \cite{Baas, Perlmutter} of bordism for manifolds allowed to have singularities $\mathrm{Cone}(P)$ modelled on cones over some specified collection of manifolds $\{P\}$.

Concretely,  our applications focus on Lagrangian homology spheres, where the combination of degree arguments and working with Baas-Sullivan singularities together suffice to kill all obstructions.  

\subsection{Results\label{Sec:results}} 
Let $\Omega_n^{fr}$ denote the $n$-th framed cobordism group, which is also the $n$-th stable stem $\pi_n^{st} = \pi_n(\bS)$. More generally we write $\Omega_n^{fr}(Y)$ for the framed bordism group of a topological space $Y$.

Let $X$ be a Liouville manifold with $TX$ stably trivial as a complex vector bundle. A `spectral Lagrangian brane' (Definition \ref{defn:spectral_brane})  is a closed exact Lagrangian submanifold $L\subset X$ equipped with a nullhomotopy of its stable Lagrangian Gauss map which is compatible with the given stable framing of $X$, and a choice of grading in the sense of \cite{Seidel:graded} (the nullhomotopy of the Gauss map implies existence of a grading, distinguished up to even shift).

\begin{thm}\label{thm:exists}
    There is a unital associative category $\scrF(X;\bS)$ with objects spectral Lagrangian branes, whose morphism groups are $\bZ$-graded modules over $\Omega^{fr}_*$, and with the property that isomorphic objects in $\scrF(X;\bS)$ represent the same class in $\Omega_n^{fr}(X)$.
\end{thm}
In Proposition \ref{prop: pss}, we identify the endomorphisms of a spectral Lagrangian brane $L$ with the framed bordism groups $\Omega_*^{fr}(L)$.
\begin{rmk}\label{rmk:forget}
    There is a forgetful (`truncation') functor $\tau_{\leq0}: \scrF(X;\bS) \to H\scrF(X;\bZ)$ with image the full subcategory of the integral Donaldson-Fukaya category of those exact Lagrangians which admit spectral brane structures. 
    \end{rmk}

\begin{rmk}
    In this paper we work with homological grading: under the truncation functor $\tau_{\leq 0}$ of Remark \ref{rmk:forget}, degree $i$ morphisms in $\tau_{\leq 0}\scrF(X;\bS)$ correspond to $CF^{-i}(L,K;\bZ)$.
\end{rmk}

A constraint in applying Theorem \ref{thm:exists} is the difficulty in verifying isomorphism in $\scrF(X;\bS)$.

The multiplication map 
\[
\pi_1(\bS) \times \pi_{n-1}(\bS) \to \pi_n(\bS) \qquad \mathrm{equivalently} \qquad \Omega_1^{fr} \otimes \Omega_{n-1}^{fr} \to \Omega_n^{fr}
\]
define a subgroup $\mathrm{image}(\eta) \subset \Omega_n^{fr}$ (here $\eta$ is the non-trivial element in $\pi_1(\bS) = \bZ/2$). Since $\eta$ is $2$-torsion, so is $\mathrm{image}(\eta)$.

\begin{thm} \label{thm:main}
    Let $X$ be a Liouville manifold with $TX$ stably trivial (as a complex vector bundle) of dimension $2n > 6$, and fix a stable trivialisation $\phi$. Let $L, K\subset X$ be closed exact Lagrangian integer homology spheres  whose stable Gauss maps are stably trivial compatibly with $\phi$. If $L$ and $K$ represent quasi-isomorphic objects of $\scrF(X;\bZ)$ up to shift, then they represent the same class in the quotient $\Omega_n^{fr}(X)/(\Omega_1^{fr}(X) \cdot \Omega^{fr}_{n-1})$. In particular, $[L]$ and $[K]$ agree in $\Omega_n^{fr} / \mathrm{image}(\eta)$.
\end{thm}

We consider $\Omega_n^{fr; (n-1)\mathrm{-BS}}(X) \cong \Omega_n^{fr}(X)/(\Omega_1^{fr}(X) \cdot \Omega^{fr}_{n-1})$ the group of bordism classes of stably framed manifolds with Baas-Sullivan singularities of dimension $n-1$ mapping to $X$; see Section \ref{Sec:BS} for details. 
To prove the theorem, we prove that $L$ and $K$ are isomorphic in $\tau_{\leq n}\scrF(X;\bS)_{(n-1)\mathrm{-BS}}$, a category which has the same objects as $\scrF(X;\bS)$, but which only considers manifolds of dimension $\leq n+1$ and allowing Baas-Sullivan singularities of dimension $n-1$.

\begin{rmk}
    The first constraints on the topology of Lagrangian submanifolds going beyond homotopy type were obtained by Abouzaid \cite{Abouzaid:framed}, who showed that an exact Lagrangian homotopy sphere  $L \subset T^*S^{4k+1}$ framed bounds.  Then  \cite{ES,ES2,EKS} showed that if homotopy $(2k+1)$-spheres $\Sigma, \Sigma'$ have $T^*\Sigma \cong T^*\Sigma'$ then $[\Sigma] = \pm[\Sigma'] \in \Theta_{2k+1}/bP_{2k}$. Using the Hopf correspondence \cite{Torricelli} Torricelli showed there is a manifold $P \simeq \bC\bP^4$ which does not embed as  a Lagrangian in $T^*\bC\bP^4$, and a `circle bundle construction' shows that if $\bR\bP^{4k-1} \# \Sigma \subset \bC\bP^{4k-1}$ is Lagrangian, then $\Sigma \# \Sigma$ framed bounds. Crucially, the results in \cite{Abouzaid:framed, ES} rely on a displacing Hamiltonian isotopy for a suitable Lagrangian ($S^{2k+1} \subset \bC\bP^k \times \bC^{k+1}$ respectively an immersed $S^n \hookrightarrow \bC^n$), and build an explicit framed bounding manifold from the union of a moduli space of Floer holomorphic curves and an `artificial cap' which closes off a boundary component of broken curves. The known generalisations have been  largely tied to settings umbilically close to the cotangent bundle of an odd-dimensional sphere.  The point of Theorem \ref{thm:main} is to provide a general statement for Weinstein manifolds in which a displacing isotopy is not required. 
\end{rmk}

The following are examples of some of the applications proved in Section \ref{Sec:Applications}.

\begin{ex}[\cite{AAGCK}] \label{ex:1}
If $X = T^*S^n$ and $L$ is the zero-section then any $K$ satisfies the hypotheses, and we conclude that $K\# K$ framed bounds.
\end{ex}

\begin{ex} \label{ex:2}
If $X$ is the $A_2$-plumbing $T^*S^n \#_{pt} T^*S^n$ with $n=8$, we conclude that any exact $K$ with vanishing Maslov class is necessarily diffeomorphic to $S^8$.
\end{ex}

\begin{ex} \label{ex:3}
Let $X$ be the plumbing of $T^*S^n$ and $T^*\Sigma$ for a homotopy sphere $\Sigma$. Suppose $n$ is even. If $\Sigma$ does not have order 2 in the quotient $\Theta_n / bP_{n+1}$ of the group of homotopy spheres by those that framed bound, then the spherical twist autoequivalence $T_{\Sigma}$ of $\scrF(X;\bZ)$ is not realised by any symplectomorphism.
\end{ex}

\begin{ex} \label{ex:4}
Attach a Weinstein four-handle to $T^*S^7$ to obtain a Stein manifold $X$ homotopy equivalent to $S^7 \vee S^4$. A Lagrangian sphere $L\subset X$ has a class in $\pi_7(S^4)$ whose image in $\pi_3^{st} = \bZ/24$ is two-torsion.
\end{ex}

\begin{ex} \label{ex:5}
Let $d\geq 3$ and let $n$ satisfy the constraints (a) $n \in \{0,4,6,7\}$  mod $8$ and (b) there is some element not of order 2 in $\Theta_{n+1}/bP_{n+2}$. If $X_{n,d} \subset \bC^{n+1}$ is the general affine hypersurface of degree $d$ and complex dimension $n$ then there is a non-trivial element of the symplectic mapping class group of $X$ which acts trivially on the compact and wrapped Fukaya categories $\scrF(X;\bZ)$ and $\scrW(X;\bZ)$.
\end{ex}

Example \ref{ex:1} was independently established by Abouzaid, Alvarez-Gavela, Courte and Kragh \cite{AAGCK} using methods more specifically tailored to working in cotangent bundles (generating functions and tube spaces). As far as we know Example \ref{ex:2} is new, and one of the first constraints on the smooth structure of an exact Lagrangian away from the case of cotangent bundles. Example \ref{ex:3} resolves a question raised by Seidel in \cite{Seidel:exotic}. Quasi-isomorphic exact Lagrangians are known to be homologous (the argument is briefly recapped below), but the question of when quasi-isomorphic exact Lagrangian spheres are homotopic is more subtle; Example \ref{ex:4} gives a constraint in that direction.
Example \ref{ex:5} extends work of \cite{ER14}, and shows that `with positive probability' (with respect to natural density) the symplectomorphism group of an affine hypersurface does not act faithfully on the Fukaya category. 

\begin{figure}[h]
\begin{center}
\begin{tikzpicture}[framed,background rectangle/.style={double, thick,draw=gray}, loose background]
\begin{scope}[scale=0.8,every tqft/.style={transform shape}]
\node[tqft/pair of pants,draw] (a) {};
\node[tqft/reverse pair of pants,draw,anchor=incoming boundary 1, tqft/boundary lower style = {draw,dashed}] (b) at (a.outgoing boundary 2) {};
\node[tqft/cylinder to next,draw,anchor=incoming boundary 1,tqft/boundary lower style = {draw,dashed}] (c) at (a.outgoing boundary 1) {};
\node[tqft/cylinder to next,draw,anchor=outgoing boundary 1] (e) at (b.incoming boundary 2) {};
\end{scope}
\end{tikzpicture}
\caption{Framed bordism.}
\end{center}
\end{figure}
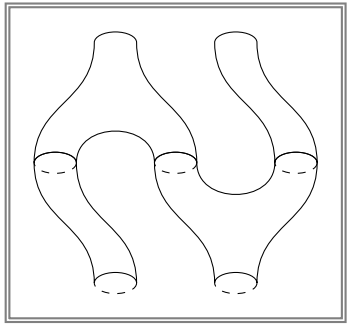

\subsection{Motivation, techniques and a sketch of the argument} 
It is classical that if $L$ and $K$ are exact Lagrangian branes in an exact symplectic manifold $X$, and if $L$ and $K$ are quasi-isomorphic in $\scrF(X;R)$ (for a commutative ring $R$, e.g. $R = \bZ/2$ or $R = \bZ$ if $L, K$ are spin) then $L$ and $K$ represent the same homology class in $H_n(X;R)$. \par   A proof considers a moduli space of strips depicted in Figure \ref{Fig:CO_homology_relation}, where the morphisms $\alpha \in CF(L,K;R)$ and $\beta \in CF(K,L;R)$ are inverse quasi-isomorphisms i.e. $\alpha \cdot \beta = \id_K$ and $\beta \cdot \alpha = \id_L$. Evaluation at the variable marked point along the dotted line defines a homology relation witnessing the equality $[L] = [K] \in H_n(X;R)$.  The breakings coming from the evaluation marked point escaping to the boundary of the strip give boundary strata which, by exactness, involve moduli spaces of constant discs which (algebraically) sweep the fundamental class of the boundary component.

\begin{figure}[ht]
\begin{center} 
\begin{tikzpicture}[scale=0.8]

\draw[semithick] (-8,0) -- (-2,0);
\draw[semithick] (-8,-2) -- (-2,-2);
\draw (-1.75,-1) node {$\alpha$};
\draw (-8.25,-1) node {$\beta$};
\draw (-5,0.5) node {$K$};
\draw (-5,-2.5) node {$L$};
\draw[semithick,dashed] (-5,0) -- (-5,-2);
\draw[fill,gray] (-5,-1.5) circle (0.1);
\draw (-5.5,-1.5) node {$\mathrm{ev}$};

	\draw[dashed,color=gray] [ ->] (-0.5,-0.5) -- (1.5,0.5);
	\draw[dashed,color=gray] [ ->] (-0.5,-1.5) -- (1.5,-2.5);

\draw[semithick] (2,1)--(8,1);
\draw[semithick] (2,-0.5) -- (8,-0.5);
\draw[semithick] (5,1.75) circle (0.75cm);
\draw[semithick,dashed] (5,-0.5) -- (5,1);

\draw (5,1.75) node {\small{$CO(K)$}};
\draw (5,-3.75) node {\small{$CO(L)$}};

\draw[semithick] (2,-3)--(8,-3);
\draw[semithick] (2,-1.5) -- (8,-1.5);
\draw[semithick] (5,-3.75) circle (0.75cm);
\draw[semithick,dashed] (5,-1.5) -- (5,-3);

\end{tikzpicture}
\end{center}
\caption{Homology relation from quasi-isomorphism and the length two open-closed map\label{Fig:CO_homology_relation}}

\end{figure}
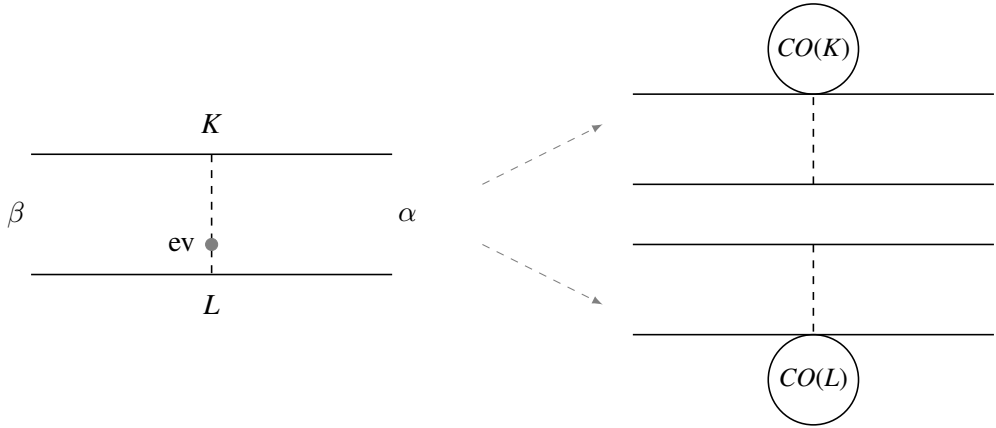

More algebraically, the boundary strata on the right of the figure can be understood in terms of the open-closed map
\[
\mathrm{OC}: CF^*(L,L) \to C_*(X)
\]
which sends the unit to the fundamental class $[L] \in H_n(M;\bZ)$. Whilst the full open-closed map involves the $A_{\infty}$-structure on the Hochschild complex, the `length two' version defines a map
\[
\mathrm{Cone}\left( CF(L,K) \otimes CF(K,L) \stackrel{\mu}{\longrightarrow} CF(L,L) \oplus CF(K,K) \right) \ \longrightarrow C^*(X),
\]
where $\mu(\alpha,\beta) = (\alpha\cdot \beta, \beta \cdot \alpha)$, which uses only homological  information from the Donaldson-Fukaya category. 

\begin{rmk}
    The previous argument is folklore and well-known to experts; a sketch is given for instance in \cite[Section 3]{YangLi}. There the homology relation defined by the universal family of holomorphic strips is interpreted as the existence of a `bordism current' between the two fundamental classes.
\end{rmk}
A corollary of this fact is the following, which tells us that that the Fukaya category detects some bordism-theoretic information.
\begin{cor}\label{cor:pontrjagin}
    In the situation outlined above with $R = \bZ$ or $\bQ$, then $L$ and $K$ have the same Pontrjagin numbers.
\end{cor}
\begin{proof}
    Since $TL \otimes \bC$ is isomorphic to $TX|_L$, the Pontrjagin numbers of $L$ are integrals of the form $\int_L p = \int_L c|_L$, where $p = p_{i_1}(L) \ldots p_{i_k}(L) \in H_n(L)$ is a product of Pontrjagin classes of $L$ and $c=c_{i_1}(X) \ldots c_{i_k}(X) \in H_n(X)$ is the corresponding product of Chern classes. But this latter expression is equal to $[L] \cdot c$ where $[L] \in H_n(X)$ is the fundamental class; since $[L]=[K]$ the result follows.
\end{proof}
\begin{rmk}
    We recall that it is a classical fact that two closed manifolds are bordant if and only if they have the same Stiefel-Whitney numbers, and two oriented closed manifolds are oriented bordant if and only if they have the same Stiefel-Whitney and Pontrjagin numbers.
\end{rmk}
\begin{rmk}
    Dualising the argument arising from Figure \ref{Fig:CO_homology_relation} at chain level, and using the whole complex (rather than just focusing on the fundamental class), yields a homotopy commutative diagram 
    \[
    \xymatrix{
    C^*(X) \ar[r] \ar@{=}[d] & C^*(L) \ar[d]^-{x \mapsto \alpha x \beta} \\
    C^*(X) \ar[r] & C^*(K)
    }
    \]
    with the horizontal maps being the classical restrictions (by exactness of $L$ and $K$). This shows that the distinguished isomorphism $\lambda: H^*(L;\bZ) \to H^*(K;\bZ)$ arising from the Floer quasi-isomorphism identifies the Pontrjagin classes of $L$ and $K$, $\lambda (p_j(L)) = p_j(K)$. This strengthens Corollary \ref{cor:pontrjagin}.
\end{rmk}

If $L$ carries some further tangential structure, e.g. a stable framing, it is natural to wonder if Floer theory detects its class in $\Omega_n^{fr}(X)$.

Suppose now $X$ is a stably framed symplectic manifold and one has defined a spectral Donaldson-Fukaya category over the sphere spectrum $\bS$ with morphism groups $CF(L,K;\bS)$ for spectral Lagrangian branes $L$ and $K$. If $L$ and $K$ are integer homology spheres which are quasi-isomorphic in the usual integral Fukaya category, a version of the Whitehead theorem shows there is $\theta \in \pi_{n-1}(\bS)$ with
\[
CF(L,K;\bS) = \mathrm{Cone}(\Sigma^{n-1}\bS \stackrel{\theta}{\longrightarrow} \bS).
\]
For instance, 
if  $L \pitchfork K = \{p,q\}$ (transversely), the moduli space of holomorphic strips asymptotic to $p,q$, modulo the re-parametrization action of $\bR$, is a compact $(n-1)$-manifold $\overline \cM_{pq}$. Vanishing of the stable Gauss maps of $L,K$ imply that $\overline \cM_{pq}$ is stably framed. The Pontrjagin-Thom construction then associates to $\overline \cM_{pq}$ an element in the stable stem $\pi_{n-1}^s$, which is exactly $\theta$. 

If $\theta = 0$ then $L$ and $K$ would be quasi-isomorphic over $\bS$, and one would expect to infer equality of their framed bordism classes in $\pi_n(\bS)$. In particular, if one could pass to working over a `quotient spectrum' $R = \bS / \pi_{n-1}(\bS)$ then one would kill the obstruction, without necessarily killing all the information we're trying to access (on the bordism classes of $L$ and $K$, which live in a different degree).  The main theorem in this paper realises this idea, working in a formalism of a spectral Donaldson-Fukaya category defined through bordism classes of flow modules, and with the role of the quotient spectrum $R$ being played by bordism with Baas-Sullivan singularities.

\begin{rmk} \label{rmk:dimension_constraint}
    There are structured categories of ring spectra which contain all colimits, and in which one can form a quotient such as $\bS/\theta$ for $\theta \in \pi_{n-1}(\bS)$. Since killing a class in stable homotopy will also kill its Whitehead self-products, the homotopy groups of this spectrum differ from those of $\bS$ not just in degree $n-1$ but also in larger degrees, starting from degree $2(n-1)$.  Since $2(n-1)>n$ (the degree in which the bordism information we seek to constrain lies), this is not problematic for the heuristic sketch above. The same phenomenon  appears in the setting of flow modules with Baas-Sullivan singularities modelled over cones on $k$-manifolds, where we require products of manifolds with such singularities to canonically carry the same structure. This relies on working with truncations of flow modules,   keeping track of moduli spaces only up to some given dimension $n$,  chosen so that $2k-1>n$, cf. Remark \ref{rmk:dimension_constraint_2}.
\end{rmk}

\begin{rmk}
    Part of the ongoing work of Abouzaid and Blumberg, constructing foundations of Floer homotopy theory ``without spectra'', has now appeared in \cite{Abouzaid-Blumberg:I}, after this paper was released. In particular, this includes a proof of a version of Conjecture \ref{conj:AB}.\par 
    Our treatment of stable framings has also been influenced directly by their paper.
\end{rmk}

    \subsection{Acknowledgements}

Thanks to Mohammed Abouzaid, Dani Alvarez-Gavela, Sylvain Courte, Amanda Hirschi, Thomas Kragh and Tim Large for helpful discussions, and to Hiro Lee Tanaka for pointing out a formatting error in a previous version of the paper.
%Also thanks to https://www.lemurconservationnetwork.org/ for the excellent work they do supporting lemur populations in the wild.

This paper is based upon work supported by the National Science Foundation under Grant No.~1440140, while the authors were in residence at the Mathematical Sciences Research Institute in Berkeley, California, during the Fall semester 2022.   I.S.~is grateful to the Clay Foundation for support at MSRI as a Clay Senior Scholar, and to EPSRC  for support through Frontier Research grant EP/X030660/1 (in lieu of an ERC Advanced Grant). N.P.~is grateful to the EPSRC for support through PhD studentship 2261120 and standard grant EP/W015889/1.

\section{Applications}\label{Sec:Applications}

In this section, we provide several concrete applications of Theorem \ref{thm:main} to symplectic topology. In this section we only require the statement of Theorem \ref{thm:main}, and apart from that will not require any results from the rest of the paper, and in particular will not refer directly to any flow-categorical technology.

By a homology sphere, we always mean an integral homology sphere. Our conventions for the `ordinary'  Fukaya category follow those from \cite{Seidel:book}.

\subsection{Smooth structures on Lagrangian submanifolds\label{Sec:first_application}}

\begin{lem}\label{lem:framed}
    If $\Sigma$ is a homology sphere, then $T\Sigma$ is stably trivial.
\end{lem}

\begin{proof}
    See \cite[p.70]{Kervaire}.
\end{proof}

It follows that $T^*\Sigma$ admits a stable framing, so we can talk about the category $\scrF(T^*\Sigma;\bS)$. The zero-section $\Sigma$ certainly forms one object of this category.

\begin{lem}
    If $\Sigma$ is a homotopy sphere and $L \subset T^*\Sigma$ is a closed exact Lagrangian submanifold, then $L$ admits the structure of a spectral Lagrangian brane.
\end{lem}

\begin{proof}
    The stable Gauss map of $L$ is zero on homotopy groups by \cite[Theorem A]{ACGK}. $L$ is a homotopy sphere by \cite{Abouzaid:homotopy} so this implies the stable Gauss map of $L$ is nullhomotopic.
\end{proof}

Now assume that $L \subset T^* \Sigma$ is a closed exact Lagrangian, and either $\Sigma$ is a homotopy sphere, or $\Sigma$ is a homology sphere and $L$ admits the structure of a Lagrangian brane.
\begin{cor}The class $[L]-[\Sigma]$ in $\Omega_n^{fr}(\ast)$ lies in the image of multiplication by $\eta \in \Omega^1_{fr}(\ast)$.
\end{cor}

\begin{proof}
    $L$ and the zero-section $\Sigma$ are quasi-isomorphic in $\scrF(T^*\Sigma;\bZ)$ by \cite{FSS}. The result then follows from Theorem \ref{thm:main}.
\end{proof}

Since $\eta$ is 2-torsion, we conclude that if $\Sigma=S^n$, $L \# L$ framed bounds. For instance, this provides new constraints on exact Lagrangian submanifolds in $T^*S^{10}$. The group $\Theta_{10} / bP_{11} = \bZ/6$ is not 2-torsion, and four of the five exotic differentiable structures on the ten-dimensional topological sphere are ruled out from admitting an exact Lagrangian embedding.

We now consider the $A_2$-plumbing $A_2^{(n)} :=T^*S^n \#_{pt} T^*S^n$.  This is also the affine hypersurface
\[
\left\{ \sum_{i=0}^n z_i^2 + w^3 = 1 \right\} \, \subset \, \bC^{n+1} 
\]
equipped with the restriction of the standard symplectic form. There is a symplectic fibration over the space $\scrC$ of cubic polynomials $p(w)$ with distinct roots summing to zero, with fibre $\{\sum z_i^2 + p(w) = 0\}$ over $p$. 

\begin{lem}
    $A_2^{(n)}$ admits a polarisation for which the two core components admit the structure of spectral Lagrangian branes.
\end{lem}

\begin{proof}
    We follow an argument from \cite{SS:localization} and \cite{KeatingSmith}. There is an exact sequence
    \begin{equation} \label{eqn:total_space_bundles}
    T (A_2^{(n)}) \to T\bC^{n+1} \to T\bC
    \end{equation}
    A (contractible) choice of Hermitian metric gives a splitting, which gives a stable trivialisation of the tangent space. 

The Lagrangian core components $S_a$ and $S_b$ arises as vanishing cycles. More precisely there is a partial compactification $\overline{\scrC}$ of $\scrC$ of cubics with at most a double root, and an extension of the symplectic fibration over $\overline{\scrC}$ with smooth total space and having nodal fibres over the discriminant (locus of cubics with a double root), see e.g. \cite{KhovanovSeidel, SS:localization}.  We can find a disc $D \subset \overline{\scrC}$  normal to the discriminant and a one-parameter family $\scrX_D \to D$ with fibre $A_2^{(n)} = X \mapsto 1$ and such that $\scrX_D \supset \Delta_a$ contains a Lagrangian thimble fibred over an arc $[0,1] \subset D$ and with boundary $\Delta_a \cap X = S_a$.  

Consider the following commutative diagram of vector bundles over $S_a$:
\begin{equation} \label{eq:triv-triv}
\xymatrix{
0 \ar[r] & \ar[d]^{\cong} TS_a \otimes_\bR \bC \ar[r] & \ar[d]^{\cong} ((T\Delta_a)|_{S_a}) \otimes_\bR \bC \ar[r] & S_a \times \bC \ar[r] \ar[d]^{=} & 0 \\
0 \ar[r] & TX|_{S_a} \ar[r] & T(\scrX_D)|_{S_a}\ar[r] & S_a \times \bC \ar[r] & 0
}
\end{equation}
Since $\Delta_a$ is contractible, its tangent bundle $T\Delta_a$ is trivial canonically up to homotopy, and that induces a stable trivialization of $TS_a$. One obtains a spectral brane structure on $S_a$ by extending the stable Gauss map from $S_a$ to the bounding thimble $\Delta_a$. The required compatibility, namely that the resulting trivialisation of $TS_a \otimes_{\bR} \bC = TX|_{S_a}$  is homotopic to that arising from the trivialisation of $TX$ fixed in \eqref{eqn:total_space_bundles}, follows from the commutativity of \eqref{eq:triv-triv}.
    \end{proof}

The previous argument shows any matching sphere defines a spectral brane, since the stable polarisation of $A_2^{(n)}$ is preserved by the natural action of the braid group.

\begin{lem}
    If $n\equiv \{0,4,6,7\}$ mod $8$ then any closed exact Lagrangian in $X$ with Maslov class zero admits the structure of a spectral Lagrangian brane, which is quasiisomorphic in the Fukaya category over $\bZ$ to an actual Lagrangian sphere.
\end{lem}

\begin{proof}
Under the Maslov class hypothesis, 
    $L$ is a homology sphere, which is quasi-isomorphic to an actual Lagrangian sphere \cite{AS:plumbings}. The dimension hypothesis shows $\pi_n(U/O) = 0$ so the stable Gauss map of $L$ (which depends only on its stable homotopy type) vanishes for degree reasons.
\end{proof}
Then this, combined with Theorem \ref{thm:main}, proves:
\begin{cor}
    Under the previous conditions on $n$, if $L \subset A_2^{(n)}$ is a closed exact Lagrangian with vanishing Maslov class, the class of $L$ in $\Omega_n^{fr}(\ast)$ lies in the image of $\eta$, in particular is 2-torsion.
\end{cor}

It follows that any exact Lagrangian in the 8-dimensional $A_2$-plumbing is diffeomorphic to $S^8$.

        \subsection{Exotic twists - non-existence}

Consider the plumbing $X = T^*S^n \#_{pt} T^*\Sigma$ of the cotangent bundles of the standard sphere $S^n$ and an exotic sphere $\Sigma$; this is also the Weinstein manifold obtained by attaching a handle to the boundary of one fibre in $T^*\Sigma$.  An application of our main result shows:

\begin{thm}
    If $n$ is even and $\Sigma$ does not have order 2 in $\Theta_n / bP_{n+1}$ then the algebraic twist $T_{\Sigma} \in \mathrm{Auteq}(\scrF(X;\bZ))$ is not realised by any symplectomorphism.
\end{thm}

This partially verifies a speculation of Seidel from \cite[Section 5]{Seidel:exotic}.  We give a proof by contradiction, so suppose $\tau$ is a symplectomorphism acting on the compact Fukaya category via the algebraic twist. For now $n$ can be arbitrary, but we will need to constrain it in the course of the argument. 

Let $L=S^n$ denote the core component which is the standard sphere.  By Lagrange surgery, there is an embedded exact Lagrangian $L\# \Sigma$ in the homotopy class of the sum of the core components.

        \begin{lem}
    In the integral Fukaya category $\scrF(X;\bZ)$, the objects $T_{\Sigma}(L)$ and $L \# \Sigma$ admit gradings with respect to which they are quasi-isomorphic. 
        \end{lem}

        \begin{proof} 
        There is a Lagrangian cobordism associated to the Polterovich surgery \cite{Polterovich}, giving an exact Lagrangian in $X\times \bC$ with ends carrying $L$, $\Sigma$ and $L\#\Sigma$. It follows from \cite{BC} (who worked in the monotone case but use only an index-action relation on holomorphic discs which holds trivially in the exact case) that there is an exact triangle realising $L\#\Sigma$ as a mapping cone $\mathrm{Cone}(L \to \Sigma)$ of some morphism. Since $L \pitchfork \Sigma = \{p\}$ is a single point, the possible morphisms are $\lambda p$ for some $\lambda \in \bZ$. The cone is irreducible over any field $\bF_p$ because it has endomorphisms $H^*(L\#\Sigma;\bF_p)$, which shows that $\lambda \in \{\pm 1\}$. Adjusting the grading (and orientation) of $L$ if necessary, the result follows.
        \end{proof}

    \begin{lem}
        $TX$ admits a trivialisation, with respect to which the two core components $S^n$ and $\Sigma$ have nullhomotopic stable Gauss map.
    \end{lem}
    \begin{proof}
        We choose trivialisations of $TS^n$ and $T\Sigma$, which, since $T(T^*S^n) \cong TS^n \otimes \bC$ (and similarly for $\Sigma$) induce trivialisations of the tangent bundles of the two cotangent bundles, with respect to which the zero sections have nullhomotopic stable Gauss map.\par 
        The two trivialisations don't agree on the nose over the plumbing region, but since the space of choices is contractible we may homotope them so that they do (without changing the fact that the zero-sections had nullhomotopic Gauss maps). We then can glue these trivialisations together over the plumbing region to get a trivialisation of $TX$.
    \end{proof}

        \begin{lem}
        If $n$ is even, the previously constructed polarisation is the unique stable polarisation of $X$, hence is preserved up to homotopy by $\tau$.
        \end{lem}

        \begin{proof}
            Two different stable polarisations differ by a map $X \to U$. Since $X$ has the homotopy type of a wedge of even-dimensional spheres and $\pi_{even}(U)=0$ the result follows. 
        \end{proof}

It follows that $\tau$ acts on the set of objects of the spectral Donaldson-Fukaya category. Since $\tau(L)$ and $L$ are diffeomorphic, they are certainly framed bordant, and since $L$ is the standard sphere it follows that $\tau(L)$ framed bounds. However, Theorem \ref{thm:main} implies that $\tau(L)$ and $L \# \Sigma \cong_{C^{\infty}} \Sigma$ are quasi-isomorphic in $\scrF(X;\bS)_{\tau\leq n, (n-1)-\mathrm{BS}}$ and hence have the same class in $\Omega_n^{fr} / \langle \eta \rangle$. It follows that $\Sigma \# \Sigma$ framed bounds, i.e. $\Sigma$ has order 2 in $\Theta_n/bP_{n+1}$.

\subsection{Exotic twists - existence}

Evans and Dmitroglou Rizell \cite{ER14} construct `exotic' symplectomorphisms from Dehn twists defined with respect to non-standard parametrisations of the sphere. Their work relies on non-existence theorems for exact Lagrangian embeddings of certain homotopy spheres into cotangent bundles of the standard sphere, and gives rise to exotic symplectomorphisms of Stein manifolds which arise as fibres of a Lefschetz fibration on $T^*S^k$ for appropriate $k$, in particular to the $A_n$ Milnor fibres in dimension $k-1$.  Theorem \ref{thm:main} allows one to extend their constructions to a wider and more natural class of affine varieties.

\begin{hyp}
    Let $X$ be a Weinstein manifold which contains an $A_2$-chain of Lagrangian spheres $L,K$ (so $L \pitchfork K = \{pt\}$). We assume that $X$ admits a polarisation with respect to which each of $L$ and $K$ have spectral brane structures.
\end{hyp}

\begin{ex}
    We say that a smooth affine algebraic variety $X$ admits an $A_2$-degeneration if there is a larger affine variety $\cX \to \bC^2$ with general fibre $X$, with $0$-fibre having an isolated $A_2$-singularity, and with discriminant locally near $(0,0)$ a cuspidal curve (with a nodal singularity over the smooth points). If $X$ is an affine complete intersection, a variation on the argument from Section \ref{Sec:first_application}, using that the vanishing cycles bound Lagrangian thimbles, shows that $X$ satisfies the given hypotheses. 
    \end{ex}

\begin{ex}
    The Milnor fibre of any isolated hypersurface singularity of Milnor number at least 2 admits an $A_2$-degeneration. The general smooth affine degree $d$ hypersurface in $\bC^n$ admits an $A_2$-degeneration for $n\geq 2$ and $d\geq 3$.  (In fact the locus of hypersurfaces with an $A_2$-singularity is irreducible, so the $A_2$-configuration is essentially unique \cite{Shimada}.)
\end{ex}

For a diffeomorphism $\phi: S^n \to S^n$ write 
\[
\tau_L^{\phi} = (d\phi) \circ \tau_L \circ (d\phi)^{-1}
\]
which is still compactly supported near $L$, and let 
\[
\psi = \tau_L^{-1} \tau_L^{\phi}.
\]
These are both compactly supported symplectomorphisms of $X$.
\begin{lem}
    $\psi$ preserves $K$ setwise but changes its parametrization by $\phi$.
\end{lem}

\begin{proof}
    See \cite[Lemma 3.5]{ER14}.
\end{proof}

Let $\pi: \cX \to \bC$ denote the Stein manifold which is the total space of a Lefschetz fibration with fibre $X$, with two critical fibres, both of which have vanishing cycle $K$.  This contains a copy of $S^{n+1}$ as a matching sphere.

\begin{lem}
    If $\psi \simeq \id$ in $\Symp_{ct}(X)$ then $\cX$ contains a Lagrangian submanifold $\Sigma(\phi)$ which is the homotopy $(n+1)$-sphere defined by $\phi$. 
\end{lem}

\begin{proof}
    See \cite[Section 3.3 (proof of theorem A)]{ER14}; the argument uses the suspension of a Hamiltonian isotopy.
\end{proof}

\begin{lem}
    In the situation of the previous Lemma, if $\psi \simeq \id$ then the $\Sigma(\phi)$ is quasi-isomorphic to the matching $S^{n+1}$ in the Fukaya category $\cF(\cX;\bZ)$.
\end{lem}

\begin{proof}
    From one of the main results of \cite{Seidel:book}, the compact Fukaya category embeds in the Fukaya category of the Lefschetz fibration $\cF\cS(\pi)$ so it is sufficient to prove quasi-isomorphism there.  The category $\cF\cS(\pi)$ is generated by the thimbles (for fields of characteristic not equal to two this is proved in \cite{Seidel:book}, and over $\bZ$ it is proved in \cite{GPS}). We can assume that $S^{n+1}$ is fibred over the upper half-circle and $\Sigma(\phi)$ is fibred over a fattened arc which agrees with the lower half-circle near its end-points and is disjiont from $S^{n+1}$ except at the end-points. This shows that both Lagrangians have $CF$ of rank 1 with each of the two Lefschetz thimbles, so both are cones on a morphism $\Delta_1 \to \Delta_2$. But this is a morphism from $K$ to $K$ in the Fukaya category of the fibre $X$, i.e. an element of $H^*(S^n;\bZ)$. The morphism must be primitive because the cone is a homology sphere, and there is a unique morphism of the correct degree up to sign.  Both Lagrangians admit canonical gradings which pins down the orientation and the sign.
\end{proof}

The space $\cX$ is obtained from $X \times \bC$ by adding two $(n+1)$-handles along spheres with trivial Gauss map. If $\pi_{n+1}(U/O) = 0$, which happens when $n \in \{3,5,6,7\}$ modulo $8$, the stable polarisation extends, and the spheres $S^{n+1}$ and $\Sigma(\phi)$ are also automatically then polarised branes. The main theorem then implies that they define the same class in $\Omega_{n+1}^{fr}(\cX)/(\Omega^{fr}_1(\cX) \cdot \Omega^{fr}_n)$.

\begin{cor}
    Under the previous conditions, if $\phi$ does not have order 2, then $\psi$ is not symplectically trivial.
\end{cor}

\begin{cor} \label{cor:general_affine}
    Let $X$ be a generic affine hypersurface of degree $d\geq 3$ and of dimension $n \in \{3,5,6,7\}$ mod $8$. Assume that there is an element $\phi \in \Theta_n/bP_{n+1}$ which is not 2-torsion. Then there is a compactly supported graded symplectomorphism which acts trivially on the integral Fukaya category.
\end{cor}

\begin{rmk}
    For the $A_k$-Milnor fibre of complex dimension $n\geq 2$, it is known that the composite map \[
    \pi_1\mathrm{Conf}_{k+1}(\bC) = Br_k\longrightarrow \pi_0\Symp_{ct}(A_k^{(n)}) \longrightarrow \mathrm{Auteq}(\scrF(A_k^{(n)};\bZ)) 
    \]
    is injective (working with the $\bZ$-graded integral Fukaya category), and one infers from the existence of `exotically framed' Dehn twists  that in certain dimensions $n$ the classical monodromy map
    \[
    \pi_1\mathrm{Conf}_{k+1}(\bC) \longrightarrow \pi_0\Symp_{ct}(A_k^{(n)})
    \]
    is not surjective. We similarly conjecture that, in the setting of Corollary \ref{cor:general_affine}, the monodromy homomorphism is not onto the symplectic mapping class group. Note that the conditions on the degree and dimension suggest that this is a `typical' phenomenon (happens with positive probability relative to natural density on the set of pairs $(d,n)$). 
\end{rmk}

\subsection{Homotopy classes of Lagrangian spheres}

Quasi-isomorphic Lagrangian spheres are homologous. Are they homotopic? In general little is known. 

Add a (subcritical!) $k$-handle to $T^*S^n$ where $1<k<n$. The resulting Stein manifold $X$ has homotopy type $S^n \vee S^k$.  Assume that $k$ and $n$ both have modulo 8 reduction in $\{0,4,6,7\}$. Then the hypothesis on $k$ implies that the stable polarisation of $T^*S^n$ extends over the $k$-handle; the hypothesis on $n$ implies that an exact Lagrangian sphere $L \subset X$ also admits a stable polarisation.  

\begin{lem}
    $L$ is quasi-isomorphic to the zero-section over $\bZ$.
\end{lem}

\begin{proof}
    $L$ is a sphere so has vanishing Maslov class, hence defines a $\bZ$-graded proper module over the wrapped category $\scrW(X)$. Since subcritical handle addition doesn't change the wrapped category \cite{GPS},  the wrapped category of the cotangent bundle is generated by the fibre, and $HW(T_x^*, T_x^*) = \bZ[x]$ with $|x| = 1-n$, we have that 
    \[
\scrW(X) \simeq \scrW(T^*S^n) \hookrightarrow \mathrm{mod}_{\bZ}(\bZ[x])
    \]
    A compact exact Lagrangian defines a proper module; the ascending degree filtration is a filtration by $A_{\infty}$-submodules, from which it follows that a Lagrangian sphere $L$ in $X$ is a twisted complex on the (image under $T^*S^n \subset X$ of the) zero-section. The result then follows by (an easier version of) the homological algebra of \cite[Section 2]{AS:plumbings}, see also \cite{Abouzaid-Diogo}.
\end{proof}

Theorem \ref{thm:main} then implies that the classes of $L$ and the zero-section agree in $\Omega_n^{fr}(X)/(\Omega_1^{fr}\cdot \Omega_{n-1}^{fr})$. 

Comparing the $(n+1)$-skeleton of $X$ and the $(n+1)$-skeleton of $S^k \times S^n$ shows that $\pi_n(X) = \bZ \oplus \pi_n(S^k)$. Projecting the fundamental class of $L$ in $\Omega^{fr}_n(X) = \pi^{st}_n(X)$ to the final factor gives a class in $\pi^{st}_{n-k}$; this is the same as the (stabilisation of the) projection of the homotopy class of $L$ to a class in $\pi_n(S^k)$. The previous conclusion then constrains the class of $[L] \in \pi_n(X)$ in the image of the second factor in $\pi_{n-k}^{st}$:  

\begin{example}
    Adding a 4-handle to $T^*S^7$ yields a manifold $X$ for which any Lagrangian sphere $L$ has a class in 
\[
\bZ \oplus \bZ/12 = \pi_7(S^4) \to \pi_3^{st} = \bZ/24
\]

and the image of the class on the RHS must be two-torsion.
\end{example}

\section{Manifolds with faces}

There is an extensive literature on manifolds with corners, with a number of different conventions and definitions, surveyed in \cite[Remark 2.11]{Joyce}, cf. also \cite{Joyce2}. We will largely deal with the special class typically called `manifolds with faces'. 

    \subsection{Corners}
        \begin{defn}
            A \emph{manifold with corners} $M$ is a second-countable Hausdorff topological space equipped with an atlas of charts giving local homeomorphisms to open subsets of $(\bR_{\geq 0})^n$ such that the transition functions are all smooth.\par 
            For a point $p$ in $M$, write $\Gamma(p)=\Gamma^M(p)$ for its \emph{codimension}, meaning the codimension of the corner strata it lives in. This defines an upper semicontinuous function $\Gamma: M \rightarrow \bR$. We write $M^\circ$ for the interior of $M$, i.e. $M \setminus \partial M$. Equivalently this is the set of codimension 0 points in $M$.\par 
            A \emph{manifold with faces} is a manifold with corners $M$ such that the closure of any connected component of the set of points $p$ in $M$ of codimension $i$ (for any $i$) is an embedded submanifold with corners.\par 
            A \emph{face} of codimension $i$ is a union of disjoint closures of connected components of the set of points $p$ in $M$ of codimension $i$. A \emph{boundary face} will refer to a face of codimension 1. \par 
            Often we will say $F$ is a face of $M$ but $F$ will not be defined to be a subset of $M$, in this case this means that we equip $F$ with an embedding into $M$, so that its image is a face of $M$.
        \end{defn}
        \begin{defn}
            An embedding of manifolds with faces $f: M \rightarrow N$ of codimension $i$ \emph{strictly respects corner strata} if $\Gamma^M(p) = \Gamma^N(p) + i$ for all $p$ in $N$.
        \end{defn}
        \begin{ex}
            The inclusion of any face $F$ into $M$ strictly respects corner strata.
        \end{ex}
        \begin{rmk}
            Let $M^n$ be a manifold with corners, and $p$ a point in $M$. Then there is a chart near $p$ $\phi: U \subseteq M \rightarrow V \subseteq \bR^{n-i} \times \bR_{\geq 0}^{i}$ sending $p$ to $0$, where $i = \Gamma(p)$ is the codimension of $p$ in $M$.
        \end{rmk}
        The following characterisation is often taken as the definition of a manifold with faces, see e.g. \cite{Janich}.
        \begin{lem}
            Let $M$ be a manifold with corners. Then $M$ is a manifold with faces if and only if for each point $p$ in $M$, $p$ lies in the closure of exactly $\Gamma(p)$ boundary faces.
        \end{lem}
        \begin{proof}
            Note that $p$ necessarily belongs to the closure of at most $\Gamma(p)$ boundary faces, since it belongs to exactly $\Gamma(p)$ components of the union of the boundary faces intersected with a sufficiently small local chart $p \in U \subset (\bR_{\geq 0})^n$. If it belongs to strictly fewer than $\Gamma(p)$ boundary faces then there must be some face whose closure is not embedded.
        \end{proof}
        \begin{ex}
            The $n$-cube $[0, 1]^n$ is a manifold with faces for all $n \geq 0$.\par 
            The teardrop (two-dimensional disc with one corner) is a manifold with corners but not a manifold with faces.
        \end{ex}

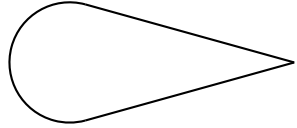
\begin{figure}[ht]
\begin{center}
    \begin{tikzpicture}[scale=0.4]
      \pgfmathsetmacro{\b}{75}
      \pgfmathsetmacro{\R}{2}
      \pgfmathsetmacro{\P}{\R*tan(\b)}
   \begin{scope}[thick]
        \draw (\b:\R) arc (\b:360-\b:\R) ;
        \draw (\b:\R) -- ( \P, 0 ); 
        \draw (-\b:\R) -- ( \P, 0 );
      \end{scope}
\end{tikzpicture}
\caption{The teardrop}
\end{center}
\end{figure}

        \begin{lem}
            Let $M$ be a manifold with faces. Then all of its faces are themselves manifolds with faces.
        \end{lem}
        \begin{proof}
            We first argue that any face $F$ of codimension $i$ of $M$ is a manifold with corners.\par 
            Pick a point $p$ in $F$. Pick a chart $\phi: U \subseteq M \rightarrow V \subseteq \bR^n_{\geq 0}$, such that $U$ and all its faces are connected, and $U$ contains $p$. Then after reordering the coordinates $\phi(F)$ must be of the form $U \cap \{0\}^i \times \bR^{n-i}_{\geq 0}$. Identifying $\{0\}^i \times \bR^{n-i}_{\geq 0}$ with $\bR^{n-i}_{\geq 0}$ gives a chart on $F$ as a manifold with corners. Allowing $p$ to vary gives an atlas for $F$ as a manifold with corners.\par 
            Let $G^\circ$ be a component of the set of codimension $j$ points in $F$. Then it is in particular a component of the set of codimension $i + j$ points in $M$, so its closure is embedded in $M$ and so also embedded in $F$.
        \end{proof}
        \begin{lem}
            Let $M$ be a manifold with faces. If $F$ is a face of $G$ and $G$ is a face of $M$, then $F$ is a face of $M$.
        \end{lem}
        \begin{proof}
            Given a codimension $i$ face $F$ of $M$, the set of codimension $j$ points in $F$ is a union of components of the set of codimension $i + j$ points in $M$, so the closure of any connected component is embedded in $M$ and therefore also is embedded in $F$.
        \end{proof}
        From this proof, we see that if $F$ is a codimension $i$ face of $M$, then the set of codimension $j$ faces in $F$ is the same as the set of codimension $i + j$ faces in $M$ which touch $F$.
        \begin{lem}
            Let $M^n$ be a manifold with faces, and let $F$ and $G$ be faces of $M$. Then their intersection is a union of faces of $M$.
        \end{lem}
        \begin{proof}
            We first show $F \cap G$ is a manifold with corners. Let $F$ be a codimension  $i$ face and $G$ a codimension $j$ face with $i\leq j$. To prove the lemma it suffices to treat the case in which $F$ and $G$ are connected.  If $G \subset F$ the result is obvious, so suppose $G \not\subset F$, and let $p$ be a point of $F \cap G$.  Then $p$ is not in the interior of $G$, so we can replace $G$ by a component of $\partial G$ which replaces $(i,j)$ by $(i, j+1)$.  The proof is completed by downwards induction on $j-i$, where the maximal base case is where $G$ is a point and $F$ is (a connected component of) the whole space, and the result is clear.
            
            More explicitly, pick a point $p$ in $F \cap G$, and let $i = \Gamma^M(p)$. Pick a chart $\phi: U \subseteq M \rightarrow V \subseteq \bR^{n-i} \times \bR^i_{\geq 0}$ near $p$, such that $U$ and all its faces are connected, and $\phi(p) = 0$. Then after reordering the co-ordinates, we have that 
            $$\phi(F \cap U) = \{(x, y) \in \bR^{n-i} \times \bR_{\geq 0}^i\,|\, y_1 = \ldots = y_a = 0$$
            and 
            $$\phi(G \cap U) = \left\{(x, y) \in \bR^{n-i} \times \bR_{\geq 0}^i\,|\, y_{a-c + 1} = \ldots = y_{a-c+b}\right\} = 0$$
            where $a$ is the codimension of $F$, $b$ is the codimension of $G$ and $c$ is some other number. Therefore
            $$\phi(F \cap G \cap U) = \left\{ (x, y) \in \bR^{n-i} \times \bR_{\geq 0}^i\,|\, y_{a-c+1} = \ldots = y_a = 0\right\}.$$
            This provides a chart for $F \cap G$ as a manifold with corners of dimension $c$. Here $p$ has codimension $i-c$ in $F \cap G$. We see further from this local model that the closure of the smooth locus of $F \cap G$ is an embedded submanifold with corners in $M$ and therefore also in $F \cap G$.
        \end{proof}

        \begin{defn}
            A \emph{system of boundary faces} for a manifold with corners $M$ is a (finite) collection of embedded closed (as subsets) manifolds with corners $F_i \subseteq M$, lying in $\partial M$, all of codimension 1, such that each point $p$ in $M$ of codimension 1 lies in exactly one of the $F_i$.\par 
            A \emph{system of faces} for a manifold with faces $M$ is a collection of faces $F_i^j$ of $M$ such that each $F_i^j$ has codimension $j$, and all $F_i^j$ have disjoint smooth loci, and any (positive codimension) face of $M$ is a disjoint union of components of the $F_i^j$. Equivalently, any point $p$ in $M$ of codimension $j > 0$ lives in the interior of exactly one of the $F_i^j$.\par 
            We will often say that $\{F_i^j\}_{i,j}$ is a system of faces for $M$ without the $F_i^j$ being submanifolds of $M$; by this we mean that they are also equipped with embeddings $F_i^j \hookrightarrow M$ making their images into a system of faces. Often some of the $F_i^j$ will naturally be faces of each other, in which case the embeddings $F_i^j \hookrightarrow M$ are assumed to be compatible with this structure. 
        \end{defn}
        \begin{lem}
            Let $M$ be a manifold with corners, and $\{F_i\}_{i=1,\ldots, k}$ a system of boundary faces, which are all assumed to be manifolds with faces. Then $M$ is a manifold with faces, with a system of faces given by the intersections of the $F_i$, i.e. the 
            $$F_I := \bigcap\limits_{i\in I} F_i$$
        \end{lem}
        \begin{proof}
            Each $F_I$ is a closed (as a subset) submanifold; it lies in some $F_i$ which is a manfold with faces so $F_I$ must be embedded in $M$. Therefore it suffices to show that each $p \in \partial M$ lies in the interior of exactly one $F_I$.\par 
            Let $l = \Gamma(p)$. Then locally, $p$ lies in the intersection of $l$ of the boundary faces, say $F_{i_1}, \ldots, F_{i_l}$. Letting $I = \{i_1, \ldots, i_l\}$, $p$ lies in the interior of $F_I$, and this is uniquely determined. 
        \end{proof}
        \begin{lem}
            Let $M$ and $N$ be manifolds with faces. Then $M \times N$ is a manifold with faces, and its connected faces are exactly products of connected faces of $M$ and $N$.
        \end{lem}
        \begin{proof}
            A connected component of the set of points $p$ in $M \times N$ of codimension $i$ must be a product of such things in $M$ and $N$ (of codimensions $j$ and $k$, such that $j + k = i$) respectively. Its closure is the product of the closures of the faces in $M$ and $N$ and since these are embedded this must therefore also be embedded. 
        \end{proof}
    
    \subsection{Collars}
        Let $M$ be a manifold with faces of dimension $n$. We write $D(M)$ for the set of codimension 1 connected faces of $M$, and $D({F \subseteq M})$ for the set of codimension 1 (connected) faces of $M$ which touch $F$. This is a set with $\#D({F \subseteq M})=\Gamma(F)$.\par 
        If $F \subseteq M$ is a boundary face, there is a natural map $D(F) \to D(M)\setminus \{F\}$, sending $F \cap G$ to $G$ for $G \in D(M)$.\par 
        In practice, when we list systems of faces, we will allow these sets to contain empty faces, and can therefore assume the map $D(F) \to D(M)\setminus \{F\}$ is a bijection, for the sake of combinatorial sanity.
        \begin{defn}
            Let $M$ be a manifold with faces. A \emph{system of normals} $\nu$ on $M$ consists of non-zero inwards pointing sections $\nu^F$ in $\gamma(TM|_F)$ for $F$ in $D(M)$, such that for all $G$ in $D({F \subseteq M})$, $\nu^F|_{F \cap G}$ lies in the subspace $\gamma(TG|_{F \cap G})$ of $\gamma(TM|_{F \cap G})$.
        \end{defn}
        Since the space of systems of normals on a fixed $M$ is convex, we have that
        \begin{lem}
            The space of systems of normals on $M$ is non-empty and convex, and hence contractible.
        \end{lem}
        \begin{lem}\label{products and normals}
            Let $M$ and $N$ be manifolds with faces, equipped with systems of normals $\nu$ and $\nu'$ respectively. Then there is a natural system of normals on $M \times N$.
        \end{lem}
        \begin{proof}
            Boundary one faces of $M \times N$ are of the form $F \times N$ or $M \times G$, for $F$ and $G$ boundary faces of $M$ and $N$ respectively. To each face $F \times N$ we assign the normal vector field $(\nu^F, 0)$ in $\Gamma(TM \oplus TN|_{F \times N})$, and similarly for the case $M \times G$.
        \end{proof}
        \begin{defn}
            Let $M$ be a manifold with faces. A \emph{system of collars} $\cC$ on $M$ consists of collar neighbourhoods
            $$\cC=\cC^{FG}: F \times [0, \varepsilon)^{D({F \subseteq G})} \hookrightarrow G$$
            meaning an embedding onto a neighbourhood of $F \subseteq G$, whenever $F$ and $G$ are faces of $M$ of codimension $i$ and $j$ respectively, such that $F \subseteq G$. We require that the following diagram commutes.
            \[ \xymatrix{
                F \times [0, \eps)^{D({F\subseteq M})} \ar[r]_\cC \ar[d]_= &
                M \ar[dd]_=\\
                F \times [0, \eps)^{D({F \subseteq G})} \times [0, \eps)^{D({G \subseteq M})} \ar[d]_\cC & \\
                G \times [0, \eps)^{D({G\subseteq M})} \ar[r]_\cC &
                M
            }
            \]
            There is an equivalent inductive definition. We could define a \emph{system of collars} on $M$ to consist of systems of collars on all positive codimension faces in $M$, which restrict to each other on subfaces, along with maps 
            $$\cC: F \times [0, \varepsilon) \hookrightarrow M$$ 
            for all boundary faces $F$ such that for all distinct pairs of boundary faces $F_i, F_j$, the following diagram commutes,
            \begin{equation}\label{star}
            \xymatrix{
                F_i \cap F_j \times [0, \varepsilon)^2 \ar[rrr]^{\cC^{F_i} \times \id_{[0, \varepsilon)}} \ar[d]_{\left(\cC^{F_j} \times \id_{[0, \varepsilon)}\right) \circ \tau}  &&& F_i \times [0, \varepsilon) \ar[d]^{\cC} \\
                F_j \times[0, \varepsilon) \ar[rrr]_{\cC} &&& M
            }
            \end{equation}
            where $\tau$ swaps the two factors in $[0, \varepsilon)^2$.
        \end{defn}
        \begin{defn}
            By taking the derivative of $\cC^F$ in the inwards normal direction for each boundary face $F$, a system of collars on $M$ determines a system of normals on $M$, $\nu(\cC)$. We say a system of collars $\cC$ \emph{extends} a system of normals $\nu$ if $\nu(\cC) = \nu$. If this holds, we can extend $\nu^F$ to a section of $TM$ over the whole of $\im \cC^F$ by taking the derivative of $\cC^F$ in the inwards normal direction.
        \end{defn}
        \begin{lem}\label{products and collars}
            Let $M$ and $N$ be manifolds with faces, equipped with systems of collars $\cC$ and $\cD$ respectively. Then there is a natural system of collars on $M \times N$.
        \end{lem}
        \begin{proof}
            We define a system of corners $\cE$ on $M \times N$. Let $F \subseteq F'$ be faces of $M$ of codimension $i$ and $i'$ respectively, and $G \subseteq G'$ be faces of $N$ of codimensions $j$ and $j'$ respectively. Note that $F \times G$ is a face in $F' \times G'$ of codimension $i-i'+j-j'$. Then we define
            $$\cE: F \times G \times [0, \varepsilon)^{i-i' + j-j'} \hookrightarrow F' \times G'$$
            to be $\cC_{FF'} \times \cD_{GG'}$.
        \end{proof}
        We will use extensively the following version of the collar neighbourhood theorem for manifolds with faces.
        \begin{lem}\label{systems of collars exist}
            Let $M$ be a manifold with faces, with a system of boundary faces $F_1, \ldots, F_k$. Assume all the boundary faces $F_i$ are equipped with systems of faces which agree on their overlaps. Then there exists a system of collars on $M$ extending those on the $F_i$. Furthermore this is unique up to isotopy of such systems of collars, and in particular is unique up to an ambient diffeomorphism of $M$ which is the identity on $\partial M$.\par 
            For a fixed system of normals $\nu$, existence and uniqueness still holds if we require that the systems of collars (and isotopy) all extend $\nu$.
        \end{lem}
        \begin{proof}
            We use the second definition of a system of collars, and proceed inductively. We first prove existence.\par 
            We build the collar neighbourhoods we need by induction. We first assume we have collar neighbourhoods $\cC_i: F_i \times [0, \varepsilon) \hookrightarrow M$ such that Square \ref{star} commutes for $i, j \leq l-1$. We will build a collar neighbourhood $\cC_{l}: F_l \times [0, \varepsilon) \hookrightarrow M$ so that Square \ref{star} commutes for $i, j \leq l$. Then by induction on $l$ this will complete the proof.\par
            We first choose any collar neighbourhood $\cC'_l: F_l \times [0, \varepsilon) \hookrightarrow M$. We induct again, and assume Square \ref{star} commutes for $i \leq m-1, j=l$, and we will isotope $\cC'_l$ without breaking this, so that Square \ref{star} commutes for $i = m, j=l$ too. Then by induction on $m$ this will complete the proof.\par 
            We consider Square \ref{star} for $i=m, j=l$. This doesn't necessarily commute, but by uniqueness of collar neighbourhoods it does up to some isotopy $\Phi$, which we can extend to an isotopy of the whole collar $\cC'_l$. However it does commute in a neighbourhood of $F_i \cap F_m \cap F_l$ since we know that Square \ref{star} commutes, for all $i \leq m-1$. Therefore we can choose the isotopy $\Phi$ to be constant in a neighbourhood of all $F_i \cap F_m \cap F_l$ for all $i \leq m-1$. Then the other end of the isotopy provides a collar $\cC_l$ such that Square \ref{star} commutes for $i \leq m, j \leq l$.\par 
            We next prove uniqueness. Let $\cC$ and $\cD$ be two systems of collars on $M$ extending the given ones on the $F_i$. We wish to build smooth 1-parameter families of embeddings 
            $$\{\cI_i^t\}_{t \in [0, 1]}: F_i \times [0, \varepsilon) \hookrightarrow M$$
            onto a neighbourhood of $F_i$ for all $i$, such that $\cI^0_i = \cC_i$, $\cI^1_i = \cD_i$ and such that the following analogue of Square \ref{star} commutes for all $t$, $i$:
            \begin{equation}\label{star2}
            \xymatrix{
                F_i \cap F_j \times [0, \varepsilon)^2 \ar[rrr]^{\cC^{F_i} \times \id_{[0, \varepsilon)}} \ar[d]_{\left(\cC^{F_j} \times \id_{[0, \varepsilon)}\right) \circ \tau}  &&& F_i \times [0, \varepsilon) \ar[d]^{\cI_i^t} \\
                F_j \times[0, \varepsilon) \ar[rrr]_{\cI_j^t} &&& M
            }
            \end{equation}
            We proceed inductively as before. Assume we've build $\cI_i^t$ such that Square \ref{star2} commutes for $i, j \leq l-1$. We then build an isotopy $\cI_i^t$ so that Square \ref{star2} commutes for $i, j \leq l$, then conclude by induction on $l$.\par 
            We first choose any isotopy $\cI'^t_l: F_l \times [0, \varepsilon) \hookrightarrow M$ between $\cC_l$ and $\cD_l$, using the uniqueness of collar neighbourhoods. We induct again, and assume Square \ref{star2} commutes for $i \leq m-1$, $j = l$, and we will isotope $\cI'^t_l$ (through isotopies from $\cC_i$ to $\cD_i$) without breaking this, so that Square \ref{star2} commutes for $i=m$, $j=l$ too. Then conclude by induction on $m$.\par 
            We consider Square \ref{star2} for $i=m$, $j=l$. This commutes for $t = 0$ and $t=1$ but not arbitrary $t$, but by uniqueness of collar neighbourhoods in families it does up to some $t$-dependent isotopy $\Phi^t$, which is the identity for $t=0$ and $t=1$. We extend this to an isotopy of the whole isotopy $\cI_l^t$. However since Square \ref{star2} commuted in a neighbourhood of all $F_i \cap F_m \cap F_l$ for $i \leq m-1$, we can choose $\Phi^t$ to be constant in a neighbourhood of all $F_i \cap F_m \cap F_l$ for $i \leq m-1$. Then the other end of the isotopy $\Phi$ provides the isotopy $\cI_l^t$ such that Square \ref{star2} commutes for $i \leq m$, $j \leq l$.\par 
            The entirety of the above proof can be performed the same with a fixed system of normals $\nu$ if we require that all collar neighbourhoods (and isotopies) extend $\nu$.
        \end{proof}
        \begin{cor}
            Let $M$ be a manifold with corners, and assume some of its faces are equipped with systems of collars. Then $M$ admits a system of collars extending them, and this is unique up to isotopy of such things, and in particular is unique up to diffeomorphisms of $M$ which are the identity on the faces which were already equipped with systems of collars.
        \end{cor}
    
    \subsection{Abstract gluing}
      
            Let $M$ be a manifold with faces, equipped with a system of collars. We can describe a neighbourhood of $\partial M$, as a manifold with corners, in terms of the corner strata of $M$. Let $F_1, \ldots, F_k$ be a system of boundary faces for $M$. Define
            $$\tilde{M} := \left(\bigsqcup\limits_{i} F_i \times [0, \varepsilon)\right)/\sim$$
            where we define $\sim$ by $(\cC^{F_i \cap F_j, F_i}(x, s), t) \sim (\cC^{F_i \cap F_j, F_j}(x, t), s)$ for $x$ in $F_i \cap F_j$ ($j \neq i$ here), and $s, t$ in $[0, \varepsilon)$. This equivalence relation is defined so that as a topological space, $M$ is the coequaliser of the two maps defined above
            $$\bigsqcup\limits_{i < j} F_i \cap F_j \times [0, \varepsilon)^2 \rightrightarrows \bigsqcup\limits_i F_i \times [0, \varepsilon).$$
            Then the maps $\cC$ define an embedding $\tilde{M} \hookrightarrow M$ onto a neighbourhood of the boundary.\par 
            It follows from the fact that $\tilde{M}$ is an open subset in $M$ that $\tilde{M}$ is a manifold with faces, but as we see in Lemma \ref{lem:abstract gluing}, we can prove this without reference to the ambient manifold $M$.

        In Section \ref{sec:flow cat}, there will be situations where we would like (but don't yet have) $M$, but we do have the data of its boundary faces along with the data of how they overlap; we then form $\tilde{M}$ using Lemma \ref{lem:abstract gluing} and then use this as a first step towards bulding $M$.
        \begin{lem}\label{lem:abstract gluing}
            Let $\{M_i\}_{i\in I}$ be a finite collection of manifolds with faces.\par 
            Let $\{N_w\}_{w \in W}$ be a finite collection of manifolds with faces, equipped with inclusion maps
            \[ \xymatrix{
                N_w \ar[r]^{\alpha_w}\ar[d]^{\beta_w} &
                M_{i_w} \\
                M_{j_w} &
            }
            \]
            each of which is a diffeomorphism onto a boundary face, and such that no boundary face is the image of two of these maps.\par 
            Define
            $$Z := \left(\bigsqcup\limits_i M_i \times [0, \eps)\right)/\sim$$
            where $\sim$ is the equivalence relation defined by
            $$(\cC(\alpha_w(x), u), v) \sim (\cC(\beta_w(x), v), u)$$
            for $x$ in any $N_w$ and $(u,v)$ in $[0, \eps)^2_{u,v}$.\par 
            Equivalently, $Z$ is the coequaliser of the following diagram.
            $$\bigsqcup\limits_w N_w \times [0, \eps)^2_{u,v} \rightrightarrows \bigsqcup_i M_i \times [0, \eps)_s$$
            where the two maps $A$ and $B$ send $(x, u,v)$ to $(\cC(\alpha_w(x), u), v)$ and $(\cC(\beta_w(x), v), u)$ respectively.\par 
            Then $Z$ is a manifold with corners.
        \end{lem}
        In the cases where we apply this, $Z$ will generally be a manifold with faces, but we will verify this separately (and describe a system of faces) in each case.
        \par 
        For a manifold $Z$ constructed as above, we write $C(Z)$ for the subspace $\cup_i M_i \times \{0\}$. This is itself not a manifold, but is a homotopy-equivalent subspace of $Z$. We call this the \emph{carapace} of $Z$.
        \begin{proof}
            First note that since no boundary face of any $M_i$ is the image of more than one of the $\alpha_w$ or $\beta_w$, the quotient map 
            $$\bigsqcup\limits_i M_i \times [0, \eps) \to Z$$
            has fibres which are all discrete and are either one or two points, and the restriction to each $M_i \times [0, \eps)$ is an embedding of topological spaces.\par 
            We first show that $Z$ is Hausdorff. Indeed suppose $\{p_n\}_n$ is a sequence which has two limits $(x, s)$ and $(x', s')$ in $Z$. Here $s$ and $s'$ lie in $[0, \eps)$, $x$ lies in $M_i$ and $x'$ lies in $M_{i'}$.\par 
            If $i=i'$, i.e. $x$ and $x'$ both lie in $M_i \times [0, \eps)$, then the sequence $p_n$ would also eventually lie in the image of $M_i \times [0, \eps)$ in $Z$, but because $M_i \times [0, \eps)$ is itself Hausdorff, and no two points in $M_i \times [0, \eps)$ are identified with each other under $\sim$, the image of $M_i \times [0, \eps)$ in $Z$ is Hausdorff.\par 
            We now assume $i \neq i'$. Then for sufficiently large $n$, $p_n$ must lie in both $M_i \times [0, \eps)$ and $M_{i'} \times [0, \eps)$ (since they're open neighbourhoods of $x$ and $x'$ respectively) (and can live in no other $M_j$ with $j \neq i,i'$), and so must be of the form $(a_n, s_n)$ for $a_n$ in $M_i$ and $s_n$ in $[0, \eps)$, but $p_n$ must also be of the form $(b_n, t_n)$ for $b_n$ in $M_{i'}$ and $t_n$ in $[0, \eps)$. \par 
            Since there are only finitely many $N_w$, restricting to a subsequence if necessary, $p_n$ must therefore be of the form $(c_n, u_n, v_n)$, where $c_n$ lies in some $N_w$, where $i_w = i$ and $j_w=i'$. $x$ and $x'$ can't lie in the image of $N_w \times [0, \eps)^2$ in $Z$ as otherwise they would both lie in $M_i$. For these to converge to points in $M_i \times [0, \eps)$ and $M_{i'} \times [0, \eps)$ not in $N_w \times [0, \eps)^2$, we must have that $u_n$ and $v_n$ both converge to $\eps$. But such a sequence has no limit in $Z$.\par 
            Then since $A$ and $B$ are smooth embeddings of manifolds with faces which strictly respect corner strata, since it is Hausdorff, $Z$ will be a manifold with corners. 
        \end{proof}
        We will need conditions under which this $Z$ is a manifold with faces, and a description of a system of faces in such a case. We work in the setting of Lemma \ref{lem:abstract gluing}. 
        \begin{lem}\label{lem:abstract gluing with faces}
            Assume we are given a finite set $J$ and that $W = \{(i,i') \in I\times I \,|\, i \neq i'\}$, $i_{(i,i')} = i$ and $j_{(i,i')} = i'$. \par 
            Assume we are also given bijections $\lambda_i:I\setminus\{i\} \sqcup J \to D(M_i)$ for all $i$, and bijections $\lambda_{ii'}: I\setminus\{i,i'\} \sqcup J \to D(N_{(i,i')})$ such that the following diagrams commute.
            \[\xymatrix{
                I\setminus\{i,i'\}\sqcup J \ar[r] \ar[d]_{\lambda_{ii'}} &
                I\setminus\{i\}\sqcup J\ar[d]_{\lambda_i} \\
                D(N_{(i,i')}) \ar[r] &
                D(M_i)
            }
            \]
            and
            \[\xymatrix{
                I\setminus\{i,i'\}\sqcup J \ar[r] \ar[d]_{\lambda_{ii'}} &
                I\setminus\{i'\}\sqcup J\ar[d]_{\lambda_{i'}} \\
                D(N_{(i,i')}) \ar[r] &
                D(M_{i'})
            }
            \]
            where the top horizontal maps are the natural inclusions, and the bottom horizonal maps are those induced by the inclusions $\alpha_{(i,i')}$ and $\beta_{(i,i')}$ respectively. \par 
            For $K \subseteq J$, write $M_i^K,N^K_{(i,i')}$ for the faces in $M_i$ and $N_{(i,i')}$ respectively, which are the intersections of the boundary faces corresponding to the subsets $\lambda_i(K)$ and $\lambda_{ii'}(K)$ of $D(M_i)$ and $D(N_{(i,i')})$ respectively. Define $Z^K$ to be the coequaliser of the following maps
            $$\bigsqcup\limits_{(i,i')\in W} N_{(i,i')}^K \times [0, \eps)^2 \rightrightarrows \bigsqcup_i M_i^K \times [0, \eps)$$
            where the two maps are defined by the same formulas as in Lemma \ref{lem:abstract gluing}.

            Then each $Z^K$ is a manifold with corners by Lemma \ref{lem:abstract gluing}, and in fact is a manifold with faces, with a system of faces given by
            $$Z^{K'}$$
            for $K \subseteq K' \subseteq J$, and
            $$M_{I'}^{K'} \times \{0\}$$
            for for $K \subseteq K' \subseteq J$ and $I' \subseteq I$, where $M_{I'}\times \{0\}$ is defined to be the intersection of all $M_i \times\{0\}$ for $i \in I'$.\par 
            In particular, $Z^K$ has a system of boundary faces given by
            $$Z^{\{j\}\cup K}$$
            for $i \in I$ and $j \in J\setminus K$, and
            $$M_i^K \times \{0\}$$
        \end{lem}
        Note that in particular, the faces of $Z^K$ of the form $Z^{K'}$ are non-compact, whereas those of the form $M^{K'}_{I'}\times \{0\}$ are compact.
        \begin{proof}
            It suffices to prove this in the case $K=\emptyset$, and since each of the list of faces given is indeed the intersection of some subset of the given list of boundary faces, it suffices to show the given list of boundary faces is indeed a system of boundary faces.\par 
            Let $p \in Z^\emptyset$ be a point of codimension 1. Then $p = (x, t)$ lies in some $M_i \times [0, \eps)$. If $t=0$, $x$ must lie in the interior of $M_i$ (since $\Gamma(p)=1$), so $p$ lies in the interior of $M_i \times \{0\}$.\par 
            If $t \neq 0$, $p$ lies in the interior of some boundary face $F$ of $M_i$. Then $p$ lies in the interior of $Z^{\{F\}}$.
        \end{proof}
       
    \subsection{Smooth functions on manifolds with faces}
        \begin{defn}
            Let $f: M \rightarrow \bR$ be a smooth map, where $M$ is a manifold with faces. We define a regular value of $f$ inductively in the dimension. We say $p \in \bR$ is a \emph{regular value} of $f$ if $df_x: T_x M \rightarrow T_p \bR$ is surjective for all $x \in f^{-1}\{p\}$, and the restrictions of $f$ to all faces of $M$ have $p$ as a regular value.
        \end{defn}
        Sard's theorem implies that for a fixed $f$, this condition is satisfied for generic $p$. We also have a version of the preimage theorem for manifolds with faces:
        \begin{lem}\label{preimage theorem}
            Let $f: M \rightarrow \bR$ be a smooth map, where $M$ is a manifold with faces and $p \in \bR$ a regular value. Then $f^{-1}\{p\}$ and $f^{-1} (-\infty, p]$ are naturally smooth manifolds with faces.\par 
            Let $F$ be a face of $M$. Then $F \cap f^{-1}\{p\}$ is a face of $f^{-1}\{p\}$, and any face of $f^{-1}\{p\}$ is a union of components of such things.\par 
            Similarly, $F \cap f^{-1}\{p\}$ and $F \cap f^{-1}(-\infty, p]$ are faces of $f^{-1}(-\infty, p]$, and any face of $f^{-1}(-\infty, p]$ is a union of components of such things.
        \end{lem}
        \begin{proof}
        Write $M = \cup_i M_i$ where $M_i$ denotes the set of points of codimension at most $i$, so $M_0 = M^{\circ}$. We proceed by induction on $i$, with the case $i=1$ being the usual preimage theorem for manifolds with boundary. Let $p$ be a regular value of $f$; we assume that $f^{-1}(p) \cap M_{i-1}$ is a smooth manifold with corners, whose corner strata are given by the complements $f^{-1}(p) \cap M_j \backslash M_{j-1}$ for $j<i$. Now let $x\in f^{-1}(p) \cap M_i$, and pick a chart
        \[
        \bR^{n-i} \times [0,\infty)^i \supset U \stackrel{\iota}{\longrightarrow} M
        \]
        near $x$, and translate in the range so $p=0$. By smoothness of $f$, we can find $g$ an extension of $f\circ\iota$ to a neighbourhood of $x \in \bR^n$, so that $g^{-1}(0) = W \subset U$ is locally a smooth hypersurface (without boundary) near $x$. 

        By the inductive hypothesis, $W \cap \bR^{n-i+1} \times [0,\infty)^{i+1}$ is a smooth manifold with the expected corner strata. Write $pr_j$ for the $j$-th co-ordinate projection on $\bR^n$.  Note 
        \[
        W \cap \bR^{n-i} \times [0,\infty)^i = \left\{ w \in W \cap \bR^{n-i+1} \times [0,\infty)^{i+1} \, | \, pr_j(w) \geq 0 \, \forall n-i+1 \leq j \leq n \right\}.
        \]
        Pick a point $w$ where the projection $pr_j$ vanishes for at most $i-1$ indices $j\in \{n-i+1,\ldots,n\}$ and let $j_0$ denote the exceptional index. It suffices to prove that $0$ is a regular value of $pr_{j_0}$. But if not, then $T_wW \subset \bR^{n-1} = \{j_0=0\}$ and then a dimension count shows that $\ker(dg|_w)$ has dimension at most $n-2-i$, which violates the hypothesis of regularity of $f$ on the codimension $i$ locus.

        It follows by induction that $f^{-1}(p)$ is a manifold with corners, whose corners are given by the level sets of $f$ on the corner strata of $M$.  Since embeddedness of closures of components is inherited from $M$, it follows that $f^{-1}(p)$ is a manifold with faces, and its faces are as claimed. The result for $f^{-1}(-\infty,p]$ is analogous.
        \end{proof}
        \begin{lem}\label{extending smooth maps}
            Let $M$ be a manifold with faces and $\{F_i\}$ a collection of faces of any codimension in $M$ (we do not assume they form a system of faces of $M$ or even that they cover the whole of $\partial M$). Assume we have smooth maps $f_i: F_i \rightarrow \bR$ which agree on the overlaps of the $F_i$ in $M$. Then there is a smooth map $f: M \rightarrow \bR$ which restricts to $f_i$ on each $F_i$. \par 
            If all the $f_i$ are proper over $(-\infty, 1+2\delta]$ for some $\delta >0$, then an extension $f$ can be chosen which is proper over $(-\infty, 1+\delta]$; similarly if the $F_i$ are proper over $(-\infty, 1+2\delta]$ and also bounded below.
        \end{lem}
        \begin{proof}
            By induction, we may assume $M$ is equipped with a system of faces and we are given smooth maps $F \rightarrow \bR$ for all faces $F$, which agree on their overlaps. By using a partition of unity, it then suffices to construct the extension $f$ in the case when $M = \bR^{n-j} \times \bR_{\geq 0}^{j}$.\par
            
            The existence of some smooth extension can be given directly, for instance,  one can take
            $$f(x, y) := \sum\limits_{1 \leq k_1 < \ldots < k_l \leq j} (-1)^{1+l} f_{k_1}(x, \bar{y})$$
            where in each term, $\bar{y}$ is $y$ with the $k_1^{\mathrm{th}}, \ldots, k_l^\mathrm{th}$ entries set to 0. 
            
            We now establish the claim on properness. 
            Let $f: M \to \bR$ be a global extension of the $\{f_i\}$, so $f: \partial M \to \bR$ is proper over the given interval $(-\infty, 1+2\delta]$.  Pick a metric on $M$, with distance function $d$, and let 
            \[
            U = \{x \in M \, | \, d(x,\partial M) < \delta \} 
            \]
            which is an open neighbourhood of the boundary $\partial M$. We further pick a retraction
            \[
            r: Cl(U) \to \partial M
            \]
            with the property that $d(x,r(x)) < \delta'$ for $x\in Cl(U)$, where $\delta' > \delta$ is determined by the Lipschitz constant of the retraction $r$. We now let
            \[
            U \supset U' = \{ x\in U \, | \, d(f(x),f(rx)) < \delta/2\}
            \]
            which is a smaller open neighbourhood of $\partial M$. 
            Take a partition of unity $\lambda, \lambda'$ subordinate to the open cover $U', M^\circ$ of $M$, where $M^{\circ}$ denotes the interior (complement of all boundary faces). We consider the function
            \[
            G(x) = \lambda(x) f(x) + \lambda'(x) (1+2\delta)
            \]
            which also restricts to $f$ on $\partial M$. 
            For $[a,b] \subset G^{-1}(-\infty, 1+\delta])$ consider $G^{-1}[a,b]$. If $x \not\in Cl(U')$ then $\lambda'(x)=1$ and $\lambda(x)=0$ so $G(x) > 1+\delta$, so $G^{-1}[a,b] \subset U' \subset U$.

            Now suppose we have a sequence $x_n \in G^{-1}[a,b]$ which does not converge. The condition that $d(f(x_n),f(r(x_n))) < \delta/2$ shows that $r(x_n) \in f^{-1}[a-\delta/2, b+\delta/2]$. This is a compact subset of $\partial M$ by properness of $f|_{\partial M}$ over $(-\infty, 1+2\delta]$. Therefore $r(x_n)$ have some limit $w$. If $d(x_n,x_0) \to \infty$ then $d(x_n, r(x_n))$ will also diverge to infinity, since $r(x_n)$ converge in $\partial M$, a contradiction. It follows that eventually the $x_n$ lie in a closed ball inside $U'$ and hence have a convergent subsequence. This establishes properness of the extension $G$ over $(-\infty, 1+\delta]$.  If the $F_i$ are in addition bounded below then so is the constructed $G$.
            \end{proof}

            \begin{rmk}
                In the situation of the previous Lemma, if $1$ is a regular value of all the $F_i$, then we may construct the extension $f$ (proper over an interval and bounded below as appropriate) so that $1$ is still a regular value. 
            \end{rmk}
            
    \subsection{Framed functions}
        In this section, we consider smooth maps $f:M \to \bR$ equipped with some extra data, that will be useful when we consider stable framings.
        \begin{defn}
            Let $M$ be a (possibly non-compact) manifold with faces and a system of collars. A \emph{framed function}  is a pair $(f, s)$, where $f: M \to \bR$ is a smooth map with 1 as a regular value and $s: \bR \to TM|_{f^{-1}\{1\}}$ is a map of vector bundles over $f^{-1}\{1\}$, satisfying
            \begin{enumerate}
                \item $f$ is negative on compact faces of $M$.
                \item $f^{-1}(-\infty, 2]$ is compact.
                \item $df$ vanishes on each inwards-pointing normal vector along each codimension 1 face.
                \item Over each codimension 1 face $F$, $s$ lives inside this face, i.e. $s|_F$ factors through $TF \subseteq TM|_F$. 
                \item The composition 
                $$\bR \xrightarrow{s} TM|_{f^{-1}\{1\}} \xrightarrow{df} \bR$$
                is the identity.
            \end{enumerate}
        \end{defn}
        The final condition implies that a framed function determines an isomorphism of vector bundles, which we call $S=di+s: Tf^{-1}\{1\} \oplus \bR \to TM|_{f^{-1}\{1\}}$, where $i:f^{-1}\{1\} \to M$ is the inclusion. 
        \begin{prop}\label{prop: res of fram fn is fram}
            Let $M$ be a manifold with faces and a system of collars, and $(f, s)$ a framed function on $M$. Then the restriction of $(f, s)$ to a face $F$ of $M$ is a framed function.\par 
            Furthermore $f^{-1}\{1\}$ is a manifold with faces, with a system of faces given by the set of $f|_F^{-1}\{1\}$ for $F$ a non-compact face of $M$, and $f^{-1}\{1\}$ inherits a natural system of collars.\par 
            The isomorphisms $S$ are compatible with restriction to the boundary, meaning for each boundary face $F$ of $M$, there is a commutative diagram of isomorphisms of vector bundles over $\im \cC^F$:
            \[\xymatrix{
                Tf|_F^{-1}\{1\} \oplus \bR \oplus \bR\nu^F \ar[r]_{S} \ar[d] &
                TF \oplus \bR \nu^F \ar[d] \\
                Tf^{-1}\{1\} \oplus \bR \ar[r]_S &
                TM
            }
            \]
        \end{prop}
        \begin{proof}
        We first show that $(f,s)$ restricts to a face $F$ as a framed function. Since every compact face of $F$ is a compact face of $M$, the first two conditions in the definition are immediate. 
        Let $G\subset F$ be a codimension one face of $F$ with inwards-pointing normal $\nu$. Then there is some codimension one face $\hat{G} \subset M$ containing $F$ with inwards normal $\nu$. It follows that $d(f|_F)(\nu) = 0$, which establishes the third condition. The fourth condition similarly follows from the fact that any codimension one face $G\subset F$ is an intersection $F\cap \hat{G}$ of $F$ with another codimension one face $\hat{G}\subset M$, and that $s|_{\hat{G}}$ factors through $T\hat{G}$. The final condition $df \circ s = \id$ restricts to give the same identity and hence splitting on $TF|_{f^{-1}(1)}$.

        The second paragraph holds by Lemma \ref{preimage theorem}, and the diagram is then immediate from the construction.
        \end{proof}
        \begin{prop}\label{prop: ext fram fn}
            Let $M$ be a manifold with faces and a system of collars. Assume each boundary face $F$ of $M$ is equipped with a framed function $(f^F, s^F)$, so that these agree on the overlaps between different boundary faces.\par 
            Then there exists a framed function $(f, s)$ on $M$ restricting to $(f^F, s^F)$ over each boundary face $F$ of $M$.
        \end{prop}
        \begin{proof}
            We first choose an extension $f$ of each $f^F$, which we can do by Lemma \ref{extending smooth maps}; postcomposing with an appropriate automorphism of $\bR$ if necessary to make it proper over $f^{-1}(-\infty, 2]$. Generically perturbing if necessary we may assume $f$ has 1 as a regular value.\par 
            For fixed $f$, the space of bundle maps $s: \bR \to TM|_{f^{-1}\{1\}}$ satisfying the appropriate conditions is defined by a local condition, which is also a convex condition, so there exists an appropriate extension by a standard partition of unity argument.
        \end{proof}
    \subsection{Stable isomorphisms}
        In this section, we fix some conventions regarding stable vector bundles and isomorphism classes between them; in particular we do not work up to homotopy here. We implicitly assume all vector spaces and bundles we encounter are equipped with fibrewise inner products, compatible with restriction and direct sum decompositions; such choices always exist.
        \begin{defn}
            Let $E=E^+-E^-$ and $F=F^+-F^-$ be stable vector bundles over a compact base $B$ (so $E^\pm, F^\pm$ are both actual vector bundles). A \emph{stable isomorphism} between $E$ and $F$ is a pair $(S, \psi)$, where $S \to B$ is a vector bundle, and $\psi$ is an isomorphism of vector bundles
            $$\psi: E^+ \oplus F^- \oplus S \to F^+ \oplus E^-\oplus S$$
            We often write this as
            $$\psi: E \to F$$
            We call $S$ the \emph{stabilising bundle}.\par 
            Given a stable isomorphism $(S, \psi)$ and an embedding $\iota: S \hookrightarrow T$, we obtain another stable isomorphism $(T, \psi')$, by setting $\psi'$ to be the identity on the orthogonal complement of $S$ inside $T$, and to be $\psi$ on the rest; we call the resulting stable isomorphism the \emph{extension} of $(S, \psi)$ \emph{to} $T$ or \emph{along} $\iota$.\par 
            Given stable isomorphisms $(S, \psi): E \to F$ and $(T, \phi): F \to G$, there is a stable isomorphism $(S \oplus F^- \oplus T, \theta)$, where $\theta$ is given by $\phi$ composed with $\psi$, acting on the appropriate summands. Composition of stable isomorphisms is then associative, though see Remark \ref{rmk:isbell}.
        \end{defn}
        \begin{rmk}\label{rmk:isbell}
            In the usual model for the category of vector spaces it is not true that the direct sum is strictly associative (meaning $A \oplus (B \oplus C)$ and $(A \oplus B) \oplus C$ aren't equal, they are instead related by a canonical associator isomorphism). There are two ways around this issue; we briefly sketch them both here, and the reader may choose which to apply. 
            \begin{enumerate}
                \item Following \cite{Isbell} (see also \cite[Section 2.8]{Lawson-Lipshitz-Sarkar}), one may work with a category equivalent to the category of vector spaces which is equipped with a strictly associative direct sum operation.
                \item One may keep track of all these canonical associator isomorphisms; in many cases when we compare two stable isomorphisms, we do not require that they are equal on the nose, but instead require that they agree after applying appropriate associator isomorphisms.
            \end{enumerate}
        \end{rmk}
        We may apply the same definition to stable vector spaces, which we think of as stable vector bundles over a point.
        \begin{rmk}\label{rmk:fram sgn}
           Given stable isomorphisms $(S, \psi): E \oplus V \to F$ and $(T, \phi): F \oplus W \to G$, there are now two natural stable isomorphisms $E \oplus V \oplus W \to G$.\par 
           The first is the following composition of stable isomorphisms:
           $$E \oplus V \oplus W \xrightarrow{\psi} F \oplus W \xrightarrow{\phi} G$$
           constructed as above, and in particular has stabilising bundle $(S \oplus F^- \oplus T)$.\par  
           The second is the vector bundle isomorphism $\psi \oplus \phi$ with stabilising bundle $S \oplus F^+ \oplus F^- \oplus T$.\par
           When working with flow categories we will use the first construction, but when equipping the product of stably framed manifolds with a stable framing (as is needed to equip the framed bordism groups with the structure of a ring) we use the second.\par 
           If we fix stable isomorphisms from each of $E,F,G,V,W$ with $\bR^i$ for some $i$, these two constructions give the same homotopy class of stable framing up to the usual Koszul sign which comes from swapping factors of $\bR$ past each other.
        \end{rmk}
\section{A category of flow categories}\label{sec:flow cat}
    \subsection{Flow categories}\label{sec:unor flow cat}
        \begin{defn}
            A \emph{flow category} $\cF$ consists of a finite set $ob \cF$ (often abbreviated just to $\cF$), a function (called the \emph{grading}) $|\cdot |: \cF \rightarrow \bZ$, and compact smooth manifolds with faces $\cF_{xy}$ of dimension $|x| - |y| - 1$ for all $x, y \in \cF$. There are also \emph{concatenation maps} 
            $$c = c^\cF: \cF_{xz} \times \cF_{zy} \hookrightarrow \cF_{xz}$$
            which are the inclusion of the boundary face, such that the collection of these (allowing $z$ to vary) forms a system of boundary faces for $\cF_{xy}$. We further assume that the concatenation maps are associative, meaning the following diagram commutes:
            \[
                \xymatrix{
                    \cF_{xz} \times \cF_{zw} \times \cF_{wy} \ar[rr]^{c} \ar[d]_{ c}  && \cF_{xz} \times \cF_{zy} \ar[d]^{c} \\
                    \cF_{xw} \times \cF_{wy} \ar[rr]_{c} && \cF_{xy}
               }
            \]
            We will often write $\cF_{x_0 \ldots x_i}$ as shorthand for $\cF_{x_0 x_1} \times \ldots \times \cF_{x_{i-1} x_i}$. These are faces of $\cF_{x_0 x_i}$, and any face of $\cF_{x_0 x_i}$ is a union of components of these. In other words, there is a system of faces for $\cF_{xy}$ given by
            $$\cF_{x_0 \ldots x_i}$$
            for $x_0 = x, \ldots, x_i = y$ in $\cF$; similarly each $\cF_{xx''}$ has a system of boundary faces given by the $\cF_{xx'x''}$ as $x'$ varies over $\cF$.\par 
            A \emph{system of collars} on a flow category $\cF$ consists of a system of collars $\cC_{xy}$ on all $\cF_{xy}$, restricting to the product system of collars on $\cF_{xz} \times \cF_{zy}$ induced by Lemma \ref{products and collars} from $\cC_{xz}$ and $\cC_{zy}$.\par 
            A \emph{collared flow category} is a flow category $\cF$ equipped with a system of collars.
        \end{defn}
        Note that in particular, a flow category forms a (non-unital) topological category.\par 
        It follows from Lemmas \ref{products and collars} and \ref{systems of collars exist} that
        \begin{lem}
            Any flow category admits a system of collars, which is unique up to isotopy.
        \end{lem}
        Given this, we will often choose systems of collars on flow categories and denote them $\cC$, without explicitly stating that we do so, and we will generally omit the word collared.
        \begin{ex}\label{ex:mor}
            Given a Morse function $f$ on a closed manifold $M$ and a suitably generic metric $g$, Wehrheim \cite{Wehrheim} constructs a flow category $\overline\cM^f$ with objects critical points of $f$, gradings given by the Morse indices, and $\overline\cM^f_{xy}$ given by compactified moduli spaces of gradient flow lines between any two given critical points.
        \end{ex}
        \begin{ex}
            Let $M$ be a closed manifold of dimension $n$. There is a flow category $\cF$ with two objects $x$ and $y$, with $|x| = n+1$, $|y| = 0$, $\cF_{xy} = M$ and no other data.
        \end{ex}
        \begin{example}
            Let $\cF$ be a flow category with three objects $x,y,z$, with $|x|>|y|>|z|$. Let $m=|x|-|y|$ and $n=|y|-|z|$. Then the manifold data of $\cF$ consists of a closed $(m-1)$-manifold $\cF_{xy}$, a closed $(n-1)$-manifold $\cF_{yz}$, and a nullbordism $\cF_{xz}$ of the product of these closed manifolds.
        \end{example}
        \begin{defn}
            
            $*$ is the flow category with one object (which we also call $*$), graded in degree 0. This admits a canonical system of collars.\par 
            For a flow category $\cF$, we write $\cF[i]$ for the flow category which is the same as $\cF$, except all the gradings are increased by $i$. Explicitly, this means $|\cdot|^{\cF[i]} = |\cdot|^\cF + i$.\par 
            A \emph{morphism} of flow categories $\mathcal{W}: \mathcal{F} \rightarrow \mathcal{G}$ consists of compact smooth manifolds with faces $\mathcal{W}_{xy}$ of dimension $|x|-|y|$ for $x$ in $\mathcal{F}$ and $y$ in $\mathcal{G}$, along with maps
            $$c = c^\cW: \cF_{xz} \times \cW_{zy} \rightarrow \cW_{xy}$$
            and
            $$c = c^\cW: \cW_{xz} \times \cG_{zy} \rightarrow \cW_{xy}$$
            which are suitably associative and compatible with the $c^\cF$ and $c^\cG$ (in the sense that the following diagrams commute:
            \[\xymatrix{
                \cF_{xx'} \x \cF_{x'x''} \x \cW_{x''y} \ar[r]_-{c^\cW} \ar[d]_{c^\cF}  &
                \cF_{xx'} \x \cW_{x'y}\ar[d]_{c^\cW} \\
                \cF_{xx''} \x \cW_{x''y} \ar[r]_-{c^\cW} &
                \cW_{xy}
            }
            \]
            \[\xymatrix{
                \cF_{xx'} \x \cW_{x'y'} \x \cG_{y'y} \ar[r]_-{c^\cW} \ar[d]_{c^\cW} &
                \cF_{xx'} \x \cW_{x'y} \ar[d]_{c^\cW}\\
                \cW_{xy'} \x \cG_{y'y} \ar[r]_-{c^\cW} &
                \cW_{xy}
            }
            \]
            \[\xymatrix{
                \cW_{xy''} \x \cG_{y''y'} \x \cG_{y'y} \ar[r]_-{c^\cG} \ar[d]_{c^\cW} &
                \cW_{xy''}\x \cG_{y''y}\ar[d]_{c^\cW} \\
                \cW_{xy'} \x\cG_{y'y} \ar[r]_-{c^\cW} &
                \cW_{xy}
            }
            \]
            which define a system of faces
            $$\cF_{x_0 \ldots x_i} \times \cW_{x_i y_0} \times \cG_{y_0 \ldots y_j}$$
            for $x = 0, \ldots x_i$ in $\cF$ and $y_0, \ldots, y_j = y$ in $\cG$.\par

            A \emph{bordism} $\mathcal{R}$ between two such morphisms $\mathcal{W}$ and $\mathcal{V}$ consists of compact smooth manifolds with faces $\mathcal{R}_{xy}$ (of dimension $|x|-|y|$ + 1) for $x$ in $\mathcal{F}$ and $y$ in $\mathcal{G}$, along with maps
            $$c=c^\cR: \cW_{xy}, \cV_{xy} \rightarrow \cR_{xy}$$
            and 
            $$c = c^\cR: \cF_{xz} \times \cR_{zy} \rightarrow \cR_{xy}$$
            and
            $$c = c^\cR: \cR_{xz} \times \cG_{zy} \rightarrow \cR_{xy}$$
            compatible with $c^\cF$, $c^\cG$, $c^\cW$, $c^\cV$ (meaning that similar diagrams to those above over each codimension two face commute), such that the images of the faces in the domains defines a system of faces for $\cR_{xy}$ (after deleting repeats). In other words, there is a system of faces for $\cR_{xy}$ given by the following:
            $$\cW_{x_0 \ldots x_i; y_0 \ldots y_j}$$
            and
            $$\cV_{x_0 \ldots x_i; y_0 \ldots y_j}$$
            and
            $$\cR_{x_0 \ldots x_i; y_0 \ldots y_j}$$
            for $x_0=x, \ldots, x_i$ in $\cF$ and $y_0, \ldots, y_j$ in $\cG$. We write $\cW_{x_0 \ldots x_i; y_0 \ldots y_j}$ as shorthand for $\cF_{x_0 \ldots x_i} \times \cW_{x_i y_0} \times \cG_{y_0 \ldots y_j}$, and similarly for $\cV$, $\cR$ and any similar construction we encounter later. \par
            A \emph{system of collars} on a morphism $\cW$ is a system of faces on $\cF$, $\cG$ and all $\cW_{xy}$, which are compatible with the product systems of collars on all faces. A \emph{system of collars} on a bordism $\cR$ between $\cW$ and $\cV$ is a system of collars on $\cF$, $\cG$, $\cW$, $\cV$ and all $\cR_{xy}$, which are compatible with the product system of collars on all faces.
        \end{defn}
        \begin{defn}
            A morphism from $\ast[i]$ to $\cF$ is called a (right) \emph{flow module} over $\cF$ of \emph{degree} $i$. Similarly, a morphism from $\cF$ to $\ast[j]$ is a \emph{left flow module} of degree $j$.
        \end{defn}
        \begin{rmk}
            Right (respectively left) flow modules over $\cF$ are in particular right (left) modules over $\cF$ (viewed as a (non-unital) topological category). Similarly a morphism $\cW: \cF \to \cG$ in particular forms a bimodule over these two topological categories. Similarly a bilinear map (see Section \ref{sec:bilin}) of flow categories forms a multimodule over the relevant flow categories.
        \end{rmk}
        It follows from Lemma \ref{systems of collars exist} that 
        \begin{lem}
            Every morphism $\cW: \cF \rightarrow \cG$ of flow categories admits a system of collars, and if $\cF$ and/or $\cG$ were already equipped with systems of collars, this can be chosen to be compatible with those.\par 
            Every bordism $\cR$ between morphisms $\cW, \cV: \cF \rightarrow \cG$ admits a system of collars, and if some of $\cF$, $\cG$, $\cW$ and $\cV$ are already equipped with systems of collars (compatibly with each other), this can be chosen to be compatible with these.\par 
            These are all unique up to isotopy.
        \end{lem}
        Given this, we will often choose systems of collars on morphisms or homotopies and denote them $\cC$, without explicitly stating that we do so. If we have already chosen systems of collars on the domain or target flow categories or morphisms, we choose these systems of collars to extend them. We will similarly implicitly do this on future manifolds we counter.
        \begin{ex}
            Setting $\cW_{xy} = \emptyset$ for all $x \in \cF$, $y \in \cG$ defines a map of flow categories $\cF \rightarrow \cG$, which we call the \emph{empty morphism}.\par 
            If $\cW, \cV: \cF \rightarrow \cG$ are morphisms, setting
            $$\left(\cW \sqcup \cV\right)_{xy} := \cW_{xy} \sqcup \cV_{xy}$$
            defines a morphism $\cW \sqcup \cV: \cF \rightarrow \cG$.
        \end{ex}
        \begin{ex}
            Let $\cW: \cF \rightarrow \cG$ be a morphism of flow categories. There is a flow category $\Cone(\cW)$ with objects $\cF[1] \sqcup \cG$, with morphisms given by
            $$\Cone(\cW)_{xy} := \begin{cases}
                \cF_{xy} & \,\mathrm{if}\,x,y\in\cF\\
                \cG_{xy} & \,\mathrm{if}\,x,y\in\cG\\
                \cW_{xy} & \,\mathrm{if}\,x\in\cF,\,y\in\cG\\
                \emptyset & \,\mathrm{if}\,x\in\cG,\,y\in\cF
            \end{cases}$$
            Then $\Cone(\cW)$ contains $\cF[1]$ and $\cG$ as full subcategories.
        \end{ex}
        \begin{defn}\label{flow morphism def}
            Let $\mathcal{F}$ and $\mathcal{G}$ be flow categories. Define $[\cF, \cG]$ to be the set of flow morphisms $\cF \rightarrow \cG$, modulo the equivalence relation generated by
            \begin{enumerate}
                \item $\cW \sim \cV$ if $\cW$ and $\cV$ are bordant.
                \item $\cW \sim \cV$ if $\cW$ and $\cV$ are diffeomorphic (meaning there are diffeomorphisms $\cW_{xy} \cong \cV_{xy}$ for all $x$ in $\cF$ and $y$ in $\cG$, compatible with all concatenation maps).
            \end{enumerate}
            This is a commutative monoid under disjoint union, with unit given by the empty morphism.
        \end{defn}
        \begin{lem}\label{lem:unor inv}
            $[\cF, \cG]$ contains inverses, and hence is a group.
        \end{lem}
        \begin{proof}
            Let $\cW:\cF \to \cG$ be a morphism. Then letting $\cR_{xy}=\cW_{xy}\times[0,1]$ for $x$ in $\cF$ and $y$ in $\cG$ provides a nullbordism of $\cW \sqcup \cW$; this implies that $\cW$ is its own inverse.
        \end{proof}
        \begin{example}\label{ex:pt unor}
                A morphism $*[i] \to *$ consists of a closed $i$-manifold; a bordism between such morphisms consists of a  bordism between the relevant closed manifolds in the usual sense. From this, we see that
                \[
                [*[i],*] = \Omega_i(\ast) = \Omega_i
                \]
                reconstructs the usual (unoriented) bordism groups of a point $\ast$.
            \end{example}
        Given morphisms of flow categories $\mathcal{W}: \mathcal{F} \rightarrow \mathcal{G}$ and $\mathcal{V}: \mathcal{G} \rightarrow \mathcal{H}$, we construct their composition as follows.\par 
        We first define (non-compact) manifolds with faces $\tilde{\cQ}_{xz}$ for all $x \in \cF$, $z \in \cH$, by
        \begin{equation}\label{tilde Q}
            \tilde{\cQ}_{xz} = \tilde{\cQ}_{xz}(\cW, \cV) := \left(\bigsqcup_{y \in \cG} \cW_{xy} \times \cV_{yz} \times [0, \varepsilon) \right) / \sim
        \end{equation}
        where $\sim$ is the equivalence relation such that for all $a \in \cW_{xy}$, $b \in \cG_{yy'}$, $c \in \cV_{y'z}$, $(u,v) \in [0, \varepsilon)^2_{u,v}$, 
        \begin{equation}\label{equiv}
        (a, \cC(b, c, u), v) \sim (\cC(a, b, v), c, u)
        \end{equation}
        As a topological space, we can write this as the coequaliser of the two maps above:
        $$\bigsqcup_{y,y' \in \cG} \cW_{xy} \times \cG_{yy'} \times \cV_{yz} \times [0, \varepsilon)^2 \rightrightarrows \bigsqcup_{y \in \cG} \cW_{xy} \times \cV_{yz} \times [0, \varepsilon)$$
        By Lemma \ref{lem:abstract gluing}, $\tilde{\cQ}_{xz}$ is a manifold with corners. There are natural embeddings into it from the following manifolds with faces.
        $$ \cW_{x_0 \ldots x_i; y_0} \times \cG_{y_0 \ldots y_j} \times \cV_{y_j; z_0 \ldots z_k} \times \{0\}$$
        and 
        $$ \cF_{x_0 \ldots x_i} \times \tilde{\cQ}_{x_i z_0} \times \cH_{z_0 \ldots z_j}$$
        for $x_0 = x, \ldots, x_i$ in $\cF$, $y_0, \ldots, y_j$ in $\cG$ and $z_0, \ldots, z_k = z$ in $\cH$. 
        \begin{lem}
            These form a system of faces for $\tilde{\cQ}_{xz}$. In particular, it is a manifold with faces.
        \end{lem}
        \begin{proof}
            This follows from Lemmas \ref{lem:abstract gluing} and \ref{lem:abstract gluing with faces}; we spell this out in this case. We use $I=ob \cG$, $J = ob\cF \sqcup ob\cH$, set $M_z = \cW_{xz} \times \cV_{yz}$, $N_{yy'} = \cW_{xy} \times \cG_{yy'} \times \cV_{y'z}$, $i_{yy'} := y$ and $j_{yy'} := y'$. We take each $\alpha_{yy'}, \beta_{yy'}$ to be given by $c$.\par 
            For $y' \in I \setminus \{y\}$, we let $\lambda_y(y')$ be $\cW_{xy} \times \cG_{yy'} \times \cV_{y'z}$ (swapping $y$ and $y'$ in this expression if $|y'|>|y|$), and for $x' \in \cF$ and $z' \in \cH$, we set $\lambda(x')$ to be $\cW_{xx';y} \times \cV_{yz}$ and $\lambda_y(z')$ to be $\cW_{xy} \times \cV_{y;z'z}$; we similarly set $\lambda_{yy'}(y'')$ to be $\cW_{xy} \times \cG_{yy'} \times \cG_{y'y''} \times \cV_{y''z}$ (swapping the order of the elements of $\cG$ if necessary), and setting $\lambda_{yy'}(x')$ to be $\cW_{xx';y} \times \cG_{yy'} \times \cV_{yz}$ and $\lambda_{yy'}(z')$ to be $\cW_{xy} \times \cG_{yy'} \times \cV_{y';z'z}$.
        \end{proof}
        
        We now choose framed functions $(f_{xz}: \tilde{\cQ}_{xz} \rightarrow \bR, s_{xz})$ 
        such that $(f_{xz}, s_{xz})$ restricts to $(f_{x'z}, s_{x'z})$ and $(f_{xz'}, s_{xz'})$ on the boundary faces $\tilde{\cQ}_{xx';z}$ and $\tilde{\cQ}_{x;z'z}$ respectively. This can be done, by Lemma \ref{extending smooth maps}. Note that we won't use the sections $s_{xz}$ until we deal with framings in Section \ref{Sec:framings}.\par 
        Let $\cQ_{xz}^{\leq 1} = f_{xz}^{-1} (-\infty, 1]$ and $\cU_{xz} = f_{xz}^{-1}\{1\}$. Noting that $\cU_{xz}$ is always disjoint from $\cW_{xy} \times \cV_{yz}$ for all $y$, Lemma \ref{preimage theorem} implies that $\cQ_{xz}^{\leq 1}$ is a compact manifold with faces, with a system of faces given by
        $$ \cW_{x_0 \ldots x_i; y_0} \times \cG_{y_0 \ldots y_j} \times \cV_{y_j; z_0 \ldots z_k}$$
        and 
        $$ \cF_{x_0 \ldots x_i} \times \cQ^{\leq 1}_{x_i z_0} \times \cH_{z_0 \ldots z_j}$$
        and
        $$ \cF_{x_0 \ldots x_i} \times \cU_{x_i z_0} \times \cH_{z_0 \ldots z_j}$$
        for $x_0 = x, \ldots, x_i$ in $\cF$, $y_0, \ldots, y_j$ in $\cG$ and $z_0, \ldots, z_k = z$ in $\cH$.\par 
        Lemma \ref{preimage theorem} also implies that $\cU_{xz}$ is a manifold with faces, with a system of faces given by 
        $$\cF_{x_0 \ldots x_i} \times \cU_{x_i z_0} \times \cH_{z_0 \ldots z_j}$$
        for $x_0 = x, \ldots, x_i$ in $\cF$ and $z_0, \ldots, z_j = z$ in $\cH$.
        This implies the collection $\cU_{xz}$ form a morphism of flow categories $\cU: \cF \rightarrow \cH$.
        \begin{defn}
            We define the \emph{composition} of $\cW$ and $\cV$ to be $\cU$.
        \end{defn}
        \begin{lem} \label{composition well-defined}
            The equivalence class of $\cU$ in $[\cF, \cH]$ is independent of the intermediate choices.
        \end{lem}
        \begin{proof}
            The choices made in constructing $\cU$ were the choice of framed functions $(f_{xz}, s_{xz})$. \par  
            Let $1\geq q > p > 0$ be regular values of all the $f_{xz}$, and let $\cU^{f,p}_{xz}:= f^{-1}_{xz}\{p\}$ and $\cU^{f, q}_{xz} := f^{-1}_{xz}\{q\}$.\par 
            Let $\cR_{xz} = f_{xz}^{-1}[p, q]$. These form a bordism $\cR$ between $\cU^{f, p}$ and $\cU^{f, q}$. In particular, $\cU^{f, p}$ and $\cU^{f,q}$ are both (bordant to) allowable constructions of representatives of $\cV \circ \cW$. \par 
            Suppose we are given two choices of framed functions: say $(f_{xz},s_{xz})$ and $(f'_{xz}, s'_{xz})$. Let $1 > p > 0$ be a regular value of all $f_{xz}$, and let $1 \geq q > p$ be a regular value of all $f'_{xz}$ close enough to 1 such that $f'_{xz}(f_{xz}^{-1}(-\infty, \mu]) < \mu$ for all $x, z$, where $\mu:= \frac{p+q}2$. We choose smooth functions $\alpha, \beta: \bR \rightarrow [0, 1]$ such that 
            $$\alpha(t) = \begin{cases} 
                1 & \, \mathrm{ if }\, t \leq p\\
                0 & \, \mathrm{ if }\, \mu \leq t
            \end{cases}$$
            and
            $$\beta(t) = \begin{cases} 
                0 & \, \mathrm{ if }\, t \leq \mu\\
                1 & \, \mathrm{ if }\, q \leq t
            \end{cases}$$
            We then define $g_{xz}: \tilde{\cQ}_{xz} \rightarrow \bR$:
            $$
            g_{xz}(r) :=
            \begin{cases}
                f_{xz}(r) & \, \mathrm{if}\, f_{xz}(r) \leq p\\
                
                \alpha(f_{xz}(r)) f_{xz}(r) + (1-\alpha(f_{xz}(r))) \left(\mu\right) & \, \mathrm{if}\, p \leq f_{xz}(r) \leq \mu\\
                
                \mu & \, \mathrm{if}\, \mu \leq f_{xz}(r) \, \mathrm{and}\, f'_{xz}(r) \leq \mu\\
                
                \beta(f'_{xz}(r)) f'_{xz}(r) + (1-\beta(f'_{xz}(r))) \left(\mu\right) & \, \mathrm{if}\, \mu \leq f'_{xz}(r) \leq q\\

                f'_{xz}(r) & \, \mathrm{if}\, q \leq f'_{xz}(r)
            \end{cases}$$
            Then taking $\cR_{xz} := g_{xz}^{-1}[p,q]$ forms a bordism $\cR$ between $\cU^{f,p}$ and $\cU^{f',q}$.
        \end{proof}
        \begin{lem}\label{composition is well-defined}
            Composition is compatible with the equivalence relation defined in Definition \ref{flow morphism def}, and induces a well-defined bilinear map
            $$\circ: [\mathcal{G}, \mathcal{H}] \otimes [\mathcal{F}, \mathcal{G}] \rightarrow [\mathcal{F}, \mathcal{H}].$$
        \end{lem}
        \begin{proof}
            Let $\cW, \cW': \cF \rightarrow \cG$ and $\cV, \cV': \cG \rightarrow \cH$ be morphisms. Clearly if $\cW$ and $\cW'$ are diffeomorphic and $\cV$ and $\cV'$ are diffeomorphic then $\cV \circ \cW$ and $\cV' \circ \cW'$ are diffeomorphic (note all $\tilde{\cQ}(\cW,\cV)_{xy}$ and $\tilde{\cQ}({\cW',\cV'})$ are diffeomorphic, and we can pull back framed functions along such a diffeomorphism).\par 
            Let $\cR$ be a bordism between $\cW$ and $\cW'$. We construct a bordism $\cS$ between $\cV \circ \cW$ and $\cV \circ \cW'$ as follows. The strategy is to perform the same construction as that of $\cV \circ \cW$, but with $\cW$ replaced with $\cR$.\par 
            First define 
            $$\tilde{\cS}_{xz} = \tilde{\cS}_{xz}(\cR, \cV) = \left(\bigsqcup_{y \in \cG} \cR_{xy} \times \cV_{yz} \times [0, \varepsilon)\right)/\sim$$
            where $\sim$ is defined the same as in the definition of $\tilde{\cQ}$- more explicitly, we say that for all $a \in \cR_{xy}$, $b \in \cG_{yy'}$, $c \in \cV_{y'z}$, $(u,v) \in [0, \varepsilon)^2$, 
            $$(a, \cC(b, c, u), v) \sim (\cC(a, b, v), c, u).$$
            Lemmas \ref{lem:abstract gluing} and \ref{lem:abstract gluing with faces} show that each $\tilde{\cS}_{xz}$ is a manifold with faces, with a system of faces given by
            $$\cR_{x_0 \ldots x_i; y_0} \times \cG_{y_0 \ldots y_j} \times \cV_{y_j; z_0 \ldots z_k} \times \{0\}$$
            and
            $$\cW_{x_0 \ldots x_i; y_0} \times \cG_{y_0 \ldots y_j} \times \cV_{y_j; z_0 \ldots z_k} \times \{0\}$$
            and
            $$\cW'_{x_0 \ldots x_i; y_0} \times \cG_{y_0 \ldots y_j} \times \cV_{y_j; z_0 \ldots z_k} \times \{0\}$$
            and
            $$\cF_{x_0 \ldots x_i} \times \tilde{\cQ}_{x_i z_0} \times \cH_{z_0 \ldots z_j}$$
            and
            $$\cF_{x_0 \ldots x_i} \times \tilde{\cQ}'_{x_i z_0} \times \cH_{z_0 \ldots z_j}$$
            and
            $$\cF_{x_0 \ldots x_i} \times \tilde{\cS}_{x_i z_0} \times \cH_{z_0 \ldots z_j}$$
            for $x_0 = x, \ldots, x_i$ in $\cF$, $y_0, \ldots, y_j$ in $\cG$ and $z_0, \ldots, z_k = z$ in $\cH$, and where $\tilde{Q} := \tilde{Q}(\cW, \cV)$ and $\tilde{Q}' := \tilde{Q}(\cW', \cV)$.\par
            Pick framed functions $(f_{xz}, s_{xz})$ on all $\tilde{\cQ}_{xz}, \tilde{\cQ}'_{xz}$ and $\tilde{\cS}_{xz}$, restricting to each other on overlaps.\par 
            Let $\cS_{xz}^{\leq 1} = f_{xz}^{-1}(-\infty, 1]$. Similarly to before, Lemma \ref{preimage theorem} implies that $\cS_{xz}^{\leq 1}$ is a compact manifold with faces, with a system of faces given by the above list with all $\tilde{\cS}$s replaces with $\cS^{\leq 1}$s, along with
            $$\cF_{x_0 \ldots x_i} \times \cT_{x_i z_0} \times \cH_{z_0 \ldots z_j}$$
            where $\cT_{xz} := f_{xz}^{-1}\{1\}$. $\cT_{xy}$ is a compact manifold with faces, with a system of faces given by 
            $$\cF_{x_0 \ldots x_i} \times \cU^{f, p}_{x_i z_0} \times \cH_{z_0 \ldots z_j}$$
            and
            $$\cF_{x_0 \ldots x_i} \times \cU'^{f, p}_{x_i z_0} \times \cH_{z_0 \ldots z_j}$$
            and
            $$\cF_{x_0 \ldots x_i} \times \cT_{x_i z_0} \times \cH_{z_0 \ldots z_j}$$
            for $x_0 = x, \ldots, x_i$ in $\cF$ and $z_0, \ldots, z_k = z$ in $\cH$, and where $\cU := \cU^{f, p}(\cW, \cV)$ and $\cU' := \cU^{f, p}(\cW', \cV)$ are representatives of $\cV \circ \cW$ and $\cV \circ \cW'$ respectively. Therefore $\cT$ is a bordism between $\cU^{f, p}$ and $\cU'^{f, p}$, as required.\par 
            The same argument shows that if $\cV$ and $\cV'$ are bordant, so are $\cV \circ \cW$ and $\cV' \circ \cW$.\par 
            It is clear from construction that
            $$(\cV \sqcup \cV') \circ (\cW \sqcup \cW') = (\cV \circ \cW) \sqcup (\cV \circ \cW') \sqcup (\cV' \circ \cW) \sqcup (\cV' \circ \cW')$$
            which implies that composition induces a bilinear map of groups
            $$\circ: [\mathcal{G}, \mathcal{H}] \otimes [\mathcal{F}, \mathcal{G}] \rightarrow [\mathcal{F}, \mathcal{H}]$$
        \end{proof}
        \begin{lem} \label{composition is associative}
            The composition map is associative, i.e. the following diagram commutes:
            \begin{equation}
                \xymatrix{
                    \left[\cH, \cI\right] \otimes \left[\cG, \cH\right] \otimes \left[\cF, \cG\right] \ar[rr]_-{\circ} \ar[d]_{\circ} && \left[\cG, \cI\right] \otimes \left[\cF, \cG\right] \ar[d]^{\circ} \\
                    \left[\cH, \cI\right] \otimes \left[\cF, \cH\right] \ar[rr]_-{\circ} && \left[\cF, \cI\right]
                }
            \end{equation}
        \end{lem}
        \begin{proof}
            Choose morphisms $\cW: \cF \to \cG$, $\cV: \cG \to \cH$ and $\cX: \cH \to \cI$.\par 
            Let $\tilde{\cQ}_{xz} = \tilde{\cQ}(\cW, \cV)_{xz}$ and $\tilde{\cQ}'_{yw} = \tilde{\cQ}_{yw}(\cV, \cX)$ for all $x \in \cF$, $y \in \cG$, $z \in \cH$ and $w \in \cI$.\par
            Choose framed functions $(f_{xz}, s_{xz})$ on each $\tilde{\cQ}_{xz}$ and $(f'_{yw}, s'_{yw})$ on each $\tilde{\cQ}'_{yw}$, suitably compatible on overlaps.\par 
            For $x \in \cF$ and $w \in \cI$, define $\tilde{\cP}_{xw}$ by
            $$\tilde{\cP}_{xw} := \left(\bigsqcup\limits_{z \in \cH} \tilde{\cQ}_{xz} \times \cX_{zw} \times [0, \eps)_s\right)/\sim$$
            where $\sim$ is defined so that $\tilde{\cP}_{xw}$ is the coequaliser of the following diagram:
            \[\xymatrix{
                \bigsqcup\limits_{z,z' \in \cH} \tilde{\cQ}_{xz} \times \cH_{zz'} \times \cX_{z'w} \times [0,\eps)^2_{u,v} \ar@<+.5ex>[r] \ar@<-.5ex>[r] &
                \bigsqcup\limits_{z \in \cH} \tilde{\cQ}_{xz} \times \cX_{zw} \times [0, \eps)_s
            }\]
            where the two maps send $(a, b, c, u, v)$ to $(a, \cC(b,c,u),v)$ and $(\cC(a,b,v),c,u)$ respectively.\par 
            Similarly we define $\tilde{\cP}'_{xw}$ to be the coequaliser of the following diagram:
            \[\xymatrix{
                \bigsqcup\limits_{y,y' \in \cG} \cW_{xy} \times \cG_{yy'} \times \tilde{\cQ}'_{y'w} \times [0,\eps)^2_{u,v} \ar@<+.5ex>[r] \ar@<-.5ex>[r] &
                \bigsqcup\limits_{y \in \cG} \cW_{xy} \times \tilde{\cQ}'_{yw} \times [0, \eps)_s
            }\]
            with the maps defined similarly.\par
            Note $\tilde{\cP}_{xw}$ and $\tilde{\cP}'_{xw}$ both have dimension $|x|-|w|+2$.\par 
            By Lemma \ref{lem:abstract gluing with faces}, $\tilde{\cP}_{xw}$ has a system of boundary faces given by $\tilde{\cP}_{xx';w}$, $\tilde{\cP}_{x;w'w}$ and $\tilde{\cQ}_{xz} \times \cX_{zw}$, for $x' \in \cF$, $w' \in \cI$ and $z \in \cH$. Choose framed functions $(h_{xw}, r_{xw})$ on each $\tilde{\cP}_{xw}$, compatibly on boundary faces.\par 
            Similarly, $\tilde{\cP}'_{xw}$ has a system of boundary faces given by $\tilde{\cP}'_{xx';w}$, $\tilde{\cP}'_{x;w'w}$ and $\cW_{xy} \times \tilde{\cQ}'_{xw}$, for $x' \in \cF$, $w' \in \cI$ and $y \in \cG$. Choose framed functions $(h'_{xw}, r'_{xw})$ on each $\tilde{\cP}'_{xw}$, compatibly on boundary faces.\par 
            Let $\cM_{xw} = h_{xw}^{-1}\{1\}$ and $\cM'_{xw} = h'^{-1}_{xw} \{1\}$; then $\cM$ and $\cM'$ form representatives of $\cX \circ (\cV \circ \cW)$ and $(\cX \circ \cV) \circ \cW$ respectively.\par 
            Now define $\tilde{\cZ}_{xw}$ to be the coequaliser of the following diagram:
            \[\xymatrix{
                \bigsqcup\limits_{\substack{y \in \cG\\z \in \cH}} \cW_{xy} \times \cV_{yz} \times \cX_{zw} \times [0, \eps)^2_{u,v} \ar[r] \ar[dr] &
                \bigsqcup\limits_{y \in \cG} \cW_{xy} \times \tilde{\cQ}_{yw} \times [0, \eps)_s \\
                & \bigsqcup \limits_{z \in \cH} \tilde{\cQ}_{xz} \times \cX_{zw} \times [0, \eps)_s
            }\]
            where the maps are defined similarly to above. This has a system of boundary faces given by $\tilde{\cZ}_{xx';w}$, $\tilde{\cZ}_{x;w'w}$, $\tilde{\cQ}_{xz} \times \cX_{zw} \times \{0\}$, $\cW_{xy} \times \tilde{\cQ}_{yw}' \times \{0\}$, for $x' \in \cF$, $w' \in \cI$, $y \in \cG$ and $z \in \cH$. Choose framed functions $(g_{xw}, t_{xw})$ on the $\tilde{\cZ}_{xw}$ extending the $(f_{xz}, s_{xz})$ on the $\tilde{\cQ}_{xz}$ and the $(f'_{yw}, s'_{yw})$ on the $\tilde{\cQ}'_{xz}$ respectively. Note that faces of $\tilde{\cZ}_{xw}$ of the form $\tilde{\cQ}_{xz} \times \cX_{zw} \times \{0\}$ and $\cW_{xy} \times \tilde{\cQ}_{yw}' \times \{0\}$ can only overlap in compact codimension 2 faces, so when considering $g_{xw}^{-1}\{1\}$ the corresponding faces won't overlap at all, since $g_{xw}$ takes negative values on compact faces.\par 
            Let $\cY_{xw} = g_{xw}^{-1}\{1\}$; this is a compact manifold with faces of dimension $|x|-|w|+1$. Let $\cU_{xz} = f^{-1}_{xz}\{1\}$ and $\cU'_{yw} = f'^{-1}_{yw}\{1\}$; then $\cU$ and $\cU'$ are representatives of $\cV \circ \cW$ and $\cX \circ \cV$ respectively.\par 
            $\cY_{xw}$ has a system of boundary faces given by $\cY_{xx';w}$, $\cY_{x;w'w}$, $\cU_{xz} \times \cX_{zw}$ and $\cW_{xy} \times \cU'_{yw}$, for $x' \in \cF$, $w' \in \cI$, $y \in \cG$ and $z \in \cH$. As noted before, the last two types of face here do not intersect each other: $\cU_{xz} \times \cX_{zw}$ and $\cW_{xy} \times \cU'_{yw}$ are disjoint, for all $y \in \cG$ and $z \in \cH$. \par 
            By the collar neighbourhood theorem, there are disjoint embeddings $\tilde{\cP}_{xw}, \tilde{\cP}'_{xw} \hookrightarrow \cY_{xw}$ onto neighbourhoods of $\cup_{z \in \cH} \cU_{xz} \times \cX_{zw}$ and $\cup_{y \in \cG} \cW_{xy} \times \cU'_{yw}$ respectively; furthermore these can be chosen to be compatible with each other under the inclusions $c$.\par 
            We now remove (the images under the above embeddings) of each $h^{-1}_{xw}(-\infty, 1) \subseteq \tilde{\cP}_{xw}$ and $h'^{-1}_{xw}(-\infty, 1) \subseteq \tilde{\cP}'_{xw}$ from each $\cY_{xw}$ to get a manifold $\cL_{xw}$; explicitly
            $$\cL_{xw} = \cY_{xw} \setminus \left(h_{xw}^{-1}(-\infty, 1) \sqcup h'^{-1}_{xw}(-\infty, 1)\right)$$
            Now each $\cL_{xw}$ has a system of boundary faces given by $\cL_{xx';w}$, $\cL_{x;w'w}$, $\cM_{xw}$ and $\cM'_{xw}$; therefore $\cL$ forms a bordism between $\cM_{xw}$ and $\cM'_{xw}$.
        \end{proof}
      
        \begin{lem}\label{composition recognition}
            Let $\cW: \cF \to \cG$, $\cV: \cG \to \cH$ and $\cT: \cF \to \cH$ be morphisms. Suppose we are given compact manifolds with faces $\cR_{xz}$ for all $x$ in $\cF$ and $z$ in $\cH$, of dimension $|x|-|z|$, with systems of faces given by
            $$\cW_{x_0 \ldots x_i; y_0} \times \cG_{y_0 \ldots y_j} \times \cH_{y_j; z_0 \ldots z_k}$$
            and
            $$\cT_{x_0 \ldots x_i; z_0 \ldots z_k}$$
            and
            $$\cR_{x_0 \ldots x_i; z_0 \ldots z_k}$$
            for $x_0 = x, \ldots, x_i$ in $\cF$, $y_0, \ldots, y_j$ in $\cG$ and $z_0, \ldots, z_k=z$ in $\cH$.\par 
            Then $\cT$ is a representative of $\cV \circ \cW$.
            
        \end{lem}
        \begin{proof}
            Compatible collar neighbourhoods define compatible embeddings 
            $$\tilde{\cQ}_{xz}(\cW, \cV) \hookrightarrow \cR_{xz}$$
            for all $x$ in $\cF$ and $z$ in $\cH$.
            We choose framed functions $(f_{xz}, s_{xz})$ on each $\tilde{\cQ}_{xz}(\cW,\cV)$, compatible with each other, along with compatible extensions $g_{xz}$ to functions $\cR_{xz} \to \bR$ which are $>1$ outside each $\tilde{\cQ}_{xz}(\cW,\cV)$.\par 
            Then the $g_{xz}^{-1}[p, \infty)$ forms a bordism of morphisms between $\cV \circ \cW$ and $\cT$.
        \end{proof}
        We have constructed (and showed well-definedness of):
        \begin{defn}
            The \emph{(homotopy) category of flow categories}, $\Flow$, has objects flow categories, morphisms $[\cdot, \cdot]$, and composition given as above. This is a (non-unital) category enriched in abelian groups.\par 
            By Example \ref{ex:pt unor}, $\Flow$ is in fact enriched in modules over $\Omega_*$.
        \end{defn}

        \begin{rmk}
            We do not require $\Flow$ to be unital and so we do not prove that it is, but we expect it to be.
        \end{rmk}
        \begin{conj}[Abouzaid-Blumberg] \label{Perf MO Conj}
            $\Flow$ is equivalent to the homotopy category of perfect $MO$-modules. 
        \end{conj}

        \subsection{Abstract index bundles}\label{sec:abstr ind bun}
           Before defining stable framings on flow categories, we must first introduce abstract index bundles. In this subsection, we associate abstract index bundles to all of the manifolds we encountered in the previous section, along with isomorphisms between them compatible with the inclusions of boundary strata. These will be required to satisfy appropriate associativity conditions.
            \begin{defn}
                Let $\cM_{xy}$ be a manifold arising in the previous section (for example, $\cM$ could be a flow category and $x, y$ objects in it, or $\cM: \cF \to \cG$ could be a morphism, with $x \in \cF$ or $y \in \cG$). Its \emph{abstract index bundle} is the vector bundle:
                $$I^{\cM}_{xy} := T\cM_{xy} \oplus \bR \tau_y^{\cM} $$
                over $\cM$. Here $\tau^\cM_y$ is a formal generator of the real line $\bR \tau^\cM_y$.\par 
                When it is clear, we will sometimes drop the superscript $\cM$ from the notation. \par 
                We call a boundary face of $\cM_{xy}$ equipped with a product decomposition  a \emph{broken boundary face} (for example this includes all boundary faces of all manifolds associated to a flow category or a morphism of flow categories). However not all boundary faces of the previous section are of this form (for example, if $\cR$ is a bordism between two morphisms $\cW, \cW': \cF \to \cG$, then the boundary faces $\cW_{xy}$ and $\cW'_{xy}$ of $\cR_{xy}$ aren't broken boundary faces). The other boundary faces we label as either \emph{incoming} or \emph{outgoing} boundary faces.
            \end{defn}

                \begin{rmk}
                    Note that the notions of broken/unbroken and incoming/outgoing boundary faces are entirely a case of labelling and are not an intrinsic invariant of the manifolds in question.
                \end{rmk}
                
            Note that in all cases, if two boundary faces intersect, at least one of them is broken.\par 
            There were only two families of unbroken boundary faces we encountered in Section \ref{sec:unor flow cat}. The first was in the boundary faces of a bordism $\cR$ between morphisms $\cW$ and $\cW'$; we declare the boundary faces $\cW_{xy}$ of $\cR_{xy}$ to be incoming and the boundary faces $\cW'_{xy}$ of $\cR_{xy}$ to be outgoing. The second was in the boundary faces of the glued manifold $\tilde{\cS}(\cR, \cV)_{xz}$, considered in Lemma \ref{composition is well-defined} for $\cV: \cG \to \cH$ another morphism; in this case, the boundary faces $\tilde{\cQ}(\cW, \cV)_{xz}$ are declared to be incoming and the boundary faces $\tilde{\cQ}(\cW', \cV)_{xz}$ are declared to be outgoing.
            \begin{defn}\label{def:abstr ind}
                Let $\cA_{xy}$, $\cB_{yz}$ and $\cM_{xz}$ be manifolds arising in the previous section along with an inclusion of a boundary face $c: \cA_{xy} \times \cB_{yz} \to \cM_{xz}$ (for example, $\cA$, $\cB$ and $\cM$ could all be morphisms in some fixed flow category, with objects $x,y,z$, with $c$ the composition map) (for the purposes of this definition, $x$, $y$ and $z$ are entirely decorative). Assume they are equipped with systems of collars with which $c$ is compatible, along with compatible Riemannian metrics respecting the splittings on all collar neighbourhoods. Let $\nu_y^{\cA\times \cB}$ be the inwards pointing normal along this boundary face which is orthogonal to it. We define an isomorphism of vector bundles
                $$\psi=\psi_{xyz}^\cM: I^\cA_{xy} \oplus I^\cB_{yz} \to I^\cM_{xz}$$
                over the boundary face $\cA_{xy} \times \cB_{yz}$, to send $\tau_y^\cA$ to $\nu_y$ and $\tau_z^\cB$ to $\tau_z^\cM$, and to be given by $dc$ on the other factors.\par 
                Now let $c: \cD_{xy} \to \cM_{xy}$ be the inclusion of an unbroken boundary face. Again assume they are equipped with compatible systems of collars and Riemannian metrics, and let $\nu^\cD$ be the inwards pointing normal along this boundary face which is orthogonal to it. We define isomorphisms of vector bundles
                $$\psi=\psi^\cD_{xy}: I^\cD_{xy} \oplus \bR \sigma  \to I^\cM_{xy}$$
                where $\sigma = \sigma^\cD_{xy}$ is a formal generator of the real line $\bR \sigma$, to send $\sigma$ to $+\nu^\cD$ if $\cD$ is an incoming boundary face, and $-\nu^\cD$ if $\cD$ is an outgoing boundary face, and given by $dc$ on the other factors.
                We call all of the above maps $\psi$ \emph{abstract gluing isomorphisms}.
            \end{defn}
            The subscripts are not a necessary part of this definition, but they are present in all cases of interest, so we include them for clarity.
            \begin{prop}
                The $\psi$ are associative over codimension two faces, in the following sense.\par 
                First we consider a codimension two face which is the intersection of two broken boundary faces. This means there are inclusions (all called $c$) of boundary faces $\cA_{xy} \times \cB_{yz} \to \cM_{xz}$, $\cB_{yz} \times \cC_{zw} \to \cN_{yw}$, $\cA_{xy} \times \cN_{yw} \to \cO_{xw}$, and $\cM_{xz} \times \cC_{zw} \to \cO_{xw}$, such that the following diagram commutes:
                $$\xymatrix{
                    \cA_{xy} \times \cB_{yz} \times \cC_{zw} \ar[r] \ar[d] & 
                    \cA_{xy} \times \cN_{yw} \ar[d] \\
                    \cM_{xz} \times \cC_{zw} \ar[r] &
                    \cO_{xw}
                }
                $$
                Then the following diagram commutes:
                $$\xymatrix{
                    I^{\cA}_{xy} \oplus I^\cB_{yz} \oplus I^\cC_{zw} \ar[r]_-\psi \ar[d]_\psi &
                    I^\cA_{xy} \oplus I^\cN_{yw} \ar[d]_\psi \\
                    I^\cM_{xz} \oplus I^\cC_{zw} \ar[r]_-\psi &
                    I^\cO_{xw}
                }
                $$
                over $\cA_{xy} \times \cB_{yz} \times \cC_{zw}$.\par 
                Secondly, we consider a codimension two face which is the intersection of one broken and one unbroken boundary face. Let $\cA_{xy} \times \cB_{yz} \to \cM_{xz}$ be an inclusion of a broken boundary face, and $\cD_{xz} \to \cM_{xz}$ an unbroken boundary face. Assume that the codimension two face $\cD_{xz} \cap (\cA_{xy} \times \cB_{yz})$ is given by $\cA_{xy} \times \cD'_{yz}$, where $\cD'_{yz} \to \cB_{yz}$ is an unbroken boundary face. Assume also that $\cD$ and $\cD'$ are either both incoming or outgoing boundary faces. Then the following diagram commutes:
                $$\xymatrix{
                    I^\cA_{xy} \oplus I^{\cD'}_{yz} \oplus \bR \sigma \ar[r]_-\psi \ar[d]_\psi &
                    I^\cA_{xy} \oplus I^\cB_{yz} \ar[d]_\psi \\
                    I^\cD_{xz} \oplus \bR \sigma \ar[r]_-\psi &
                    I^\cM_{xz}
                }$$
                There is a similar commutative diagram in the case that $\cD_{xz} \cap (\cA_{xy} \times \cB_{yz})$ is given by $\cA_{xy} \times \cD_{yz}''$, where $\cD_{yz}'' \to \cB_{yz}$ is an unbroken boundary face.
            \end{prop}
            \begin{proof}
                In the first case, both ways around the diagram send $\tau_y$ to $\nu_y$, $\tau_z$ to $\nu_y$ and $\nu_z$ to $\tau_w$, and are given by $dc$ on the other factors.\par 
                In the second case, both ways around the diagram send $\tau_y$ to $\nu_y$, $\tau_z$ to $\tau_z$ and $\sigma$ to $\pm \nu^\cD$, where the $\pm$ depends only on whether $\cD$ and $\cD'$ are incoming or outgoing boundary faces.
            \end{proof}
        \subsection{Stable framings}\label{Sec:framings}
            We now incorporate stable framings into the category of flow categories, to obtain a category of framed flow categories.
            \begin{defn}
                A \emph{(stable) framing} on a flow category $\cF$ consists of stable vector spaces $V_x$ for each $x \in \cF$, along with stable isomorphisms
                $$st=st^\cF_{xx'}: V_x \oplus I^\cF_{xx'} \to V_{x'}$$
                with stabilising bundle $\bT_{xx'}^\cF$ for all $x, x' \in \cF$, along with embeddings
                $$\iota_{xx'x''}^\cF: \bT_{xx'} \oplus V_{x'}^- \oplus \bT_{x'x''} \hookrightarrow \bT_{xx''}$$
                over each $\cF_{xx'} \times \cF_{x'x''}$, which are associative in the sense that the following diagram commutes:
                \begin{equation}\label{eq: emb comp}
                    \xymatrix{
                        \bT_{xx'} \oplus V^-_{x'} \oplus \bT_{x'x''} \oplus V^-_{x''} \oplus \bT_{x''x'''} \ar[r] \ar[d] &
                        \bT_{xx''} \oplus V^-_{x''} \oplus \bT_{x''x'''} \ar[d] \\
                        \bT_{xx'} \oplus V^-_{x'} \oplus \bT_{x'x'''} \ar[r] &
                        \bT_{xx'''}
                    }
                \end{equation}
                
                For all $x,x',x'' \in \cF$, consider the following diagram of stable isomorphisms:
                \begin{equation} \label{eq: fram comp flow} 
                    \xymatrix{
                        V_x \oplus I^\cF_{xx'} \oplus I^\cF_{x'x''}  \ar[r]_-{st^\cF} \ar[d]_{\psi} &
                        V_{x'} \oplus I^\cF_{x'x''} \ar[d]_{st^\cF} \\
                        V_{x} \oplus I^\cF_{xx''} \ar[r]_-{st^\cF} &
                        V_{x''}
                    }
                \end{equation}
                Composition both ways round do not have the same stabilising bundle: the composition right then down has stabilising bundle $\bT_{xx'} \oplus V_{x'}^- \oplus \bT_{x'x''}$ whereas the composition down then right has stabilising bundle $\bT_{xx''}$. However we are given an embedding $\bT_{xx'} \oplus V^-_{x'} \oplus \bT_{x'x''} \to \bT_{xx'}$. We require that after extending the stable isomorphism given by composition right then down along this embedding, that this diagram commutes. 
            \end{defn}  
            We will often drop the sub- and/or super-scripts from $st$, $\bT$ and $\iota$ to avoid clutter, if unambiguous from context. Similarly, we will write, for example, $\bT_{xx'x''}$ as shorthand for $\bT_{xx'} \oplus V^-_{x'} \oplus \bT_{x'x''}$, when necessary providing superscripts to indicate which stabilising bundles this refers to.
            \begin{example}
                Given manifolds $\cA_{xy}$ and $\cB_{yz}$, stable vector spaces $V_x,V_y,V_z$, integers $a, b \geq 0$, and stable isomorphisms $st^\cA: V_x \oplus I^\cA_{xy}  \to \bR^a \oplus V_y $ and $st^\cB: V_y \oplus I^\cB_{yz} \to \bR^b \oplus  V_z $, there is an induced stable isomorphism
                $$st^\cB \circ st^\cA: V_x \oplus I^\cA_{xy} \oplus I^\cB_{yz} \to  \bR^a \oplus \bR^b \oplus V_z \to  \bR^{a+b} \oplus V_z$$
                where the final map identifies $\bR^a \oplus \bR^b$ with $\bR^{a+b}$ preserving the orders of all entries.\par 
                Similarly, if $\cA$ is an unbroken boundary face of $\cD$ and given a stable isomorphism 
                $$dt^\cD_{xy}: V_x \oplus I^\cD_{xy} \to \bR^a \oplus V_y$$
                then we may consider the stable isomorphism with stabilising bundle $\bT^\cA_{xy} \oplus \bR \sigma_{xy}^\cA$:
                $$st^\cA_{xy}:  V_x \oplus I_{xy}^\cA  \to \bR^a\oplus V_y $$
                obtained from $\psi$ and $st^\cD$, sending $\sigma$ to $\sigma$.\par 
                Note that in either case, we do not switch the order of the factors of $\bR$. 
            \end{example}
            \begin{defn}
                Let $\cM_{xz}$ be a manifold with faces, equipped with a stable isomorphism with stabilising bundle $\bT^\cA_{xz}$ 
                $$st^\cM_{xz}: V_x \oplus I^\cM_{xz}  \to \bR^m \oplus V_z $$
                Assume that for each broken boundary face $\cA_{xy} \times \cB_{yz} \to \cM$, we are given stable isomorphisms with stabilising bundles $\bT^\cA_{xy}$ and $\bT^\cB_{yz}$ respectively
                $$st^\cA_{xy}: V_x \oplus I^\cA_{xy}  \to \bR^a \oplus V_y  $$
                and
                $$st^\cB_{yz}: V_y \oplus I^\cB_{yz} \to  \bR^b \oplus V_z  $$
                such that $a+b=m$, and also that for each unbroken boundary face $\cD_{xz} \to \cM_{xz}$, we are given a stable isomorphism with stabilising bundle $\bT^\cD_{xz}$
                $$st^\cD_{xz}: V_x \oplus I^\cD_{xz}  \to \bR^d \oplus V_z $$
                with $d=m-1$. Assume we are also given embeddings 
                $$\iota_{xyz}: \bT^\cA_{xy} \oplus V_y^- \oplus \bT^\cB_{yz} \hookrightarrow \bT^\cM_{xz}$$
                over each each broken boundary face $\cA_{xy} \times \cB_{yz}$, and
                $$\iota^\bD_{xz}: \bT_{xz}^\cD \hookrightarrow \bT_{xz}^\cM$$
                over each unbroken boundary face $\cD_{xz}$.
                We say that these form a \emph{system of framings} for $\cM_{xz}$ if over each boundary face, these stable framings and embeddings are compatible with each other via the abstract gluing isomorphisms, in the following sense.\par 
                Consider the following diagrams: over each broken boundary face $\cA_{xy} \times \cB_{yz} \to \cM_{xz}$:
                \begin{equation} \tag{$\star$}\label{eq:broken comp}
                \xymatrix{
                    V_x \oplus I^\cA_{xy} \oplus I^\cB_{yz}\ar[d]_\psi \ar[dr]^{st^\cB \circ st^\cA} & 
                    \\
                    V_x \oplus I^\cM_{xz}  \ar[r]_{st^\cM} &
                    \bR^m \oplus V_z 
                }
                \end{equation}
                and over each unbroken boundary face $\cD_{xz} \to \cM_{xz}$, viewing $\bR \sigma^\cD_{xz} = $ the final copy of $\bR$ in $\bR^m$ as the stabilising bundle:
                \[\tag{$\dag$}\label{eq:unbroken comp}\xymatrix{
                    V_x \oplus I^\cD_{xz} \oplus \bR \sigma^\cD_{xz}
                    \ar[dr]^{st^\cD} \ar[d]_\psi & 
                    \\
                    V_x \oplus I^\cM_{xz}\ar[r]_{st^\cM} &
                    \bR^m \oplus V_z  
                }
                \]
                We require that (\ref{eq:broken comp}) and (\ref{eq:unbroken comp}) both commute, after extending the stable isomorphism along the diagonal along the appropriate embedding $\iota$.
            \end{defn}
            \begin{example}
                Let $\cF$ be a framed flow category. Then each $st^\cF_{xx''}$, along with the induced framings $st^\cF_{xx'} \circ st^\cF_{x'x''}$ on each boundary face and the embeddings $\iota_{xx'x''}$, form a system of framings for each $\cF_{xx''}$.
            \end{example} 
            \begin{defn}\label{def:farm}
                    We define framings for every type of object introduced in Section \ref{sec:unor flow cat}.
                     \begin{enumerate} 
                    \item Let $\cW: \cF \to \cG$ be a morphism of flow categories, and assume that $\cF$ and $\cG$ are framed. A \emph{framing} on $\cW$ consists of stable isomorphisms 
                    $$st^\cW_{xy}: V_x \oplus I^\cW_{xy}  \to \bR \oplus V_y $$
                    with stabilising bundles $\bT^\cW_{xy}$ for each $x \in \cF$ and $y \in \cG$, along with, for each boundary face of each $\cW_{xy}$, an embedding of the stabilising bundle of the induced framing on that face into $\bT^\cW_{xy}$.\par 
                    We require that this data, combined with the framings for $\cF$, $\cG$ and each lower dimensional $\cW_{x'y'}$, form systems of framings for each $\cW_{xy}$.
                    \item Let $\cR$ be a bordism from $\cW$ to $\cW'$, where $\cW, \cW': \cF \to \cG$ are framed morphisms between framed flow categories. A \emph{framing} on $\cR$ consists of stable isomorphisms
                    $$st^\cR_{xy}:V_x \oplus I^\cR_{xy}  \to \bR^2 \oplus V_y  $$
                    with stabilising bundles $\bT^\cR_{xy}$ along with embeddings of the stabilising bundle over each boundary face of each $\cR_{xy}$ into $\bT^\cR_{xy}$, such that these, combined with the framings for $\cF, \cG, \cW$, $\cW'$ and each lower dimensional $\cR_{x'y'}$, form systems of framings for each $\cR_{xy}$.
                    \item Let $\cW: \cF \to \cG$ and $\cV: \cG \to \cH$ be framed morphisms between framed flow categories. We form the manifolds $\tilde{\cQ}_{xz} := \tilde{\cQ}_{xz}(\cW, \cV)$ as in Section \ref{sec:unor flow cat}. A \emph{framing} on the collection of these manifolds $\tilde{\cQ}$ consists of stable isomorphisms 
                    $$st^{\tilde{\cQ}}_{xz}: V_x \oplus I^{\tilde{\cQ}}_{xz} \to \bR^2 \oplus V_z  $$
                    with stabilising bundles $\bT^{\tilde \cQ}_{xz}$ along with embeddings of the stabilising bundle over each boundary face of $\tilde{\cQ}_{xz}$ into $\bT^{\tilde \cQ}_{xz}$ such that, combined with the framings for $\cF$, $\cG$, $\cW$, $\cV$ and each lower dimensional $\tilde{\cQ}_{x'y'}$, form systems of framings for each $\tilde{\cQ}_{xz}$.
                    \item Let $\cW, \cW': \cF \to \cG$ and $\cV: \cG \to \cH$ be framed morphisms between framed flow categories. Assume that we are given framings on $\tilde{\cQ} := \tilde{\cQ}(\cW, \cV)$ and $\tilde{\cQ}':= \tilde{\cQ}(\cW', \cV)$. We form the manifolds $\tilde{\cS}_{xz} := \tilde{\cS}_{xz}(\cR, \cV)$ as in Lemma \ref{composition is well-defined}. A \emph{framing} on the collection of these manifolds $\tilde{\cS}$ consists of stable isomorphisms
                    $$st^{\tilde{\cS}}_{xz}: V_x \oplus I^{\tilde{\cS}}_{xz} \to \bR^3 \oplus V_z  $$
                    with stabilising bundles $\bT^{\tilde \cS}_{xz}$ along with embeddings of the stabilising bundle over each boundary face of $\tilde{\cS}_{xz}$ into $\bT^{\tilde \cS}_{xz}$ such that, when combined with all other framings in sight, form systems of framings for each manifold in sight.
                \end{enumerate}
            \end{defn}
            For expository purposes, we spell out what this means more explicitly in the simplest cases.
            \begin{example}\label{ex:fram mor}
                Let $\cW: \cF \to \cG$ be a morphism between two framed flow categories. Note all boundary faces of each $\cW_{xy}$ are broken. Given stable isomorphisms 
                $$st^\cW_{xy}: V_x \oplus I^\cW_{xy}  \to \bR \oplus V_y $$
                with stabilising bundles $\bT^\cW_{xy}$, the compatibility requirements for them to form a framing on $\cW$ is that the following.\par 
                We require diagrams analagous to (\ref{eq: emb comp}) to commute over each $\cW_{xx'x'';y}$, $\cW_{xx';y'y}$ and $\cW_{x;y''y'y}$. Similarly to (\ref{eq: fram comp flow}) over each $\cW_{xx';y}$ and $\cW_{x;y'y}$ respectively, where we extend the compositions going right then down along the appropriate embedding of stabilising bundles:
                $$\xymatrix{
                    V_x \oplus I^\cF_{xx'} \oplus I^\cW_{x'y} \ar[r]_-{st^\cF} \ar[d]_\psi &
                    V_{x'} \oplus I^\cW_{x'y} 
                    \ar[d]_{st^\cW} \\
                    V_x \oplus I^\cW_{xy} \ar[r]_{st^\cW} &
                    \bR \oplus V_y 
                }\xymatrix{
                    V_x \oplus I^\cW_{xy'} \oplus I^\cG_{y'y} \ar[r]_{st^\cW} \ar[d]_{\psi} & \bR \oplus V_{y'} \oplus 
                    I^\cG_{y'y} \ar[d]_{st^\cG} \\
                    V_x \oplus I^\cW_{xy}  \ar[r]_{st^\cW} &
                    \bR \oplus V_y
                }
                $$
            \end{example}
            \begin{defn}
                Let $\cW: \cF \to \cG$ be a framed morphism of framed flow categories. We define $\overline{\cW}$ to be the morphism $\cW$, equipped with the framing 
                $$st^{\overline{\cW}}: V_x \oplus I^\cW_{xy}  \to \bR \oplus V_y$$
                with stabilising bundles $\bT^{\overline \cW}_{xy} := \bT^\cW_{xy} \oplus \bR$, given by $st^\cW$ postcomposed with the map $-1: \bR \to \bR$.
            \end{defn}
            \begin{example}
                Let $\cR$ be a bordism between two framed morphisms $\cW, \cW': \cF \to \cG$ of framed flow categories. Given stable isomorphisms
                $$st^\cR_{xy}: V_x \oplus I^\cR_{xy} \to \bR \oplus V_y$$
                the compatibility requirement for them to form a framing on $\cR$ is that the following diagrams commute. \par 
                Over each $\cR_{xx';y}$ and $\cR_{x;y'y}$, diagrams similar to those in Example \ref{ex:fram mor} commute (note that some factors of $\bR$ must be replaced with factors of $\bR^2$ though). \par 
                Furthermore, over each $\cW_{xy}$:
                $$\xymatrix{
                    V_x \oplus I^\cW_{xy} \oplus \bR \sigma^\cW
                    \ar[dr]^{st^\cW} \ar[d]_\psi &
                    \\
                    V_x \oplus I^\cR_{xy}  \ar[r]_{st^\cR} &
                    \bR^2 \oplus V_y
                }
                $$
                Note that the diagonal arrow sends the line $\bR \sigma^\cW$ to the first factor in $\bR^2$.\par 
                Over each $\cV_{xy}$:
                $$\xymatrix{
                    V_x \oplus I^\cV_{xy} \oplus \bR \sigma^\cV\ar[dr]^{st^\cV} \ar[d]_\psi &
                    \\
                    V_x \oplus I^\cR_{xy} \ar[r]_{st^\cR} &
                    \bR^2 \oplus V_y
                }
                $$
                The boundary faces $\cW_{xy}$ and $\cV_{xy}$ are the unbroken boundary faces of $\cR_{xy}$.
            \end{example}
            \begin{example}\label{ex:fram inv}
                Let $\cW: \cF \to \cG$ be a framed morphism between framed flow categories, and $\cR$ the bordism from $\cW$ to itself constructed in Lemma \ref{lem:unor inv}. We equip $\cR$ with a framing as follows.\par 
                We write $\cR_{xy} = \cW_{xy} \times [0,1]_s$, with $s$ the coordinate on the interval. Then define
                $$st^\cR_{xy}: V_x \oplus I^\cR_{xy}  = V_x \oplus I^\cW_{xy} \oplus \bR \partial_s \xrightarrow{st^{\cW}} \bR^2 \oplus V_y$$
                where $\partial_s$ is sent to $(0, 1) \in \bR^2$. \par 
                It is straightforward to check this is a stable framing; we spell out the details over the unbroken outgoing boundary face $\cW_{xy} \times \{1\}$. Note that the inwards normal along this face is $-\partial_s$. \par 
                We must check that the following diagram commutes over $\cW_{xy} \times \{1\}$:
                $$\xymatrix{
                    V_x \oplus  I^\cW_{xy} \oplus \bR \sigma \ar[d]_\psi \ar[dr]^{st^\cW} &
                    \\
                    V_x \oplus I^\cR_{xy} \ar[r]_{st^\cR} &
                    \bR^2 \oplus V_y
                }
                $$
                Both ways around the diagram send $\sigma$ to $(0, 1) \in \bR^2$, and agree on the other factors by construction of $st^\cR$.
            \end{example}
            \begin{example}\label{ex:hom fram bord}
                Homotopic framings on the same morphism are framed bordant.
            \end{example}
            \begin{example}\label{ex:hom stab fram bord}
                Let $\cW: \cF \to \cG$ be a framed morphism, and let $E$ be some vector space. We define a framed morphism $\cW^E: \cF \to \cG$ as follows: the underlying unframed morphism of $\cW^E$ is defined to be the same as that of $\cW$. We define stabilising bundles $\bT^{\cW^E}_{xy} := \bT^\cW_{xy} \oplus E$, and we define the stable framings of $\cW^E$ to be those of $\cW$ extended under the natural inclusions $\bT^\cW_{xy} \to \bT^{\cW^E}_{xy}$. Then $\cW$ and $\cW^E$ are framed bordant.\par 
                More generally, if $E_{xy} \to \cW_{xy}$ are vector bundles and we are given embeddings $E_{x'y} \to E_{xy}$ and $E_{xy'} \to E_{xy}$ over each $\cW_{xx';y}$ and $\cW_{x;y'y}$ respectively which are suitably associative, we obtain another framed morphism $\cW^E$ with stabilising bundles $\bT^{\cW^E}_{xy} := \cW_{xy} \oplus E_{xy}$, given by extending the framings for $\cW$ along the natural inclusions $\bT^\cW_{xy} \to \bT^{\cW^E}_{xy}$. Once again, $\cW^E$ is framed bordant to $\cW$.\par 
                We call the data of the $E_{xy}$ above a \emph{system of vector bundles} on $\cW$, and a framed morphism of the form $\cW^E$ a \emph{stabilisation} of $\cW$ along $E$. One can similarly define systems of vector bundles, and stabilisations along them, for the other objects defined in Definition \ref{def:farm}.
            \end{example}
            \begin{defn}\label{def:fram mor equiv}
                Let $\cF$ and $\cG$ be framed flow categories. We define $[\cF, \cG]^{fr}$ to be the set of framed morphisms $\cF \to \cG$, modulo the equivalence relation generated by
                \begin{enumerate}
                    \item $\cW \sim \cV$ if $\cW$, $\cV$ are framed bordant.
                    \item $\cW \sim \cV$ if $\cW$ and $\cV$ are diffeomorphic, via a diffeomorphism compatible with the framings.
                \end{enumerate}
            \end{defn}
            Similarly to $[\cF, \cG]$, it is clear that $[\cF, \cG]^{fr}$ forms a commutative monoid under disjoint union, with unit given by the empty morphism.
            \begin{lem}
                Inverses exist in $[\cF, \cG]^{fr}$, i.e. it's a group.
            \end{lem}
            \begin{proof}
                Let $\cW^\bR$ be as constructed in Example \ref{ex:hom stab fram bord}; as noted before, this is framed bordant to $\cW$. Then the framed bordism constructed in Example \ref{ex:fram inv} can alternatively be viewed as a framed nullbordism of $\cW^\bR \sqcup \overline{\cW}$, implying that the inverse of $\cW$ is $\overline{\cW}$.
            \end{proof}
            \begin{lem}\label{lem:fram Q}
                Let $\cW, \cW': \cF \to \cG$ and $\cV: \cG \to \cH$ be framed morphisms of framed flow categories. Let $\tilde{\cQ} := \tilde{\cQ}(\cW, \cV)$, let $\tilde{\cQ}' := \tilde{\cQ}(\cW', \cV)$ and let $\tilde{\cS} := \tilde{\cS}(\cR, \cV)$.\par 
                Then $\tilde{\cQ}, \tilde{\cQ}'$ and $\tilde{\cS}$ admit framings, canonical up to stabilisation and isomorphism of framings.
            \end{lem}
            \begin{proof}
                We first frame $\tilde{\cQ}$; framing $\tilde{\cQ}'$ is identical. We must first construct the stabilising bundle.\par
                Inductively over the index difference $|x|-|z|$, we inductively choose vector bundles $\bT^{\tilde Q}_{xz}$ over each $\tilde\cQ_{xz}$, along with embeddings $\iota: \bT^{\cW\cV}_{xyz}, \bT^{\cF\tilde\cQ}_{xx'z}, \bT_{xz'z}^{\tilde\cQ \cH} \to \bT^{\tilde Q}_{xz}$, which are compatible on overlaps in the sense that the following diagram commutes for each $x,x' \in \cF$, $y,y' \in \cG$ and $z,z' \in \cH$:
                \begin{equation*}
                    \xymatrix{
                        \bT^{\cW\cG\cV}_{xyy'z} \ar[r] \ar[d] &
                        \bT^{\cW\cV}_{xyz} \ar[d] \\
                        \bT^{\cW\cV}_{xy'z} \ar[r] &
                        \bT^{\tilde Q}_{xz}
                    }
                \end{equation*}
                Note that the choice of these bundles and embeddings is unique up to stabilisation and isomorphism.\par 
                We define stable isomorphisms
                $$st^{\tilde{\cQ}}_{xz}: V_x \oplus I^{\tilde{\cQ}}_{xz} \to \bR^2 \oplus V_z$$
                with stabilising bundles $\bT^{\tilde \cQ}_{xz}$ as follows. On each $\cW_{xy} \times \cV_{yz} \times [0, \varepsilon)_s \subseteq \tilde{\cQ}_{xz}$, we define it as the composition
                $$st^{\tilde{\cQ}}_{xz}: V_x \oplus I^{\tilde{\cQ}}_{xz} \xrightarrow{\psi^{-1}} V_x \oplus I^\cW_{xy} \oplus I^\cV_{yz} \xrightarrow{st^\cV \circ st^\cW} \bR^2 \oplus V_z$$
                extended along the embedding $\iota: \bT^\cW_{xy} \oplus V_y^- \oplus \bT^\cV_{yz} \to \bT^{\tilde \cQ}_{xz}$, noting that the first copy of $\bR$ comes from $st^\cW$ and the second from $st^\cV$. Note that $\psi$ isn't a priori defined over all of this open subset but we can extend it from the boundary to the rest in a canonical way, since $\partial_s$ is naturally identified with $\nu_y$. \par 
                We must check these agree on overlaps for this to define a stable framing over the entirety of $\tilde{\cQ}_{xz}$. Noting that on this subset, we have that $T\tilde{\cQ} = \bR \partial_s \oplus T\cW_{xy} \oplus T\cV_{yz}$, we can define an isomorphism of vector bundles
                $$\bar{\psi}: \bR \partial_s \oplus \bR \tau_z \oplus T\cW_{xy} \oplus T\cV_{yz} \to I^\cW_{xy} \oplus I^\cV_{yz}$$
                by sending $\partial_s$ to $\tau_y$ and $\tau_z$ to $\tau_z$. Then noting that we can identify the domain with $I^{\tilde{\cQ}}$, the following diagram commutes: 
                $$\xymatrix{
                    I^{\tilde{\cQ}}_{xz} \ar[rr]_-= \ar[drr]_{\psi^{-1}} &&
                    \bR \partial_s \oplus \bR \tau_z \oplus T\cW_{xy} \oplus T\cV_{yz} \ar[d]_{\bar{\psi}} \\
                    && I^\cW_{xy} \oplus I^\cV_{yz}
                }
                $$
                Now let $y, y' \in \cG$, and consider the maps
                $$\alpha: \cW_{xy} \times \cG_{yy'} \times \cV_{y'z} \times [0, \varepsilon)^2_{u,v} \to \cW_{xy} \times \cV_{yz} \times [0, \varepsilon)_s$$
                and $$\beta: \cW_{xy} \times \cG_{yy'} \times \cV_{y'z} \times [0, \varepsilon)^2_{u,v} \to \cW_{xy'} \times \cV_{y'z} \times [0, \varepsilon)_s$$
                which we used as the gluing maps in the construction of $\tilde{\cQ}$. These were defined to send $(a, b, c, u, v)$ to $(a, \cC(b, c; u), v)$ and $(\cC(a, b; v), c, u)$ respectively. The compatibility we need is that the outside part of the following diagram commutes. Note that the top entry is $V_x \oplus \textrm{(the tangent space of the overlap)}\oplus \bR \tau_z $.
                \[
                \centerline{\xymatrix{
                    &V_x \oplus \bR \partial_u \oplus \bR\partial_v \oplus T\cW_{xy} \oplus T\cG_{yy'} \oplus T\cV_{y'z} \oplus \bR \tau_z  \ar[dl]_{d\alpha} \ar[dr]_{d\beta} \ar[d]_= & \\
                    V_x \oplus T\cW_{xy} \oplus \bR \partial_s \oplus T\cV_{yz} \oplus \bR \tau_z  \ar[d]_{\bar{\psi}} &
                    V_x \oplus I^{\tilde{\cQ}}_{xz} &
                    V_x \oplus T\cW_{xy'} \oplus \bR \partial_s \oplus T\cV_{y'z} \oplus \bR \tau_z \ar[d]_{\bar{\psi}}\\
                    V_x \oplus I^\cW_{xy} \oplus I^\cV_{yz} \ar[dr]_{st^\cV \circ st^\cW} \ar[ur]_-\psi &
                    V_x \oplus I^\cW_{xy} \oplus I^\cG_{yy'} \oplus I^\cV_{y'z} \ar[l]_\psi \ar[r]_-\psi \ar[d]_{st^\cV \circ st^\cG \circ st^\cW} &
                    V_x \oplus I^\cW_{xy'} \oplus I^\cV_{y'z} \ar[lu]_\psi \ar[ld]_{st^\cV \circ st^\cW}\\
                    & \bR^2 \oplus V_z &
                }
                }
                \]
                where we extend the three stable framings in the two bottom triangles along the embeddings from their original stabilising bundles into $\bT^{\tilde \cQ}_{xz}$. The two bottom triangles commute by the compatibility of the $st^\cW$, $st^\cG$, $\psi$ and $\iota$. The middle diamond commutes by associativity of the maps $\psi$. The top two squares together we can check commutativity of directly: all three ways from the top entry to the entry in the diagram below it send $\partial_u$ to $\nu_y$, $\partial_v$ to $\nu_{y'}$, $\tau_z$ to $\tau_z$, don't do anything to $V_x$ and are given by $dc$ on the other factors.\par 
                This implies that the $st^{\tilde{\cQ}}$, which we defined on an open cover of $\tilde{\cQ}_{xz}$, actually glue to give a well-defined stable isomorphism over the entirety of $\tilde{\cQ}_{xz}$.\par 
                We next must check that the $st^{\tilde{\cQ}}$ form an actual framing of $\tilde{\cQ}$.  Commutativity of \ref{eq:broken comp} over the boundary face $\tilde{\cQ}_{xx';z}$ follows on from commutativity of the following diagram; noting that going right and then down corresponds to applying $st^\cF$ then $st^{\tilde{\cQ}}$ whereas going down then right corresponds to applying $\psi$ then $st^{\tilde{\cQ}}$:
                \[\xymatrix{
                    V_x \oplus I^\cF_{xx'} \oplus I^{\tilde{\cQ}}_{x'z}\ar[d]_\psi &
                    V_x \oplus I^\cF_{xx'} \oplus I^\cW_{x'y} \oplus I^\cV_{yz} \ar[l]_\psi \ar[ddl]_\psi \ar[dd]^{st^\cV \circ st^\cW \circ st^\cF} \\
                    V_x \oplus I^{\tilde{\cQ}}_{xz} & \\
                    V_x \oplus I^\cW_{xy} \oplus I^\cV_{yz} \ar[r]_{st^\cV \circ st^\cW} \ar[u]_\psi &
                    \bR^2 \oplus V_z
                }
                \]
                over the intersection of this face with each open subset of the form $\cW_{xy} \times \cV_{yz} \times [0, \varepsilon)$. As usual, maps of the form $st \circ st$, the first (i.e. rightmost) copy of $st$ applied provides the first copy of $\bR$ and the second (i.e. leftmost) $st$ provides the second copy of $\bR$, and we extend the relevant stable framings here under the natural embeddings of stable framings in order that we can compare them.\par 
                The outer square is \ref{eq:broken comp} over this face; the upper left triangle commutes by compatibility of the maps $\psi$ and the lower right triangle commutes by compatibility of the $\psi$ and $st$.\par 
                Commutativity of \ref{eq:broken comp} over the boundary face $\tilde{\cQ}_{x;z'z}$ is essentially identical; commutativity of \ref{eq:broken comp} over the boundary face $\cW_{xy} \times \cV_{yz} \times \{0\}$ follows by construction of $st^{\tilde{\cQ}}$.\par 
                We frame $\tilde{\cS}$ in exactly the same way: over each $\cR_{xy} \times \cV_{yz} \times [0, \varepsilon)_s$, we first choose stabilising bundles $\bT^{\tilde \cS}_{xz}$ along with appropriate embeddings $\iota$, and define $st^{\tilde{\cS}}_{xz}$ to be the composition
                $$st^{\tilde{\cS}}_{xz}: V_x \oplus I^{\tilde{\cS}}_{xz} \xrightarrow{\psi^{-1}} V_x \oplus I^\cR_{xy} \oplus I^\cV_{yz} \xrightarrow{st^\cV \circ st^\cR} \bR^3 \oplus V_z$$
                where the first two entries of $\bR^3$ are from $st^\cR$ and the final one is from $st^\cV$, similarly extended so the stabilising bundle is $\bT^{\tilde \cS}_{xz}$.\par 
                Similarly to with $\tilde{\cQ}$, this glues together to define $st^{\tilde{\cS}}$ over the the entirety of $\tilde{\cS}_{xz}$. For this to be a framing, we have a compatibility condition over each boundary face; over the boundary faces $\tilde{\cS}_{xx';z}$, $\tilde{\cS}_{x;z'z}$ and $\cR_{xy} \times \cV_{yz} \times \{0\}$ this follows from the same argument as for $\tilde{\cQ}$, over the boundary faces $\tilde{\cQ}_{xz}$ and $\tilde{\cQ}'_{xz}$, \ref{eq:unbroken comp} commutes by compatibility between the framings of $\cR$ and $\tilde{\cW}$, $\tilde{\cW}'$.
            \end{proof}
            \begin{lem}\label{lem:fram comp well-def}
                Let $\cW: \cF \to \cG$ and $\cV: \cG \to \cH$ be framed morphisms of framed flow categories. Then their composition can be framed, and the equivalence class $[\cV \circ \cW]$ defines in $[\cF, \cH]^{fr}$ is well-defined.
            \end{lem}
            \begin{proof}
                Compatibly choose framed functions $(f_{xz}, s_{xz})$ on each $\tilde{\cQ}_{xz}$, let $\tilde{\cQ}^{\leq 1} = f_{xz}^{-1}(-\infty, 1]$ and let $\cU_{xz} = f_{xz}^{-1}\{1\}$. Then $\cU$ is a representative of $\cV \circ \cW$ as an unframed morphism; we use the framing we've constructed on $\tilde{\cQ}$ to frame $\cU$.\par 
                We define stable isomorphisms $st^\cU_{xz}: I^\cU_{xz} \oplus V_z \to \bR \oplus V_x$ with stabilising bundle $\bT^{\tilde \cQ}_{xz} \oplus \bR \xi$ (where $\xi$ is some formal generator) to be the composition:
                $$st^\cU_{xz}:  V_x \oplus I^\cU_{xz}\oplus \bR\xi \xrightarrow{di_{xz} + s_{xz}} V_x \oplus I^{\tilde{\cQ}}_{xz} \xrightarrow{st^{\tilde{\cQ}}} \bR^2 \oplus V_z = \bR \oplus \bR\xi \oplus V_z$$
                Here $i_{xz}: \cU_{xz} \to \tilde{\cQ}_{xz}$ is the natural inclusion map, $s_{xz}$ takes $\bR \xi$ as input, the first map sends $\tau_z$ to $\tau_z$, and the final identification is to make it clear that it is the second copy of $\bR$ in $\bR^2$ that is the stabilising bundle.\par 
                Commutativity of \ref{eq:broken comp} over the appropriate boundary faces follows from similar considerations to before, though now we must use the compatibility of the framed functions with those on the boundary faces. This implies $\cU$ is a framed morphism.\par
                Since the framing on $\tilde\cQ$ was canonical up to stabilisation and isomorphism, it follows from construction that (for fixed choice of framed functions $(f_{xz}, s_{xz})$) the induced framing on $\cU$ is also canonical up to stabilisation and isomorphism.\par 
                Similarly we can frame $\tilde{\cQ}^{\leq 1}$ in an appropriate sense; note commutativity of \ref{eq:unbroken comp} uses the fact that we identify $\bR^2 = \bR \oplus \bR\xi$ rather than the other way around. Proceeding as in the proof of Lemma \ref{composition well-defined} (but now with framings in hand) shows that the class of $\cU$ in $[\cF, \cH]^{fr}$ is independent of choices.
            \end{proof}
            \begin{lem}\label{lem:comp fram}
                Composition of framed morphisms is compatible with the equivalence relation in Definition \ref{def:fram mor equiv}, and induces a well-defined bilinear map
                \[
                \circ: [\cG, \cH]^{fr} \otimes [\cF, \cG]^{fr} \to [\cF, \cH]^{fr}
                \]
            \end{lem}
            \begin{proof}
                Similarly to the proof of Lemma \ref{lem:fram comp well-def}, we choose compatible framed functions on each $\tilde{\cQ}_{xz}$ and $\tilde{\cS}_{xz}$, and proceed as in the proof of Lemma \ref{composition well-defined}, framing the bordism $\cT$ similarly to above, to conclude that composition induces the desired well-defined map; the same argument as in the proof of Lemma \ref{composition well-defined} also shows it is a bilinear map.
            \end{proof}
            Similarly to Lemma \ref{composition is associative}, we have:
            \begin{lem} \label{framed composition is associative}
                The composition map is associative, i.e. the following diagram commutes:
                \[
                    \xymatrix{
                        \left[\cH, \cI\right]^{fr} \otimes \left[\cG, \cH\right]^{fr} \otimes \left[\cF, \cG\right]^{fr} \ar[rr]_-{\circ} \ar[d]_{\circ} && \left[\cG, \cI\right]^{fr} \otimes \left[\cF, \cG\right]^{fr} \ar[d]^{\circ} \\
                        \left[\cH, \cI\right]^{fr} \otimes \left[\cF, \cH\right]^{fr} \ar[rr]_-{\circ} && \left[\cF, \cI\right]^{fr}
                    }
                \]
            \end{lem}
            \begin{proof}
                We prove this by equipping all the manifolds encountered in the proof of Lemma \ref{composition is associative} with (compatible) stable framings; we furthermore assume all of the same notation as in Lemma \ref{composition is associative}. \par 
                Using the framings on $\cF$, $\cG$, $\cH$, $\cI$, $\cW$, $\cV$ and $\cX$, we equip $\tilde{\cP}$, $\tilde{\cP}'$, $\tilde{\cQ}$, $\tilde{\cQ}'$ and $\tilde{\cZ}$ with stable framings, just as in Lemma \ref{lem:fram comp well-def}. \par 
                Using the given choices of framed functions, we equip $\cM$, $\cM'$, $\cU$, $\cU'$, $\cY$, $\cL$ all with stable framings.\par 
                Unfortunately the embeddings from the collar neighbourhood theorem $\cP_{xw}, \cP'_{xw} \hookrightarrow \cY_{xw}$ may not be compatible with the stable framings, so whilst $\cL$ is a bordism between $\cM$ and $\cM'$, the framing we have equipped $\cL$ with might not be compatible with the framings on $\cM$ and $\cM'$, so it might not be a framed bordism between them.\par 
                However note that the framings on each $\tilde{\cP}_{xw}$ agree with those on $\cY_{xw}$ when restricted to the carapace of $\tilde{\cP}_{xw}$. Since each $\tilde{\cP}_{xw}$ deformation retracts to its carapace (which also includes all its boundary), the framings on $\tilde{\cP}_{xw}$ pulled back from $\cY_{xw}$ are homotopic to those constructed on the $\tilde{\cP}_{xw}$ directly; furthermore by inducting on the dimension and working relative to the boundary we may assume that these homotopies are compatible with each other for all $x \in \cF$, $w \in \cI$.\par 
                Since homotopic framings on $\tilde{\cP}_{xw}$ induce homotopic framings on $\cM_{xw}$ (since the construction of this framing works in 1-parameter families of framings on $\tilde{\cP}$), and noting that this discussion applies equally to $\tilde{\cP}'$ and $\cM'$, it follows that $\cL$ is a framed bordism between $\cM$ and $\cM'$, where $\cM$ and $\cM'$ are equipped with framings homotopic to the ones we originally equipped them with.\par 
                Combined with Example \ref{ex:hom fram bord}, it follows that $\cM$ and $\cM'$ are framed bordant.
            \end{proof}
            Similarly to Lemma \ref{composition recognition} (and by the same proof, but incorporating the `deformation to the carapace' argument from above), we have:
            \begin{lem}\label{lem: fram comp detection}
                In the situation of Lemma \ref{composition recognition}, if $\cF$, $\cG$, $\cH$, $\cW$, $\cV$, $\cT$ and $\cR$ are all compatibly framed, then $\cT$ is a representative of $\cV \circ \cW$ in $[\cF, \cH]^{fr}$.
            \end{lem}
            \begin{defn}
                We define the (homotopy) \emph{category of framed flow categories}, $\Flow^{fr}$, to have objects framed flow categories, and morphisms framed morphisms of flow categories. This is a (non-unital) category enriched in abelian groups.
            \end{defn}
            \begin{example}
                Given a framed flow category $\cF$, we can form $\cF[n]$ as a framed flow category for any $n \in \bZ$, by setting $V^{\cF[n]}_x := V^\cF_x \oplus \bR^n$ and $st^{\cF[n]} := st^\cF \oplus \id_{\bR^n}$.
            \end{example}
            \begin{example}
                A morphism $*[i] \to *$ consists of a closed (stably) framed $i$-manifold; a framed bordism between morphisms consists of a (stably) framed bordism. From this, we see that
                \[
                [*[i],*]^{fr} = \Omega^{fr}_i(\ast) = \Omega^{fr}_i
                \]
                In particular, it follows that $\Flow^{fr}$ is enriched over graded modules over $\Omega^{fr}_*$.
            \end{example}
            \begin{conj}[Abouzaid-Blumberg]\label{conj:AB}
                $\Flow^{fr}$ is equivalent to the homotopy category of finite spectra.
            \end{conj}
       
            Though not phrased like this, a functor from the category of framed flow categories to the homotopy category of finite spectra is essentially constructed in \cite{Large}, following \cite{CJS}.
    \section{Algebra over flow categories}\label{sec: flow alg}
        
        \subsection{Bilinear maps of flow categories}\label{sec:bilin}
        \begin{defn}
            Let $\cF$, $\cG$ and $\cH$ be flow categories. A \emph{bilinear map} of flow categories $\cW: \cF \times \cG \rightarrow \cH$ consists of compact smooth manifolds with faces $\cW_{xy;z}$ of dimension $|x|+|y|-|z|$ for each $x$ in $\cF$, $y$ in $\cG$ and $z$ in $\cH$, along with maps
            $$c = c^\cW: \cF_{xu} \times \cW_{uy;z} \rightarrow \cW_{xy;z}$$
            and
            $$c = c^\cW: \cG_{yv} \times \cW_{xv;z} \rightarrow \cH_{xy;z}$$
            and 
            $$c = c^\cW: \cW_{xy;w} \times \cH_{wz} \rightarrow \cW_{xy;z}$$
            which are compatible with each other and $c^\cF$, $c^\cG$ and $c^\cH$, and which define a system of faces
            $$\cF_{x_0 \ldots x_i} \times \cG_{y_0 \ldots y_j} \times \cW_{x_i y_j; z_0} \times \cH_{z_0 \ldots z_k}$$
            for $x_0 = x, \ldots, x_i$ in $\cF$, $y_0 = y, \ldots, y_j$ in $\cG$ and $z_0, \ldots, z_k = z$ in $\cH$. In particular, $\cW_{xy;z}$ has a system of boundary faces $\cF_{xx'} \times \cW_{x'y;z}$, $\cG_{yy'} \times \cW_{xy';z}$ and $\cW_{xy;z'} \times \cH_{z'z}$ for $x' \in \cF$, $y' \in \cG$ and $z' \in \cH$.
        \end{defn}
        \begin{ex}\label{ex: bilin point}
            A bilinear map $*[i] \times *[j] \rightarrow \cH$ consists of the same data as a morphism $*[i+j] \rightarrow \cH$.
        \end{ex}
        \begin{ex}\label{ex: bilin point 2}
            A bilinear map $*[i] \times \cG \to \cH$ consists of the same data as a morphism $\cG[i] \to \cH$.
        \end{ex}
        Let $\cF$, $\cG$ and $\cH$ be flow categories, $\cW: \cF \times \cG \rightarrow \cH$ a bilinear map and $\cA: *[i] \rightarrow \cF$ and $\cB: *[j] \rightarrow \cG$ morphisms.  
        We define a morphism of flow categories 
        $$\cW \circ \cA: \cG[i] \rightarrow \cH$$
        as follows.\par 
        For $y$ in $\cG$ and $z$ in $\cH$, define
        $$\tilde{\cP}_{yz} := \tilde{\cP}_{yz}(\cA,\cW) =  \left(\bigsqcup_{x \in \cF} \cA_{*x} \times \cW_{xy;z} \times [0, \eps)\right)/\sim $$
        where $\sim$ is defined so that $\tilde{\cP}_{yz}$ is the coequaliser of
        $$\bigsqcup_{x, x' \in \cF} \cA_{*x} \times \cF_{xx'} \times \cW_{x'y;z} \times [0, \eps)^2_{u,v} \rightrightarrows \bigsqcup_{x \in \cF} \cA_{*x} \times \cW_{xy;z} \times [0, \eps)_s$$
        where the two maps send $(a, b, c, u, v)$ to $(a, \cC(b, c, u), v)$ and $(\cC(a, b, v), c, u)$ respectively, as in the construction of the composition of morphisms.\par 
        By Lemmas \ref{lem:abstract gluing} and \ref{lem:abstract gluing with faces} (applying these lemmas with $I = ob \cF$ and $J = ob \cG \sqcup ob \cH$), this is a manifold with faces, with a system of faces given by 
        $$\cA_{*;x_0 \ldots x_i} \times \cG_{y_0 \ldots y_j} \times \cW_{x_i y_j; z_0} \times \cH_{z_0 \ldots z_k} \times \{0\}$$
        and
        $$\cG_{y_0 \ldots y_j} \times \tilde{\cP}_{y_j z_0} \times \cH_{z_0 \ldots z_k}$$
        for $x, x_0, \ldots, x_i$ in $\cF$, $y_0 = y, \ldots, y_j$ in $\cG$ and $z_0, \ldots, z_k = z$ in $\cH$. Pick framed functions $(f_{yz}, s_{yz})$ on each $\tilde{\cP}_{yz}$ which are compatible with each other on overlaps, and let $\cV_{yz} := f^{-1}_{yz}\{1\}$ for all $y \in \cG$, $z \in \cH$. These form a morphism of flow categories $\cG[i] \rightarrow \cH$ which we define $\cW \circ \cA$ to be. By the same argument as in the proof of Lemma \ref{composition well-defined}, this defines a well-defined morphism in $\Flow$.\par 
        Given instead $\cB: *[j] \rightarrow \cG$, we can similarly define $\tilde{\cP}_{xz}=\tilde{\cP}_{xz}(\cB, \cW)$ for $x \in \cF$ and $z \in \cH$, and define a morphism of flow categories 
        $$\cW \circ \cB: \cF[j] \rightarrow \cH$$
        \begin{lem}\label{lem:two bilin}
            For all morphisms of flow categories $\cA: *[i] \to \cF$ and $\cB: *[j] \to \cG$, we have that
            $$[(\cW \circ \cA) \circ \cB] = [(\cW \circ \cB) \circ \cA] \in [*[i+j], \cH]$$
        \end{lem}
        \begin{proof}
            For $z \in \cH$, define $\tilde{\cQ}_{*z}(\cB, \tilde{\cP}(\cA, \cW))$ to be the coequaliser of the following diagram:
            $$
                \bigsqcup\limits_{y,y' \in \cG} \cB_{*y} \times \cG_{yy'} \times \tilde{\cP}_{y'z} \times [0, \eps)^2_{u,v} \rightrightarrows \bigsqcup\limits_{y \in \cG} \cB_{*y} \times \tilde{\cP}_{yz} \times [0, \eps)_s
            $$
            where the two maps send $(a, b, c, u, v)$ to $(a, \cC(b, c, u), v)$ and $(\cC(a, b, v), c, u)$ respectively, as before.\par 
            By Lemmas \ref{lem:abstract gluing} and \ref{lem:abstract gluing with faces}, $\tilde{\cQ}_{*z}(\cB, \tilde{\cP}(\cA, \cW))$ has a system of boundary faces given by $\tilde{\cQ}_{*z'} \times \cH_{z'z}$, $\cA_{*x} \times \{0\} \times \tilde{\cP}_{xz}$ and ${\cB}_{*y} \times \tilde{\cP}_{yz} \times \{0\}$ for $x \in \cF$, $y \in \cG$ and $z' \in \cH$; we note that the final two types of boundary faces overlap in codimension two faces of the form $\cA_{*x} \times \cB_{*y} \times \cW_{xy;z} \times \{(0,0)\}$, which is a compact face.\par 
            
        Choose framed functions $(f_{xy}, s_{xy})$ and $(f_{yz}, s_{yz})$ on each $\tilde{\cP}_{xz}$ and $\tilde{\cP}_{yz}$ respectively, and let $\cU_{xz} = f_{xz}^{-1}\{1\}$ and $\cU_{yz} = f_{yz}^{-1}\{1\}$ respectively; these are representatives of $\cW \circ \cB$ and $\cW \circ \cA$ respectively.\par 
        Now choose framed functions $(g_{*z}, t_{*z})$ on each $\tilde{\cQ}_{*z}$, extending the $(f_{xy}, s_{xy})$ and $(f_{yz}, s_{yz})$ on appropriate boundary faces. Then considering $g_{*z}^{-1}\{1\}$ and using the argument in Lemma \ref{composition recognition}, the result follows.
        \end{proof}
        \begin{defn}
            We define a map
            $$\cW_*: [*[i], \cF] \otimes [*[j], \cG] \rightarrow [*[i+j], \cH]$$
            sending $[\cA] \otimes [\cB]$ to $[(\cW \circ \cA) \circ \cB] = [(\cW \circ \cB) \circ \cA] \in [*[i+j], \cH]$. This is a well-defined map, by Lemmas \ref{lem:two bilin} and \ref{composition is well-defined}.
        \end{defn}
        \begin{lem}\label{lem:bilin 2}
            $\cW_*$ is a bilinear map.
        \end{lem}
        \begin{proof}
            Since composition is bilinear, $[(\cW \circ \cA) \circ \cB]$ is linear in $[\cB]$. Similarly $[(\cW \circ \cB) \circ \cA]$ is linear in $[\cA]$.
        \end{proof}
       
        \begin{defn}\label{def:bilin bor}
            Let $\cW, \cV: \cF \times \cG \to \cH$ be bilinear maps. A \emph{bordism} between $\cW$ and $\cV$ consists of compact manifolds with faces $\cR_{xy;z}$ of dimension $|x|+|y|-|z|+1$ for each $x \in \cF$, $y \in \cG$ and $z \in \cH$, with systems of boundary faces given by $\cF_{xx'} \times \cR_{x'y;z}$, $\cG_{yy'} \times \cR_{xy';z}$, $\cR_{xz;z'} \times \cH_{z'z}$, $\cW_{xy;z}$ and $\cV_{xy;z}$. Note all of these are broken except the last two which are unbroken; we label them as incoming and outgoing respectively.\par 
            In both cases, as usual we that the inclusions $c$ of each boundary face satisfy appropriate associativity relations.
        \end{defn}
        If $\cW: \cF \times \cG \to \cH$ is a bilinear map and $\cA: \cF' \to \cF$ is a morphism, we can form another bilinear map $\cW \circ \cA: \cF' \times \cG \to \cH$ in the same way as  in the previously treated case $\cF = *[i]$: for $w \in \cF'$, $y \in \cG$ and $z \in \cH$, we let $\tilde{\cP}_{wy;z}=\tilde{\cP}_{wy;z}(\cA, \cW)$ be the coequaliser of
        $$\bigsqcup_{x, x' \in \cF} \cA_{wx} \times \cF_{xx'} \times \cW_{x'y;z} \times [0, \eps)^2_{u,v} \rightrightarrows \bigsqcup_{x \in \cF} \cA_{wx} \times \cW_{xy;z} \times [0, \eps)_s$$
        Then choosing appropriate compatible framed functions and taking the preimage of 1 gives a bilinear map $\cF' \times \cG \to \cH$; the usual arguments show this is well-defined up to bordism. Note this generalises the previous construction by Example \ref{ex: bilin point 2}.\par 
        If instead we are given a morphism $\cB: \cG' \to \cG$, we can similarly form a bilinear map $\cW \circ \cB: \cF \times \cG' \to \cH$.
        \begin{lem}\label{lem:bordant bilin}
            Let $\cW, \cV: \cF \times \cG \to \cH$ be bilinear maps, $\cA: \cF' \to \cF$ a morphism and $\cR$ a bordism from $\cW$ to $\cV$. Then $\cW \circ \cA$ and $\cV \circ \cA$ are bordant bilinear maps $\cF' \times \cG \to \cH$.\par 
            In particular, $\cW_*=\cV_*$.
        \end{lem}
        \begin{proof}
            For $w \in \cF'$, $y \in \cG$ and $z \in \cH$, let $\tilde{\cP}_{wy;z} = \tilde{\cP}_{wy;z}(\cA, \cR)$ be the coequaliser of 
            $$\bigsqcup_{x, x' \in \cF} \cA_{wx} \times \cF_{xx'} \times \cR_{x'y;z} \times [0, \eps)^2_{u,v} \rightrightarrows \bigsqcup_{x \in \cF} \cA_{wx} \times \cR_{xy;z} \times [0, \eps)_s$$
            Then choosing suitable compatible framed functions $(f_{wy;z}, s_{wy;z})$ on each $\tilde{\cP}_{wy;z}$ extending compatible framed functions on the boundary faces $\tilde{\cP}_{wy;z}(\cA, \cW)$ and $\tilde{\cP}_{wy;z}(\cA, \cV)$ and taking the preimage of 1 provides the desired bordism.
        \end{proof}
        Now let $\cF$, $\cG$, $\cH$, $\cI$, $\cJ$ and $\cK$ be flow categories, and let $\cW: \cF \times \cG \rightarrow \cI$, $\cV: \cG \times \cH \rightarrow \cJ$, $\cX: \cI \times \cH \rightarrow \cK$ and $\cY: \cF \times \cJ \rightarrow \cK$ be bilinear maps. These induce trilinear maps of abelian groups
        $$\cX_* \circ \cW_*, \cY_* \circ \cV_*: [*[i], \cF] \otimes [*[j], \cG] \otimes [*[k], \cH] \rightarrow [*[i+j+k], \cK]$$
        We will give conditions for these two maps to agree.
        \begin{defn}
            An \emph{associator} $\cZ$ for the decuplet $(\cF, \cG, \cH, \cI, \cJ, \cK, \cW, \cV, \cX, \cY)$ consists of compact manifolds with faces $\cZ_{xyz;w}$ for $x$ in $\cF$, $y$ in $\cG$, $z$ in $\cH$ and $w$ in $\cI$, of dimension $|x|+|y|+|z|-|w|+1$, with a system of faces given by
            $$\cF_{x_0 \ldots x_i} \times \cG_{y_0 \ldots y_j} \times \cH_{z_0 \ldots z_k} \times \cZ_{x_i y_j z_k; w_0} \times \cK_{w_0 \ldots w_l}$$
            and
            $$\cF_{x_0 \ldots x_i} \times \cG_{y_0 \ldots y_j} \times \cH_{z_0 \ldots z_k} \times \cW_{x_i y_j; u_0} \times \cI_{u_0 \ldots u_r} \times \cX_{u_r z_k; w_0} \times \cK_{w_0 \ldots w_l}$$
            and
            $$\cF_{x_0 \ldots x_i} \times \cG_{y_0 \ldots y_j} \times \cH_{z_0 \ldots z_k} \times \cV_{y_j z_k; v_0} \times \cJ_{v_0 \ldots v_s} \times \cY_{x_i v_j; w_0} \times \cK_{w_0 \ldots w_l}$$
            for $x_0 = x, \ldots, x_i$ in $\cF$, $y_0 = y, \ldots, y_j$ in $\cG$, $z_0 = z, \ldots, z_k$ in $\cH$, $u_0, \ldots, u_r$ in $\cI$, $v_0, \ldots, v_s$ in $\cJ$ and $w_0, \ldots, w_l = w$ in $\cK$.
        \end{defn}
        Note that in particular for an associator $\cZ$, each $\cZ_{xyz; w}$ has a system of boundary faces
        $$\cF_{xx'} \times \cZ_{x'yz; w}$$
        and
        $$\cG_{yy'} \times \cZ_{xy'z;w}$$
        and
        $$\cH_{zz'} \times \cZ_{xyz'; w}$$
        and
        $$\cW_{xy;u} \times \cX_{uz; w}$$
        and
        $$\cV_{yz; v} \times \cY_{xv; w}$$
        and
        $$\cZ_{xyz;w'} \times \cK_{w'w}$$
        for $x,x'$ in $\cF$, $y,y'$ in $\cG$, $z,z'$ in $\cH$, $u$ in $\cI$, $v$ in $\cJ$ and $w,w'$ in $\cK$.
        \begin{lem}\label{associators}
            If there exists an associator $\cZ$ for $(\cF, \cG, \cH, \cI, \cJ, \cK, \cW, \cV, \cX, \cY)$, then the two induced trilinear maps are equal:
            $$\cX_* \circ \cW_*= \cY_* \circ \cV_*: [*[i], \cF] \otimes [*[j], \cG] \otimes [*[k], \cH] \rightarrow [*[i+j+k], \cK]$$
        \end{lem}
        \begin{proof}
            $\cB \circ {\cZ}$ along with the proof of Lemma \ref{composition recognition} provides a bordism between the bilinear maps $\cX \circ(\cW \circ \cB)$ and $\cY \circ (\cV \circ \cB)$; the result then follows from Lemma \ref{lem:bordant bilin}.
        \end{proof}
\subsection{Bilinear maps: framings}
    We incorporate stable framings into Section \ref{sec:bilin} in the same way as in Section \ref{Sec:framings}. Similarly to Definition \ref{def:farm}, we have:
    \begin{defn}\label{def:bil fram}
        Let $\cF$, $\cG$ and $\cH$ be framed flow categories. A \emph{stable framing} on a bilinear map $\cW: \cF \times \cG \to \cH$ consists of stable isomorphisms
        $$st=st^\cW_{xyz}: V_x \oplus V_y \oplus I^\cW_{xy;z} \to \bR \oplus V_z $$
        over each $\cW_{xy,z}$ with stabilising bundles $\bT^\cW_{xyz}$, along with, for each boundary face of each $\cW_{xy;z}$, an embedding of the stabilising bundle of the induced framing on that face into $\bT^\cW_{xyz}$. \par 
        Here $I^\cW_{xy;z} \tau_z$ is defined to be $T\cW_{xy;z} \oplus \bR \tau_z$ (the same as in Definition \ref{def:abstr ind}), with abstract gluing isomorphisms $\psi$ defined similarly (noting that all boundary faces here are broken).\par 
        We require that these stable isomorphisms and embeddings, combined with those for $\cF$, $\cG$, $\cH$ and each lower-dimensional $\cW_{x'y';z'}$, form systems of framings for each $\cW_{xy;z}$. 
    \end{defn}
    \begin{ex}
        The analogues of Examples \ref{ex: bilin point} and \ref{ex: bilin point 2} hold for framed bilinear maps and framed morphisms.
    \end{ex}
    
    Assume $\cW: \cF \times \cG \to \cH$ be a framed bilinear map, and let $\cA: \cF' \to \cF$ and $\cB: \cG' \to \cG$ be framed morphisms. We equip each $\tilde{\cP}(\cA, \cW)$ and $\tilde{\cP}(\cB, \cW)$ (constructed as in Section \ref{sec:bilin}) with framings, in exactly the same way we framed each $\tilde{\cQ}(\cdot, \cdot)$ in Lemma \ref{lem:fram Q}, and from this we frame $\cW \circ \cA$ and $\cW \circ \cB$ in exactly the same way as in Lemma \ref{lem:fram comp well-def}. From this plus (the argument in) Lemma \ref{lem:comp fram}, we have:
    \begin{lem}
        The bilinear maps $\cW \circ \cA: \cF' \times \cG \to \cH$ and $\cW \circ \cB: \cF \times \cG' \to \cH$ admit framings, and define well-defined framed morphisms up to framed bordism.
    \end{lem}
    Similar to Lemma \ref{lem:bilin 2} (and by the same proof, with the usual incorporation of framings), we have:
    \begin{lem}
        For all framed morphisms of flow categories $\cA: *[i] \to \cF$ and $\cB: *[j] \to \cG$, we have that
            $$[(\cW \circ \cA) \circ \cB] = [(\cW \circ \cB) \circ \cA] \in [*[i+j], \cH]^{fr}$$
    \end{lem}

        The same arguments as in Lemmas \ref{lem:comp fram} and \ref{lem:bilin 2} prove
        \begin{lem}
            This construction defines a well-defined bilinear map induced by $\cW$:
            $$\cW_*: [*[i],\cF]^{fr} \otimes [*[j],\cG]^{fr} \to [*[i+j],\cH]^{fr}$$
        \end{lem}
        Next is bordisms of bilinear maps. Let $\cR$ be a bordism from $\cW$ to $\cV$, where $\cW, \cV: \cF \times \cG \to \cH$ are bilinear maps, and assume $\cW,\cV,\cF,\cG,\cH$ are all compatibly framed. By labelling the boundary faces of each $\cR_{xy;z}$ as broken or incoming/outgoing unbroken as in Definition \ref{def:bilin bor}, from Section \ref{sec:abstr ind bun} we obtain abstract index bundles $I^\cR_{xy;z}$ along with coherent abstract gluing isomorphisms over each boundary face.
    \begin{defn}
            A \emph{framing} on the bordism $\cR$ consists of stable isomorphisms 
            $$st_{xy;z}: V_x \oplus V_y \oplus I^\cR_{xy;z} \to \bR^2 \oplus V_z$$
            over each $\cR_{xy,z}$ with stabilising bundle $\bT^\cR_{xy;z}$, along with, for each boundary face of $\cR_{xy;z}$, an embedding of the stabilising bundle of the induced framing on that face into $\bT^\cR_{xyz}$.\par
            We require that these stable isomorphisms and embeddings, combined with those of $\cW,\cV,\cF,\cG,\cH$ and the framings for each lower-dimensional $\cR_{x'y';z'}$, form systems of framings for each $\cR_{xy;z}$.
        \end{defn}
        Incorporating framings into the proof of Lemma \ref{lem:bordant bilin} just as in Lemma \ref{lem:fram comp well-def}, we have that:
        \begin{lem}
            Let $\cW, \cV: \cF \times \cG \to \cH$ be framed bilinear maps, $\cA: \cF' \to \cF$ a framed morphism and $\cR$ a framed bordism from $\cW$ to $\cV$. Then $\cW \circ \cA$ and $\cV \circ \cA$ are framed bordant bilinear maps $\cF' \times \cG \to \cH$.\par 
            In particular, $\cW_*=\cV_*$.
        \end{lem}
        We now move onto associators. Let $\cZ$ be an associator for the decuplet $(\cF, \cG, \cH, \cI, \cJ, \cK, \cW, \cV, \cX, \cY)$. Note all boundary faces of each $\cZ_{xyz;w}$ are broken. 
        \begin{defn}
            A \emph{framing} on the associator $\cZ$ consists of stable isomorphisms
            $$st_{xyz;w}: V_x \oplus V_y \oplus V_z \oplus I^\cZ_{xyz;w}  \to \bR^3 \oplus V_w$$
            over each $\cZ_{xyz;w}$ with stabilising bundle $\bT^\cZ_{xyz;w}$, along with, for each boundary face of $\cZ_{xyz;w}$, an embedding of the stabilising bundle of the induced framing on that face into $\bT^\cZ_{xyz;w}$.\par 
            We require that these, combined with the framings for all items in the decuplet and the framings for each lower-dimensional $\cZ_{x'y'z';w'}$, form systems of framings for each $\cZ_{xyz;w}$.
        \end{defn}
        Similarly to Lemma \ref{associators}, by incorporating framings in the usual way, we have the following lemma:
        \begin{lem}\label{lem:framed associators}
            If there exists an framed associator $\cZ$ for a framed decuplet $(\cF, \cG, \cH, \cI, \cJ, \cK, \cW, \cV, \cX, \cY)$ as above, then the two induced trilinear maps are equal:
            $$\cX_* \circ \cW_*= \cY_* \circ \cV_*: [*[i],\cF]^{fr} \otimes [*[j],\cG]^{fr} \otimes [*[k],\cH]^{fr} \to [*[i+j+k],\cK]^{fr}$$
        \end{lem}
\subsection{Morse complex of a flow category}
\label{sec:Morse}
Let $\cM$ be a flow category (at this point not necessarily framed).  We define
\[
CM_i(\cM) = \oplus_{|x|=i} \bZ/2 \langle x \rangle.
\]
We define a differential $\partial: CM_i(\cM) \to CM_{i-1}(\cM)$ by setting
\[
\partial(x) = \sum_{y : |y|=|x|-1} |\cM_{xy}|\cdot y
\]
where $\cM_{xy}$ is a manifold of dimension $|x|-|y|-1=0$ and $|\cM_{xy}| \in \bZ/2$ denotes its cardinality.

\begin{lem}
    $\partial$ is a differential.
\end{lem}

The proof is a simplification of the case of the Morse complex of a framed flow category given below in Lemma \ref{lem:differential}.

There is a canonical isomorphism of rings $\Gamma: \Omega^{fr}_0 \xrightarrow{\cong} \bZ$ between the 0-dimensional framed bordism group and $\bZ$.
    \begin{defn}\label{morse chains definition}
        Let $\cM$ be a framed flow category.
        For each object $x$, there are two homotopy classes of stable isomorphism $V_x \cong \bR^{-|x|}$. We let $F(x)$ be the set of such homotopy classes of isomorphisms.\par  
        We define the \emph{Morse chain complex} $CM_*(\cM)$ to be the chain complex with underlying graded abelian groups given by
        $$CM_i(\cM) := \left(\bigoplus_{|x| = i, \phi \in F(x)} \bZ \phi \right) / A_i $$ where $A_i$ is the subgroup generated by expressions
        $$\sum_{\phi \in F(x)} \phi$$
        for all $x$. This has differential $\partial$, which is defined as follows. We first note that each $\cM_{xy}$ is not naturally a stably framed manifold, but does acquire a natural stable framing if we fix stable isomorphisms between $V_x$ and $V_y$ and $\bR^N$ for appropriate $N$.\par 
        Let $x$ be an object in $\cM$, with index $i = \mu(x)$. For a generator $\phi$ in $F(x)$, 
        $$\partial \phi := \sum_{\mu(y) = i - 1} \Gamma(\cM_{xy}) \psi_y$$
        where we choose some $\psi_y$ in $F(y)$ for each $y$, and $\cM_{xy}$ is equipped with the stable framing it acquires from $\phi$ and $\psi_y$. Then $\partial$ is a linear map which sends $A_i$ to $A_{i-1}$ and so descends to a well-defined linear map $CM_i(\cM) \rightarrow CM_{i-1}(\cM)$, independent of the choices of $\psi_y$.
        \end{defn}
        \begin{rmk}
            Note that an evidently equivalent definition to this is to initially pick a (stable) isomorphism $\phi_x: V_x \cong \bR^{N_x + \mu(x)}$ for each $x$, and only take one generator for each $x$, and then to not quotient by the subgroups $A_*$, observing that the resulting complex is independent of the choices of isomorphisms $\phi_x$. This is in some sense a nicer definition but requires us to keep track of more choices along the way.
        \end{rmk}

    \begin{lem} \label{lem:differential}
        $\partial$ is a differential. Equivalently,
        $$\partial \circ \partial = 0.$$
    \end{lem}
    \begin{proof}
        Let $x$ be an object of $\cM$, of index $i = \mu(x)$. Choose $\phi$ in $F(x)$. Then
        $$\partial \circ \partial(x) = \sum\limits_{\mu(y) = i - 1} \sum\limits_{\mu(z) = i - 2} \Gamma(\cM_{xy}) \Gamma(\cM_{yz})\, \theta_z$$
        where we choose $\psi_y$ in $F(y)$ and $\theta_z$ in $F(z)$ for each $y$, $z$, which induces framings on $\cM_{xy}$ and $\cM_{yz}$. We will show that the coefficient of each $\theta_z$ is 0, which implies the whole sum is 0.\par 
        The coefficient of each $\theta_z$ is 
        $$\sum\limits_{\mu(y) = i-1} \Gamma(\cM_{xy}) \Gamma(\cM_{yz})$$
        which, since $\Gamma$ is a ring map, is equal to
        $$\Gamma\left(\bigsqcup\limits_{\mu(y) = i-1} \cM_{xy} \times \cM_{yz}\right)$$
        which is equal to 
        $$\Gamma\left(\partial \cM_{xz}\right)$$
        and hence vanishes.
    \end{proof}
    \begin{defn}
        The \emph{Morse homology} $HM_*(\cM)$ of $\cM$ is defined to be the homology of the chain complex $CM_*(\cM)$.
    \end{defn}

     \begin{defn}
        Given a morphism $\cW$ of framed flow categories from $\cM^0$ to $\cM^1$, there is an induced map
        $$CM_*\left(\cW\right): CM_*\left(\cM^0\right) \rightarrow CM_*\left(\cM^1\right)$$
        defined as follows. Let $x$ be an object in $\cM^0$, with index $i = |x|$. For a generator $\phi$ in $F(x)$, 
        $$CM_*\left(\cW\right) \phi := \sum\limits_{|y| = i} \Gamma\left(\cW_{xy}\right)\, \psi_y$$
        where we choose some $\psi_y$ in $F(y)$ for each $y$, and $\cW_{xy}$ is equipped with the stable framing it acquires from $\phi$ and $\psi$. Then $CM_*\left(\cW\right)$ is a linear map which sends subgroups $A_i^0$ to $A_i^1$ and so descends to a well-defined linear map $CM_*\left(\cM^0\right)$ to $CM_*\left(\cM^1\right)$, independent of the choices of $\psi_y$.
    \end{defn}
    \begin{lem}
        $CM_*\left(\cW\right)$ is a chain map. Equivalently,
        $$\partial^1 \circ CM_*\left(\cW\right) = CM_*\left(\cW\right) \circ \partial^0.$$
    \end{lem}
    \begin{proof}
        Let $x$ be an object of $\cM^0$, of index $i = |x|$. Choose $\phi$ in $F(x)$. Then 
        $$\partial^1 \circ CM_*\left(\cW\right)(x) = \sum\limits_{\substack{y \in \cM^1\\|y| = i}} \sum\limits_{\substack{z \in \cM^1\\|z| = i - 1}}
        \Gamma\left(\cW_{xy}\right) \Gamma\left(\cM^1_{yz}(y, z)\right) \theta_z$$
        where we choose $\psi_y$ in $F(y)$ and $\theta_z$ in $F(z)$ for each $y, z$, similarly to the proof of Lemma \ref{lem:differential}. Similarly,
        $$CM_*\left(\cW\right) \circ \partial^0 (x) = \sum\limits_{\substack{y \in \cM^0\\ |y| = i-1}}\sum\limits_{\substack{z \in \cM^1\\ |z| = i-1}} \Gamma\left(\cM^0_{xy}\right) \Gamma \left(\cW_{yz}\right) \theta_z.$$
        The coefficient of each $\theta_z$ in the difference between these expressions is
        $$\sum\limits_{\substack{y \in \cM^1\\ |y| = i}} \Gamma\left(\cW_{xy}\right) \Gamma\left(\cM^1_{yz}\right) - \sum\limits_{\substack{y \in \cM^0\\  |y| = i-1}} \Gamma\left(\cM^0_{xy}\right)\Gamma\left(\cM^{01}_{yz}\right)$$
        which, since $\Gamma$ is a ring map, is
        $$\Gamma \left(\bigsqcup\limits_{\substack{y \in \cM^1\\ |y| = i}} \cW_{xy} \times \cM^1_{yz}\sqcup \bigsqcup\limits_{\substack{y \in \cM^0\\ |y| = i - 1}} \overline{\cM^0_{xy} \times \cW_{yz}}\right)$$
        which is equal to 
        $$\Gamma\left(\partial \cW_{xy}\right)$$
        and hence vanishes. Note that $\overline{\cM^0_{xy} \times \cW_{yz}}$ appears with the framing reversed; see Remark \ref{rmk:fram sgn}.
    \end{proof}

    Now let $\cM^i$ be framed flow categories, for $i\in \{0,1,2\}$, and $\cR: \cM^0 \times \cM^1 \to \cM^2$ be a framed bilinear map. 

    \begin{lem}
        The framed bilinear map $\cR$ induces a degree zero product
        \[
        \mu_{\cR}: CM_*(\cM^0) \otimes CM_*(\cM^1) \longrightarrow CM_*(\cM^2)
        \]
        which descends to homology. 
        If $\cR$ admits a framed associator, then the product on homology is associative.
    \end{lem}

    \begin{proof}
        Fix objects $x$ and $y$ of $\cM^0$ and $\cM^1$ of degrees $|x|=i$ and $|y|=j$. For each $z \in \cM^2$ with $|z| = i+j$ we have a zero-dimensional manifold $\cW_{xy;z}$ and we set 
        \[
        \mu_{\cR}(\phi_x\otimes \phi_y) = \sum_{|z|=i+j} \Gamma(\cR_{xy;z}) \, \phi_z
        \]
        where $\phi_{\bullet} \in F(\bullet)$ for $\bullet=x,y,z$ and the choices of $\phi_{\bullet}$ equip $\cW_{xy;z}$ with a stable framing. As in the construction of the differential on $CM_i(\cM)$ above, the map does not depend on the choices of the $\phi_{\bullet}$. The fact that $\cW_{xy;z}$,  with $|z|=|x|+|y|-1$, has a  system of boundary faces indexed by
        \[
        \cM^0_{xx'} \times \cR_{x'y;z}, \, \cM^1_{yy'} \times \cR_{xy';z}, \, \cR_{xy;z'}\times \cM^2_{z'z}
        \]
        for appropriate $x',y',z'$ 
        shows that $\mu_{\cR}$ satisfies the Leibniz rule and hence descends to a map on homology. The proof of the final claim is analogous.
            \end{proof}
        \subsection{Truncations}
        \begin{defn}
            Let $\cF$ and $\cG$ be flow categories. A \emph{pre-$\tau_{\leq n}$-morphism} $\cW$ from $\cF$ to $\cG$ consists of compact smooth manifolds with faces $\cW_{xy}$ (of dimension $|x|-|y|$) for all $x$ in $\cF$ and $y$ in $\cG$ such that $|x|-|y| \leq n$, along with maps
            $$c = c^\cW: \cF_{xz} \times \cW_{zy} \rightarrow \cW_{xy}$$
            and
            $$c = c^\cW: \cW_{xz} \times \cG_{zy} \rightarrow \cW_{xy}$$
            which are compatible with $c^\cF$ and $c^\cG$, which define a system of faces
            $$\cW_{x_0 \ldots x_i; y_0 \ldots y_j}$$
            for $x_0 = x, \ldots x_i$ in $\cF$ and $y_0, \ldots, y_j = y$ in $\cG$.\par 
            A \emph{pre-$\tau_{\leq n}$-bordism} $\mathcal{R}$ between two pre-$\tau_{\leq n-1}$-morphisms $\mathcal{W}$ and $\mathcal{V}$ consists of compact smooth manifolds with faces $\mathcal{R}_{xy}$ (of dimension $|x|-|y|$ + 1), for all $x$ in $\mathcal{F}$ and $y$ in $\mathcal{G}$ such that $|x|-|y|+1 \leq n$, along with maps
            $$c=c^\cR: \cW_{xy}, \cV_{xy} \rightarrow \cR_{xy}$$
            and 
            $$c = c^\cR: \cF_{xz} \times \cR_{zy} \rightarrow \cR_{xy}$$
            and
            $$c = c^\cR: \cR_{xz} \times \cG_{zy} \rightarrow \cR_{xy}$$
            compatible with $c^\cF$, $c^\cG$, $c^\cW$, $c^\cV$, such that the images of the faces in the domains defines a system of faces for $\cR_{xy}$ (after deleting repeats). More precisely, there is a system of faces for $\cR_{xy}$ given by the following:
            $$\cW_{x_0 \ldots x_i; y_0 \ldots y_j}$$
            and
            $$\cV_{x_0 \ldots x_i; y_0 \ldots y_j}$$
            and
            $$\cR_{x_0 \ldots x_i; y_0 \ldots y_j}$$
            for $x_0=x, \ldots, x_i$ in $\cF$ and $y_0, \ldots, y_j$ in $\cG$. 
        \end{defn}
        \begin{rmk}\label{explanation for truncated premorphisms}
            A more intuitive description is the following. A pre-$\tau_{\leq n}$-morphism $\cW$ is like a morphism of flow categories $\cW$ except we only require the manifolds with faces $\cW_{xy}$ to exist when their dimension is $\leq n$, but when they do exist they have the same concatenation maps, defining a system of faces in the same way as before. A similar description applies to pre-bordisms.
        \end{rmk}
        \begin{rmk}
            Similarly to with morphisms and bordisms, we equip any pre-$\tau_{\leq n}$-morphism or pre-$\tau_{\leq n}$-bordism we encounter with a system of collars, compatible with those on $\cF$ and $\cG$ (and $\cW$ and $\cV$ in the case of a pre-$\tau_{\leq n}$-bordism).
        \end{rmk}
        \begin{ex}\label{truncating pre-morphisms}
            Any morphism of flow categories $\cW: \cF \rightarrow \cG$ determines a pre-$\tau_{\leq n}$-morphism $\tau_{\leq n} \cW: \cF \rightarrow \cG$ by forgetting the manifolds $\cW_{xy}$ whenever $|x|-|y| > n$.   \par 
            Similarly, whenever $m \geq n$, any pre-$\tau_{\leq m}$-morphism $\cW$ determines a pre-$\tau_{\leq n}$-morphism $\tau_{\leq n}\cW$ in the same way.\par 
            Clearly composing these truncation operations is associative, i.e. $\tau_{\leq n}\tau_{\leq m}\cW = \tau_{\leq n}\cW$, for $n \leq m$.\par 
            Note that these operations are not necessarily surjective:  not every pre-$\tau_{\leq n}$-morphism $\cF \rightarrow \cG$ is $\tau_{\leq n} \cW$ for $\cW$ some pre-$\tau_{\leq n+1}$-morphism.\par
            There are similar truncation operations for pre-$\tau_{\leq n}$-bordisms. 
            \end{ex}
        \begin{defn}\label{truncated morphisms}
            Let $\cF$ and $\cG$ be flow categories. A \emph{$\tau_{\leq n}$-morphism} $\cW$ from $\cF$ to $\cG$ is a pre-$\tau_{\leq n}$-morphism $\cW: \cF \rightarrow \cG$ such that $\cW = \tau_{\leq n} \cW'$ for some pre-$\tau_{\leq n+1}$-morphism $\cW'$. We call $\cW'$ an \emph{extension} of $\cF$.\par 
            A \emph{$\tau_{\leq n}$-bordism} between two $\tau_{\leq n}$-morphisms $\cW$ and $\cV$ is a pre-$\tau_{\leq n}$-bordism $\cR$ between $\tau_{\leq n-1}\cW$ and $\tau_{\leq n-1}\cV$ such that $\cR = \tau_{\leq n} \cR'$ for $\cR'$ some pre-$\tau_{\leq n+1}$-bordism between $\cW$ and $\cV$. We call $\cR'$ an \emph{extension} of $\cR$. \par 
            We emphasise that in these definitions, the extensions $\cW'$ and $\cR'$ are assumed to exist, but are not kept track of as part of the data.
        \end{defn}
        \begin{rmk}
            Similarly to Remark \ref{explanation for truncated premorphisms}, we think of these as morphisms (or bordisms of morphisms) of flow categories, where we only specify the manifolds with faces $\cW_{xy}$ when their dimension is $\leq n$ (again with the appropriate system of faces), and such that there exist manifolds with faces $\cW_{xy}$ when the dimension is $n+1$ with the appropriate system of faces, but without keeping track of these $(n+1)$-dimensional manifolds with faces as part of the data.
        \end{rmk}
        \begin{rmk}
            We may think of a morphism of flow categories as a $\tau_{\leq \infty}$-morphism, or equivalently as a pre-$\tau_{\leq \infty}$-morphism.
        \end{rmk}
        \begin{ex}\label{truncating morphisms}
            Similarly to Example \ref{truncating pre-morphisms}, any morphism of flow categories $\cW: \cF \rightarrow \cG$ determines a $\tau_{\leq n}$-morphism $\tau_{\leq n} \cW$ by forgetting all manifolds of dimension $\geq n$.
        \end{ex}
        \begin{ex}\label{0-truncation and chains}
            A pre-$\tau_{\leq 0}$-morphism $\cW: \cF \rightarrow \cG$ determines a linear map $CM_*(\cW): CM_*(\cF; \bZ/2) \rightarrow CM_*(\cG; \bZ/2)$, defined by
            $$CM_*(\cW)(x) := \sum\limits_{y \in \cG\,,\, |y|=|x|} |\cW_{xy}| \cdot y$$
            for $x$ in $\cF$. Moreover, 
            $\cW$ is a $\tau_{\leq 0}$-morphism if and only if $CM_*(\cW)$ is a chain map.\par 
            If $\cW = \tau_{\leq 0} \cV$ for some morphism $\cV: \cF \to \cG$, then this is the same as the construction in Section \ref{sec:Morse}.
        \end{ex}
        \begin{ex}\label{cone chain cone}
            Given a morphism $\cW: \cF \rightarrow \cG$, there is a natural isomorphism
            $$CM_*(\Cone(\cW)) \cong \Cone(CM_*(\cW): CM_*(\cF) \to CM_*(\cG))$$
        \end{ex}
        \begin{rmk}
            Example \ref{0-truncation and chains} is one of the main reasons we will define a category built out of $\tau_{\leq n}$-morphisms, rather that pre-$\tau_{\leq n}$-morphisms.
        \end{rmk}
        Similarly to Definition \ref{flow morphism def}, we give the following definition.
        \begin{defn}
            Let $\cF$ and $\cG$ be flow categories. Define $[\cF, \cG]_{\tau_{\leq n}}$ to be the set of $\tau_{\leq n}$-morphisms $\cF \rightarrow \cG$, modulo the equivalence relation generated by
            \begin{enumerate}
                \item $\cW \sim \cV$ if $\cW$ and $\cV$ are $\tau_{\leq n}$-bordant.
                \item $\cW \sim \cV$ if $\cW$ and $\cV$ are diffeomorphic.
            \end{enumerate}
            This is an abelian group under disjoint union, with unit given by the empty morphism.
        \end{defn}
        \begin{ex}
            Truncation as in Example \ref{truncating morphisms} determines morphisms of abelian groups
            $$[\cF, \cG] \xrightarrow{\tau_{\leq m}} [\cF, \cG]_{\tau_{\leq m}} \xrightarrow{\tau_{\leq n}} [\cF, \cG]_{\tau_{\leq n}}$$
            for $m \geq n$.
        \end{ex}
        \begin{lem}\label{0-truncation is chain}
            There is a natural isomorphism $\Xi: [\cF, \cG]_{\tau_{\leq 0}} \rightarrow \Hom(CM_*(\cF), CM_*(\cG))$ (where the target is the set of chain maps $CM_*(\cF) \rightarrow CM_*(\cG)$ up to homotopy).
        \end{lem}
        \begin{proof}
        We first show that $[\cF,\cG]_{\tau_{\leq 0}} \to \Hom(CM_*(\cF), CM_*(\cG))$ is injective.  Suppose $\cW$ and $\cV$ are flow $\tau_{\leq 0}$-morphisms which are bordant, via a bordism $\cR$. Recall that the bordism relation means that we have compact smooth manifolds with faces $\cR_{xy}$ of dimensions $|x|-|y|+1$ for $x\in \cF$ and $y\in \cG$ (with appropriate faces). In particular we can define a  linear map $CM_*(\cF) \to CM_{*+1}(\cG)$ via 
        \[ 
        x \mapsto \sum_{|y| = |x|+1} \cR_{xy} \, y
        \]
        The bordism relation implies this defines a chain homotopy between the chain maps $x\mapsto \sum_{|y|=|x|} \cW_{xy} \, y$ and $x\mapsto \sum_{|y|=|x|} \cV_{xy} \, y$. A corresponding simpler argument when $\cW$ and $\cV$ are related by diffeomorphism then shows injectivity of $\Xi$. Analogously, any linear map $CM_*(\cF) \to CM_*(\cG)$, say $x \mapsto \sum_{|y|=|x|} n_{xy} y$ for integers $n_{xy}$, can be viewed as a collection of  framed $0$-manifolds $n_{xy}$  which define a pre-$\tau_{\leq 0}$-morphism. The condition that the linear map is a chain map exactly says this lifts to a $\tau_{\leq 0}$-morphism, which gives surjectivity of $\Xi$.
        \end{proof}
        Given pre-{$\tau_{\leq n}$-morphisms $\cW: \cF \rightarrow \cG$ and $\cV: \cG \rightarrow \cH$, we can form their composition, $\cV \circ \cW$, identically to the case of morphisms of flow categories. \par 
        We first define $\tilde{\cQ}_{xy} = \tilde{\cQ}_{xy}(\cW, \cV)$ as in Equation \ref{tilde Q} whenever $|x|-|y| \leq n$. The $\tilde{\cQ}_{xy}$ are manifolds with faces by the same argument as before, with a system of faces given by the same description. Choosing compatible framed functions $(f_{xy}, s_{xy})$ as before and defining $\cU_{xy} = f_{xy}^{-1}\{p\}$ whenever $|x|-|y| \leq n$. defines a pre-$\tau_{\leq n}$ morphism $\cU: \cF \to \cH$, which we define $\cV \circ \cG$ to be.
        Composition of this form is compatible with further truncation: if $n > k$ and we apply the above construction to $\tau_{\leq k}\cW$ and $\tau_{\leq k}\cV$ and take the restrictions of the same framed functions functions, the output will be exactly $\tau_{\leq k}\cU$.\par 
        \begin{defn}
            If $\cW$ and $\cV$ are $\tau_{\leq n}$-morphisms, we choose extensions $\cW'$ and $\cV'$ and compose these as above to get a pre-$\tau_{\leq n+1}$-morphism $\cU'$. Then we define the \emph{composition} of $\cW$ and $\cV$ to be the $\tau_{\leq n}$-morphism $\tau_{\leq n}\cU'$.
        \end{defn}
        \begin{rmk}
            As a pre-$\tau_{\leq n}$-morphism, $\tau_{\leq n}\cU'$ is the composition of $\cW$ and $\cV$. However if we composed $\cW$ and $\cV$ as pre-$\tau_{\leq n}$-morphisms and used different choices of framed functions, we may produce a pre-$\tau_{\leq n}$-morphism that does not admit an extension, which is why we instead take the truncation of the composition of the extensions $\cW'$ and $\cV'$.\par 
        \end{rmk}
        \begin{lem}
            Composition of $\tau_{\leq n}$-morphisms is compatible with the equivalence relation defined in Definition \ref{truncated morphisms}, and induces a well-defined bilinear map
            $$\circ: [\cG, \cH]_{\tau_{\leq n}} \otimes [\cF, \cG]_{\tau_{\leq n}} \rightarrow [\cF, \cH]_{\tau_{\leq n}}.$$
        \end{lem}
        \begin{proof}
            Identical to that of Lemma \ref{composition is well-defined} but only considering manifolds of appropriate dimensions.        
        \end{proof}
        \begin{lem}
            Composition of $\tau_{\leq n}$-morphisms is associative.
        \end{lem}
        \begin{proof}
            Identical to that of Lemma \ref{composition is associative} but only considering manifolds of appropriate dimensions.
        \end{proof}
        Similarly we may define $\tau_{\leq n}$ bilinear maps, bordisms of bilinear maps and associators, satisfying the same properties as we saw before; in particular they can be associatively composed.\par 
        We have thus shown that there is a category, $\Flow_{\tau_{\leq n}}$, with objects flow categories, and morphisms $[\cdot, \cdot]_{\tau_{\leq n}}$. Compatibility of truncations with composition defines functors
        \begin{equation}\label{eqn:truncation_functors}
            \Flow \rightarrow \Flow_{\tau_{\leq m}} \rightarrow \Flow_{\tau_{\leq n}}
        \end{equation}
        whenever $m \geq n$ which are suitably associative.\par 
        Generalising Lemma \ref{0-truncation is chain}, we conjecture the following.
    \begin{conj}\label{conj:trunc}
        The category $\Flow_{\tau_{\leq n}}$ is equivalent to the category with objects perfect $MO$-modules, and morphisms $M \rightarrow N$ given by homotopy classes of $\tau_{\leq n}MO$-linear maps $M \otimes_{MO} \tau_{\leq n}MO \rightarrow N \otimes_{MO} \tau_{\leq n} MO$, where $\tau_{\leq n}$ denotes the truncation of a spectrum with respect to the standard $t$-structure. \par 
        Furthermore this equivalence should be compatible (in an appropriate sense) with the equivalence of Conjecture \ref{Perf MO Conj}.
    \end{conj}

 Suppose $\cF$ and $\cG$ are framed flow categories, with framing stable vector spaces $V_x$ and $V_y$ for $x\in\cF$ and $y\in\cG$.  A pre- $\tau_{\leq n}$-framed morphism is a pre-$\tau_{\leq n}$ morphism for which the stable isomorphisms $st_{xy}$ from Definition \ref{def:farm} are defined and satisfy the same compatibility whenever they can i.e. whenever $|x|-|y|\leq n$.  This is a $\tau_{\leq n}$-framed morphism if it arises by truncation from a pre-$\tau_{\leq (n+1)}$-framed morphism as usual. All of the definitions and properties of $\Flow^{fr}$ are identical in this truncated setting, except that the various diagrams now only make sense when the degrees of the objects are constrained. In particular, we have the analogue
 \begin{equation}\label{eqn:framed_truncation_functors}
     {\Flow}^{fr} \rightarrow {\Flow}^{fr}_{\tau_{\leq m}} \rightarrow {\Flow}^{fr}_{\tau_{\leq n}}
 \end{equation}
for functors of truncations of framed flow categories, again associative. Entirely analogously to Conjecture \ref{conj:trunc}, we have:
    \begin{conj}\label{conj:fram trunk}
        The category $\Flow^{fr}_{\tau_{\leq n}}$ is equivalent to the category with objects perfect $\bS$-modules, and morphisms $M \rightarrow N$ given by homotopy classes of $\tau_{\leq n}\bS$-linear maps $M \otimes_{\bS} \tau_{\leq n}\bS \rightarrow N \otimes_{\bS} \tau_{\leq n} \bS$.
    \end{conj}
        
    \subsection{Obstruction theory}
        In this section, we perform some more involved homological algebra in the category of flow categories. In each case, we provide the analogous result from (ordinary) homological algebra as motivation; we expect that these should be equivalent under Conjectures \ref{conj:trunc} and \ref{conj:fram trunk}.
        \begin{thm}\label{Obstruction Theorem}
            Let $\cF$ and $\cG$ be flow categories such that either $\cF = *[i]$ or $\cG = *[i]$ for some $i$. Let $\cW: \cF \rightarrow \cG$ be a $\tau_{\leq n}$-morphism. Then there is a class
            $$[\cO] = [\cO(\cW)] \in \Hom\left(CM_{*+n+2}(\cF;\bZ/2), CM_{*}(\cG; \Omega_{n+1})\right)$$
            such that $[\cO]$ vanishes if and only if $\cW = \tau_{\leq n} \cV$, for $\cV: \cF \rightarrow \cG$ some $\tau_{\leq n+1}$-morphism.\par 
            Similarly if $\cF, \cG$ and $\cW$ are all framed, there is a class
            $$[\cO^{fr}] = [\cO^{fr}(\cW)] \in \Hom\left(CM_{*+n+2}(\cF;\bZ), CM_{*}(\cG; \Omega_{n+1}^{fr})\right)$$
            such that $[\cO^{fr}]$ vanishes if and only if $\cW = \tau_{\leq n} \cV$, for $\cV: \cF \rightarrow \cG$ some framed $\tau_{\leq n+1}$-morphism.
        \end{thm}
        We conjecture that the condition that either $\cF$ or $\cG$ is $*[i]$ is unnecessary.\par 
        We first prove the unframed case, then incorporate framings into the same proof.
        \begin{proof}[Proof without framings]
            We assume that $\cF = *[i]$ here, the case where $\cG = *[i]$ is similar.\par 
            Choose $\cW'$ an extension of $\cW$. We first construct a linear map
            $$\cO = \cO(\cW'): CM_{*+n+2}(\cF; \bZ/2) \rightarrow CM_{*}(\cG, \Omega_{n+1})$$
            We define
            $$\tilde{\cY}_{*y} = \tilde{\cY}_{*y}(\cW') := \left(\bigsqcup\limits_{w \in \cG} \cW'_{*w} \times \cG_{wy} \times [0, \eps)\right)/\sim$$
            for all $y$ in $\cG$ such that $i-|y| = n+2$, where $\sim$ is the equivalence relation defined so that as a topological space, $\tilde{\cY}_{*y}$ is the coequaliser of the diagram
            $$\bigsqcup\limits_{w, v \in \cG} \cW'_{*w} \times \cG_{wv} \times \cG_{vy} \times [0, \eps)^2_{u,v} \rightrightarrows \bigsqcup\limits_{w \in \cG} \cW'_{*w} \times \cG_{wy} \times [0, \eps)_s$$
            where the two maps send $(a,b,c,u,v)$ to $(a, \cC(b, c, u), v)$ and $(\cC(a, b, v), c, u)$ respectively. \par 
            $\tilde{\cY}_{*y}$ is a manifold with faces, with a system of faces given by
            $$\cW'_{*; y_0\ldots y_j}$$
            for $y_0,\ldots,y_j=y$ in $\cG$, and in particular a system of boundary faces given by $\cW'_{*w} \times \cG_{wy} \times \{0\}$ for $w \in \cG$. Note that all these faces are compact. \par 
            We choose framed functions $(f_{*y}, s_{*y})$ on each $\tilde{\cY}_{*y}$, noting that on all proper faces each $f_{*y}$ is negative.\par 
            Then define 
            $$\cZ_{*y} = \cZ_{*y}(\cW') := f_{*y}^{-1}\{1\}$$
            This is a closed smooth manifold of dimension $n+1$, and its bordism class $[\cZ_{*y}]$ in the bordism group $\Omega_{n+1}$ is independent of the choice of $f_{*y}$, though it does depend on the choice of extension $\cW'$ (we will use this dependence later).\par 
            We define the linear map $\cO$ by
            $$\cO(*) := \sum\limits_{i-|y|=n+2} [\cZ_{*y}] \cdot y$$
            We next show this is a chain map. This is equivalent to showing that the closed manifold
            $$\bigsqcup\limits_{s \in \cG, |s|-|y|=1} \cZ_{*s} \times \cG_{sy}$$
            represents the zero class in the bordism group $\Omega_{n+1}$ (i.e. it bounds a compact smooth manifold with boundary), for all $y$ in $\cG$ such that $i-|y|=n+3$. We note here that since a closed smooth manifold which bounds a topological manifold also bounds a smooth manifold, it would be sufficient (and easier) to show that this class is zero in the topological bordism group, but this would not adapt so well to the framed setting.\par 
            We define manifolds with faces
            $$\tilde{\cX}_{*y} := \left(\bigsqcup\limits_{w \in \cG,|w|-|y|>1} \cW'_{xw} \times \cG_{wy} \times [0, \eps)\right)/\sim$$
            for all $y$ in $\cG$ such that $i-|y|=n+3$, where $\sim$ is defined so that as a topological space, $\tilde{\cX}_{*y}$ is the coequaliser of the diagram
            $$\bigsqcup_{v, w \in \cG, |w|-|y|>1} \cW'_{*v} \times \cG_{vw} \times \cG_{wy} \times [0, \eps)^2_{u,v} \rightrightarrows \bigsqcup_{w\in \cG, |w|-|y|>1} \cW'_{xw} \times \cG_{wy} \times [0, \eps)$$
            where the two maps send $(a,b,c,u,v)$ to $(a, \cC(b, c, u), v)$ and $(\cC(a, b, v), c, u)$ respectively. Note that there are natural embeddings
            $$\tilde{\cY}_{*s} \times \cG_{sy} \hookrightarrow \tilde{\cX}_{*y}$$
            for all $s$ in $\cG$ such that $|s|-|y|=1$, and that these have disjoint images.\par 
            $\tilde{\cX}_{*y}$ has a system of faces given by
            $$\cW'_{*; y_0\ldots y_j}$$
            for $y_0, \ldots, y_j=y$ in $\cG$ such that $j > 1$ or $|y_{j-1}|-|y_j| > 1$, along with the codimension 1 faces
            $$\tilde{\cY}_{*s} \times \cG_{sy}$$
            for $s$ in $\cG$ such that $|s|-|y|=1$. Note that the only non-compact faces are the ones of the form $\tilde{\cY}_{*s} \times \cG_{sy}$.\par 
            Now choose framed functions $(g_{*y}, t_{*y})$ on each $\tilde{\cX}_{*y}$, compatible with the $(f_{*s}, s_{*s})$ on each $\tilde{\cY}_{*s} \times \cG_{sy}$. Let $\cL_{*y} = g_{*y}^{-1}\{1\}$. Then $\cL_{*y}$ is a compact smooth manifold with boundary
            $$\bigsqcup\limits_{s \in \cG, |s|-|y|=1} \cZ_{*s} \times \cG_{sy}$$
            which implies that $\cO$ is a chain map.\par 
            We now assume that $\cO$ is a null-homotopic chain map, i.e. its class in
            $$\Hom\left(CM_{*+n+2}(\cF; \bZ/2), CM_{*}(\cG, \Omega_{n+1})\right)$$
            is 0. We show that there is some other choice of extension $\cW''$ such that $\cO(\cW'') = 0$.\par 
            This assumption says that there is some linear map 
            $$\cP: CM_{*+n+1}(\cF; \bZ/2) \rightarrow CM_{*}(\cG, \Omega_{n+1})$$
            such that $\cO(*) = d \cP(*)$. More explicitly, $\cP$ sends $*$ to 
            $$\sum\limits_{y \in \cG, i-|y|=n+1} [\cP_{*y}] \cdot y$$
            for some closed manifolds $\cP_{*y}$, where $[\cP_{*y}]$ represents the bordism classes in the bordism group $\Omega_{n+1}$ of each $\cP_{*y}$, and such that
            $$\bigsqcup\limits_{s \in \cG, |s|-|y|=1} \cP_{*s} \times \cG_{sy}$$
            is cobordant to $\cZ_{*y}$ for all $y$ in $\cG$ such that $i-|y|=n+2$.\par 
            We define $\cW''_{*y} := \cW'_{*y}$ for all $y$ in $\cG$ such that $i - |y| \leq n$, and $\cW''_{*y} := \cW'_{*y} \sqcup \cP_{*y}$ for all $y$ in $\cG$ such that $i-|y|=n+1$. Then we see that 
            $$\tilde{\cY}_{*y}(\cW'') = \tilde{\cY}_{*y}(\cW') \bigsqcup\limits_{s \in \cG,|s|-|y|=1} \cP_{*s} \times \cG_{sy} \times [0, \eps)$$
            so for appropriate choices of framed functions $(f_{*x}, s_{*x})$, we see that 
            $$\cZ_{*y}(\cW'') = \cZ_{*y}(\cW')\bigsqcup\limits_{s \in \cG,|s|-|y|=1} \cP_{*s} \times \cG_{sy}$$ 
            Each of these manifolds is null-cobordant, so $\cO(\cW'') = 0$.\par 
            
            Replacing $\cW'$ with $\cW''$, we may now assume that $\cO = 0$. We build a pre-$\tau_{\leq n+2}$-morphism $\cU$ such that $\tau_{\leq n+1} \cU = \cW'$, which implies that $\cW'$ is the desired $\tau_{\leq n+1}$-morphism.\par 
            We define $\cU_{*y} = \cW'_{*y}$ for all $y$ in $\cG$ such that $i-|y| \leq n+1$. If we were to set $\cU_{*y}$ to be $\tilde{\cY}_{*y}$ for $y$ in $\cG$ such that $i-|y|=n+2$, it has the right system of faces for $\cU$ to be a pre-$\tau_{\leq n+2}$-morphism, but is non-compact. Instead consider the map $f_{*y}: \tilde{\cY}_{*y} \rightarrow \bR$. $f_{*y}^{-1}(-\infty, 1]$ is compact, and by assumption $\cZ_{*y} = f_{*y}^{-1}\{1\}$ is a closed smooth manifold which is nullcobordant, so we can glue such a nullcobordism onto $f_{*y}^{-1}(-\infty, 1]$, and set $\cU_{*y}$ to be this new manifold. Then by construction $\cU$ is a pre-$\tau_{\leq n+2}$-morphism such that $\tau_{n+1}\cU = \cW'$.\par 
            We now assume that $\cW = \tau_{\leq n}\cV$ for $\cV:\cF \rightarrow \cG$ some $\tau_{\leq n+1}$-morphism. Choose an extension $\cV'$ of $\cV$. Then when $*-|y| \leq n+2$, $\cV'_{*y}$ is a compact manifold with faces with a system of faces given by
            $$\cV'_{*;y_0\ldots y_j}$$
            for $y_0, \ldots, y_j = y$ in $\cG$. When $*-|y| = n+2$, the collar neighbourhood theorem for manifolds with faces implies that there is a neigbourhood of all the faces in $\cV'_{*y}$ diffeomorphic to $\tilde{\cY}_{*y}$. Choosing $f_{*y}$ as in the definition of $\cZ_{*y}$, 
            $$\cV'_{*y}\setminus f_{*y}^{-1}(-\infty, 1)$$
            is a compact manifold with boundary $\cZ_{*y}$. Therefore the linear map $\cO$ is 0.
        \end{proof}
        \begin{proof}[Incorporating framings]
            We incorporate framings into the unframed proof, assuming $\cF, \cG, \cW, \cW'$ are all framed.\par 
            We first note that all boundary faces of each $\tilde{\cY}_{*y}$ are broken. We frame each of these manifolds compatibly.\par 
            Inductively in $i-|*|$, we choose vector bundles $\bT^{\tilde \cR}_{*y}$ over each $\tilde \cR_{*y}$, along with, for each boundary face of $\tilde \cR_{*y}$, an embedding of the stabilising bundle over that face into $\bT^{\tilde \cR}_{*y}$, which agree on the image of the embeddings from the codimension two faces.\par 
            For each $y \in \cG$, we define the stable isomorphism
            $$st_{*y}: V_{*[i]} \oplus I^{\tilde{\cR}}_{*y} \to \bR \oplus V_y$$
            to be given by, on each $\cW'_{*w} \times \cG_{wy} \times [0, \eps)$, the composition
            $$V_{*[i]} \oplus I^{\tilde{\cR}}_{*y} \xrightarrow{\psi^{-1}} V_{*[i]} \oplus I^{\cW'}_{*w} \oplus I^\cG_{wy} \xrightarrow{st^\cG \circ st^{\cW'}} \bR \oplus V_y$$
            extended along the embedding $\iota$ into $\bT^{\cR}_{xz}$.\par 
            These glue on overlaps just as in Lemma \ref{lem:fram Q}, and similarly are compatible on boundary faces in the sense that they (combined with the other framings) form systems of framings on each $\tilde{\cY}_{*y}$.\par 
            As in Lemma \ref{lem:fram Q}, we use the sections $s_{*y}$ to then obtain stable framings 
            $$st^\cZ: V_{*[i]} \oplus I^\cZ_{*y} \to V_y$$
            on each closed manifold $\cZ_{*y}$. We choose stable isomorphisms $V_{q} \cong \bR^{|q|}$ for each $q \in \cG$ or $q=*[i]$, then this framing on each $\cZ_{*y}$ determines a class $[\cZ_{*y}]$ in $\Omega^{fr}_*$. \par 
            We then define $\cO^{fr}$ in the same way as the unframed case, but now using the fact that the manifolds in question are framed, so $\cO^{fr}$ is now a linear map 
            $$\cO^{fr} = \cO^{fr}(\cW'): CM_{*+n+2}(\cF; \bZ) \rightarrow CM_{*}(\cG, \Omega^{fr}_{n+1})$$
            Note all boundary faces of each $\tilde{\cX}_{*y}$ are broken. We frame them in exactly the same way as the $\tilde{\cY}_{*y}$, to give stable framings
            $$st^{\tilde{\cY}}_{*y}: V_{*[i]} \oplus I^{\tilde{\cX}}_{*y} \to \bR \oplus V_y$$
            which are compatible with the framings on the boundary faces in the sense that they form systems of faces over each $\tilde{\cY}_{*s}$; these framings induce framings on each $\cL_{*y}$, which implies that $\cO^{fr}$ is a chain map.\par
            The assumption says that there is some linear map 
            $$\cP^{fr}: CM_{*+n+1}(\cF; \bZ) \rightarrow CM_{*}(\cG, \Omega^{fr}_{n+1})$$
            such that $\cO^{fr}(*) = d \cP^{fr}(*)$. By replacing $\cW'$ with a different extension $\cW''$ of $\cW$ by attaching disjoint copies of $\overline{\cP}_{*y}$ to each $\cW'_{*y}$ when $i-|y|=n+1$ as in the unframed case, we may assume that $\cO^{fr} = 0$. \par 
            We take the same (unframed) pre-$\tau_{\leq n+2}$ morphism $\cU$ as in the unframed case, and by using the framings on the $\tilde{\cY}_{*y}$ and the framing of the nullbordism of each $\tilde{\cZ}_{*y}$, and gluing these framings together on the overlap, we obtain framings on $\cU$.\par 
            The claim that if $\cW = \tau_{\leq n} \cV$ for some $\tau_{\leq n+1}$ morphism $\cV: \cF \to \cG$ then the associated linear map $\cO^{fr}$ is trivial is essentially the same as in  the unframed case.
        \end{proof}
        Theorem \ref{Obstruction Theorem} should be considered an analogue in $\Flow$ (respectively $\Flow^{fr}$) of (a special case of) the following fact in homological algebra, applied to the ring spectra $MO$ (respectively $\bS$).
        \begin{prop}
            Let $R$ be a $(-1)$-connected dga or ring spectrum, and let $M$ and $N$ be objects in $R\modules$, the category of dg modules or module spectra over $R$. All tensor products we take are assumed to be derived. Let
            $$[\phi] \in \Hom_{\tau_{\leq n}R\modules }\left(M \otimes_R \tau_{\leq n}R, N \otimes_R \tau_{\leq n} R\right)$$ 
            be a morphism in the derived category of $\tau_{\leq n}R$ (here $\tau$ refers to truncation with respect to the standard $t$-structure), with some representative $\phi$ in $[\phi]$.\par 
            Then there is a class 
            $$[\cO] \in \Hom_{HR_0\modules }\left(\Sigma^{-n-2} M \otimes_R HR_0,  N \otimes_R HR_{n+1}\right)$$
            (here $HR_{i}$ is the appropriate Eilenberg-Maclane chain complex or spectrum) which vanishes if and only if $[\phi]$ is in the image of the natural map
            $$\tau_{\leq n}: \Hom_{\tau_{\leq n+1}R\modules }\left(M \otimes_R \tau_{\leq n+1}R, N \otimes_R \tau_{\leq n+1} R\right)
            \rightarrow 
            \Hom_{\tau_{\leq n}R\modules }\left(M \otimes_R \tau_{\leq n}R, N \otimes_R \tau_{\leq n} R\right)$$
        \end{prop}

        \begin{rmk} Formulating this at the level of $HR_0$-modules is analogous to working with chain complexes in Theorem \ref{Obstruction Theorem}.
        \end{rmk}
        \begin{proof}[Sketch proof]
            There is a fibration sequence in $\tau_{\leq n+1}R\modules$ (using the fact that $R$ is $(-1)$-connected here):
            $$N \otimes_R \tau_{\leq n+1} R\rightarrow N \otimes_R \tau_{\leq n} R \rightarrow N \otimes_R \Sigma^{n+2} HR_{n+1}$$
            and therefore an exact sequence
            \[\xymatrix{
                \Hom_{\tau_{\leq n+1}R\modules}\left(M \otimes_R \tau_{\leq n+1} R, N \otimes_R \tau_{\leq n+1} R\right) \ar[r] & 
                \Hom_{\tau_{\leq n+1}R\modules}\left(M \otimes_R \tau_{\leq n+1} R, N \otimes_R \tau_{\leq n} R\right) \ar[d]_{f} \\
                & \Hom_{\tau_{\leq n+1}R\modules}\left(M \otimes_R \tau_{\leq n+1} R, N \otimes_R \Sigma^{n+2} HR_{n+1}\right)
            }\]
            We identify the middle term with 
            $$\Hom_{\tau_{\leq n}R\modules}\left(M \otimes_R \tau_{\leq n}R, N \otimes_R \tau_{\leq n}R\right)$$
            the final term with
            $$\Hom_{HR_0\modules }\left(\Sigma^{-n-2} M \otimes_R HR_0,  N \otimes_R HR_{n+1}\right)$$
            and take $[\cO]$ to be $f([\phi])$.
        \end{proof}

        \begin{thm} \label{thm:onto}
            Let $\cW: \cF \to \cG$ be a morphism of flow categories such that the induced map $CM_*(\cW; \bZ/2): CM_*(\cF; \bZ/2) \to CM_*(\cG; \bZ/2)$ is a quasiisomorphism over $\bZ/2$.\par 
            Then the induced map $\cW_*: \left[*[u], \cF\right] \to \left[*[u], \cG\right]$ is surjective for all $u$.\par 
            Similarly if $\cF, \cG, \cW$ are all framed and $CM_*(\cW; \bZ): CM_*(\cF; \bZ) \to CM_*(\cG, \bZ)$ is a quasiisomorphism over $\bZ$, then the induced map $\cW_*: [*[u], \cF]^{fr} \to [*[u], \cG]^{fr}$ is surjective for all $u$.
        \end{thm}
        We conjecture (but do not require for our purposes) that the same statement holds for injectivity as well as surjectivity.
        \begin{proof}[Proof without framings]
            We show that $\cW_*: \left[*[u], \cF\right]_{\tau_{\leq v}} \to \left[*[u], \cG\right]_{\tau_{\leq v}}$ is surjective for each $v$ inductively; this holds for $v=0$ by assumption. Assume this holds for $v$; we will show it holds for $v+1$.\par 
            Let $\cB: *[u] \to \cG$ be a $\tau_{\leq v+1}$-morphism. By assumption, $\tau_{\leq v}\cB$ lifts to a $\tau_{\leq v}$-morphism $\cA: *[u] \to \cF$, and there is a bordism $\cR$ from $\cW \circ \cA$ to $\tau_{\leq v}\cB$.\par 
            Much more explicitly, we first choose an extension $\cB'$ of $\cB$, and assume that there are compact manifolds with faces $\cA'_{*x}$ for each $x$ in $\cF$ with $u-|x|\leq v+1$ of dimension $u-|x|$, each of which has a system of faces given by
            $$\cA'_{*;x_0 \ldots x_i}$$
            for $x_0, \ldots, x_i =x$ in $\cF$ (and in particular a system of boundary faces given by $\cA'_{*;x'x}$ for $x'$ in $\cF$), as well as compact manifolds with faces $\cR'_{*y}$ for each $y$ in $\cG$ with $u-|y|\leq v$ of dimension $u-|y|+1$, with a system of faces given by
            $$\cB_{*;y_0 \ldots y_j}$$
            for $y_0, \ldots, y_j = y$ in $\cG$,
            $$\cA'_{*; x_0 \ldots x_i} \times \cW_{x_i; y_0 \ldots y_j}$$
            for $x_0, \ldots, x_i$ in $\cF$ and $y_0, \ldots, y_j$ in $\cG$, and
            $$\cR'_{*;y_0 \ldots y_j}$$
            for $y_0, \ldots, y_j = y$ in $\cG$. In particular, $\cR'_{*y}$ has a system of boundary faces given by $\cB_{*y}$, $\cA_{*x} \times \cW_{xy}$ for $x$ in $\cF$, and $\cR'_{*; y'y}$ for $y'$ in $\cG$.\par 
            We next define an element $(a, r)$ in 
            $$CM_{u-v-2}\left(\cF; \Omega_{v+1} \right) \oplus CM_{u-v-1}\left(\cG; \Omega_{v+1}\right)$$
            For $x$ in $\cF$ with $u-|x|=v+2$, we define $\tilde{\cD}_{*x}$ to be the coequaliser of the following.
            $$\bigsqcup\limits_{x',x''\in\cF} \cA'_{*;x''x'x} \times [0, \eps)^2 \rightrightarrows \bigsqcup\limits_{x'\in\cF} \cA'_{*;x'x}\times [0, \eps)$$
            By Lemma \ref{lem:abstract gluing with faces}, this is a manifold with faces, with a system of faces given by
            $$\cA'_{*;x_0 \ldots x_i} \times \{0\}$$
            for $x_0, \ldots, x_i=x$ in $\cF$. In particular, all these faces are compact.\par 
            Choose framed functions $(f_{*x}, s_{*x})$ on each $\tilde{\cD}_{*x}$, and let $\cS_{*x} := f_{*x}^{-1}\{1\}$. Now define $a$ by
            $$a := \sum\limits_{\substack{x \in \cF\\
            u-|x|=v+2}} [\cS_{*x}] \cdot x \in CM_{u-v-2}\left(\cF; \Omega_{v+1}\right)$$
            Note that $a$ is a specific instance of the obstruction chain (written there as $\cO$) constructed in (the proof of) Theorem \ref{Obstruction Theorem}, and therefore has already been shown to be a closed chain.\par 
            For $y$ in $\cG$ with $u-|y|=v+1$, we define $\tilde{\cE}_{*y}$ to be the coequaliser of the following diagram.
            \[\xymatrix{
                \bigsqcup\limits_{y' \in \cG} \cB_{*y'} \times \cG_{y'y} \times [0, \eps)\times [0, -\eps) \ar[r] \ar[dr] &
                \cB_{*y} \times [0, -\eps)\\
                \bigsqcup\limits_{y',y'' \in \cG} \cR'_{*;y''y'y} \times [0, \eps)^2 \ar@<+.5ex>[r] \ar@<-.5ex>[r] &
                \bigsqcup\limits_{y'\in \cG} \cR'_{*;y'y} \times [0, \eps) \\
                \bigsqcup\limits_{\substack{x \in \cF\\ y' \in \cG}} \cA'_{*x} \times \cW_{x;y'y} \times [0, \eps)^2 \ar[ur] \ar[r] &
                \bigsqcup\limits_{x \in \cF} \cA'_{*x} \times \cW_{xy} \times [0, \eps) \\
                \bigsqcup\limits_{x, x' \in \cF} \cA_{*;x'x} \times \cW_{xy} \times [0, \eps)^2 \ar@<+.5ex>[ur] \ar@<-.5ex>[ur]  
            }
            \]
            Note the copies of $[0,-\eps)$ in some places instead of $[0, \eps)$; this will be relevant when we work in the framed setting. Here the two maps out of each $\cB_{*y'} \times \cG_{y'y} \times [0,\eps) \times [0,-\eps)$ send $(a,b,u,v)$ to $\alpha(a,b,u,v) := (\cC(a,b;u),v)$ and $\beta(a,b,u,v) := (\cC(a;-v),b,u)$ respectively.\par 
            Each $\tilde{\cE}_{*y}$ is a manifold with corners by Lemma \ref{lem:abstract gluing} (noting $[0,-\eps)$ is naturally diffeomorphic to $[0,\eps)$), and by applying Lemma \ref{lem:abstract gluing with faces} (using the indexing sets $I = \{*\} \cup \cG \cup \cF$ and $J = \emptyset$), $\tilde{\cE}_{*y}$ is a manifold with faces, with a system of faces given by
            $$\cB_{*; y_0 \ldots y_j} \times \{0\}$$
            for $y_0, \ldots, y_j = y$ in $\cG$,
            $$\cR'_{*; y_0 \ldots y_j} \times \{0\}$$
            for $j \geq 1$ and $y_0, \ldots, y_j = y \in \cG$, and
            $$\cA'_{*; x_0 \ldots x_i} \times \cW_{x_i; y_0 \ldots y_j} \times \{0\}$$
            for $x_0, \ldots, x_i$ in $\cF$ and $y_0, \ldots, y_j = y$ in $\cG$. In particular, all these faces are compact.\par 
            Choose framed functions $(f_{*y}, s_{*y})$ on each $\tilde{\cE}_{*y}$, and let $\cT_{*y} = f_{*y}^{-1}\{1\}$. Now define $r$ by 
            $$r := \sum\limits_{\substack{y \in \cG\\u-|y| = v+1}}\left[\cT_{*y}\right] \cdot y \in CM_{u-v-1}\left(\cG; \Omega_{v+1}\right)$$
            We next show that $(a, r)$ is closed with respect to the mapping cone differential: explicitly, that $da = 0$ and $CM_*(\cW)(a) + dr = 0$. We have already seen that $da = 0$.\par 
            We next define several auxiliary manifolds. For $y$ in $\cG$ with $u-|y|\leq v+1$, we define $\tilde{\cK}_{*y}$ to be the coequaliser of the following diagram.
            \[\xymatrix{
                \bigsqcup\limits_{x,x' \in \cF} \cA'_{*;x'x} \times \cW_{xy} \times [0, \eps)^2 \ar@<+.5ex>[r] \ar@<-.5ex>[r] &
                \bigsqcup\limits_{x \in \cF} \cA'_{*x} \times \cW_{xy} \times [0, \eps)
            }
            \]
            By Lemma \ref{lem:abstract gluing with faces} (with $I=\cF$ and $J = \cG \setminus\{y\}$), $\tilde{\cK}_{*y}$ is a manifold with faces, with a system of faces given by
            $$\cA_{*; x_0 \ldots x_i} \times \cW_{x_i; y_0 \ldots y_j} \times \{0\}$$
            for $x_0, \ldots, x_i$ in $\cF$ and $y_0, \ldots, y_j = y$ in $\cG$, and
            $$\tilde{\cK}_{*; y_0 \ldots y_j}$$
            for $y_0, \ldots, y_j=y$ in $\cG$.\par
            For $y$ in $\cG$ with $u-|y|=v+2$, we define $\tilde{\cJ}_{*y}$ to be the coequaliser of the following diagram.
            \[\xymatrix{
                \bigsqcup\limits_{y',y'' \in \cG} \tilde{\cK}_{*; y''y'y} \times [0, \eps)^2 \ar@<+.5ex>[r] \ar@<-.5ex>[r] &
                \bigsqcup\limits_{y'\in \cG}\tilde{\cK}_{*; y'y} \times [0, \eps)
            }
            \]
            By Lemma \ref{lem:abstract gluing}, this is a manifold with corners.\par 
            For $y$ in $\cG$ with $u-|y|=v+2$, we define $\tilde{\cL}_{*y}$ to be the coequaliser of the following diagram.
            \[\xymatrix{
                \bigsqcup\limits_{\substack{x,x' \in \cF\\
                |x|-|y|>0}} \cA'_{*; x'x} \times \cW_{xy} \times [0, \eps)^2 \ar@<+.5ex>[r] \ar@<-.5ex>[r] &
                \bigsqcup\limits_{\substack{x \in \cF\\
                |x|-|y|>0}} \cA'_{*x} \times \cW_{xy} \times [0, \eps)
            }
            \]
            By Lemma \ref{lem:abstract gluing with faces} (applied with $I = \left\{x \in \cF \,|\, |x|-|y|>0\right\}$ and $J = \left\{ x\in\cF \,|\, |x|-|y|=0\right\} \cup \left(\cG \setminus \{y\}\right)$), $\tilde{\cL}_{*y}$ is a manifold with faces, with a system of faces given by
            $$\cA_{*; x_0 \ldots x_i} \times \cW_{x_i; y_0 \ldots y_j} \times \{0\}$$
            when for $x_0, \ldots, x_i$ in $\cF$ and $y_0, \ldots, y_j$ in $\cG$, when either $i+j>0$ or $i=j=0$ and $|x_0|-|y_0|>0$, and
            $$\tilde{\cD}_{*x} \times \cW_{xy}$$
            for $x$ in $\cF$ with $|x|-|y|=0$, and
            $$\tilde{\cK}_{*;y_0 \ldots y_j}$$
            for $y_0, \ldots, y_j=y$ in $\cG$ when $j \geq 1$.\par
            For $y$ in $\cG$ with $u-|y|=v+2$, we define $\tilde{\cM}_{*y}$ to be the coequaliser of the following diagram.
            \[\xymatrix{
                \bigsqcup\limits_{\substack{y' \in \cG\\
                |y'|-|y|>1}} \cB_{*;y'y} \times [0, \eps)\times [0,-\eps) \ar[r] \ar[dr] &
                \cB'_{*y} \times [0, -\eps)\\
                \bigsqcup\limits_{\substack{y',y'' \in \cG\\
                |y'|-|y|>1}} \cR'_{*;y''y'y} \times [0, \eps)^2 \ar@<+.5ex>[r] \ar@<-.5ex>[r] &
                \bigsqcup\limits_{\substack{y'\in \cG\\
                |y'|-|y|>1}} \cR'_{*;y'y} \times [0, \eps)\\
                \bigsqcup\limits_{\substack{x\in \cF; y' \in \cG\\
                |y'|-|y|>1}} \cA'_{*x} \times \cW_{x;y'y} \times [0, \eps)^2 \ar[ur] \ar[r] &
                \bigsqcup\limits_{\substack{x \in \cF\\
                |x|-|y|>0}} \cA'_{*x} \times \cW_{xy} \times [0, \eps)\\
                \bigsqcup\limits_{\substack{x,x'\in \cF\\
                |x|-|y|>0}} \cA_{*;x'x} \times \cW_{xy} \times [0, \eps)^2 \ar@<+.5ex>[ur] \ar@<-.5ex>[ur] &
            }
            \]
            This is a manifold with corners by Lemma \ref{lem:abstract gluing}; we see (either by iteratively applying Lemma \ref{lem:abstract gluing with faces} or by direct inspection of the corner strata) that $\tilde{\cM}_{*y}$ is a manifold with faces, with a system of faces given by
            $$\cB_{*; y_0 \ldots y_j} \times \{0\}$$
            for $y_0, \ldots, y_j = y$ in $\cG$,
            $$\cR'_{*; y_0 \ldots y_j} \times \{0\}$$
            for $y_0, \ldots, y_j = y$ in $\cG$ when either $j>1$ or $j=0$ and $|y_0|-|y_1|>1$, and
            $$\cA'_{*; x_0 \ldots x_i} \times \cW_{x_i; y_0 \ldots y_j} \times \{0\}$$
            for $x_0, \ldots, x_i$ in $\cF$ and $y_0, \ldots, y_j$ in $\cG$, when either $i+j>0$ or $i=j=0$ and $|x_0|-|y_0|>0$, and
            $$\tilde{\cK}_{*; y_0 \ldots y_j}$$
            for $y_0, \ldots, y_j$ in $\cG$ with $j \geq 1$, and
            $$\tilde{\cE}_{*;y'y}$$
            for $y'$ in $\cG$ with $|y'|-|y|=1$.\par 
            Both $\tilde{\cL}_{*y}$ and $\tilde{\cM}_{*y}$ contain faces of the form $\tilde{\cK}_{*; y_0 \ldots y_j}$ (for $y_0, \ldots, y_j$ in $\cG$ with $j \geq 1$), and in both cases these faces are non-compact but their intersections with the other boundary faces not of this form are always compact. By the collar neighbourhood theorem for manifolds with faces, there are embeddings $\tilde{\cJ}_{*y} \hookrightarrow \tilde{\cL}_{*y}, \tilde{\cM}_{*y}$ onto neighbourhoods of all the faces of the form $\tilde{\cK}_{*; y_0 \ldots y_j}$.\par 
            Pick a framed function $(h_{*y}, w_{*y})$ on $\tilde{\cJ}_{*y}$, and remove $h_{*y}^{-1}(-\infty, 1)$ from $\tilde{\cL}_{*y}$ and $\tilde{\cM}_{*y}$ to obtain $\tilde{\cL}^r_{*y}$ and $\tilde{\cM}^r_{*y}$ respectively. These are still manifolds with faces, but now they both have a boundary face identified with $h_{*y}^{-1}\{1\}$, which is a compact smooth manifold with faces.\par 
            Now glue $\tilde{\cL}^r_{*y}$ to $\tilde{\cM}^r_{*y}$ along this common boundary face to obtain a new manifold with faces $\cV_{*y}$. Its non-compact boundary faces are given by $\tilde{\cD}_{*x} \times \cW_{xy}$ for $x$ in $\cF$ with $|x|-|y|=1$ (coming from $\tilde{\cL}_{*y}^r$), and also other ones which are diffeomorphic to $\tilde{\cE}_{*;y'y}$ away from a neighbourhood of the compact faces of this face, for $y'$ in $\cG$ with $|y'|-|y|=1$ (coming from $\tilde{\cM}^r_{*y}$).\par 
            For each $y$ in $\cG$ with $u-|y|=v+2$, pick framed functions $(g_{*y}, t_{*y})$ on $\cV_{*y}$, agreeing with the framed functions $(f_{*x}, s_{*x})$ and $(f_{*y}, s_{*y})$ on appropriate faces.\par 
            Then $g_{*y}^{-1}\{1\}$ is a nullbordism of the manifold representing the $y$-coefficient of $CM_*(\cW)(a) + dr$, so this chain must vanish, completing the proof that $(a, r)$ is closed.\par 
            We proceed similarly to the proof of Theorem \ref{Obstruction Theorem}. Since $CM_*(\cW)$ is a quasiisomorphism (over $\bZ/2\bZ$, which implies it is also one over $\Omega_{v+1}$), its cone is acyclic, so there is some $(a', r')$ in 
            $$CM_{u-v-1}\left(\cF; \Omega_{v+1}\right) \oplus CM_{u-v}\left(\cG; \Omega_{v+1}\right)$$
            such that $da' = a$ and $CM_*(\cW)(a') + dr' = r$.\par 
            Replacing each $\cA'_{*x}$ (for $u-|x|=v+1$) with $\cA'_{*x} \sqcup a'_x$ (where $a'_x$ is a choice of some representative of the $x$-coefficient of $a$) (so the new $\cA'$ is another extension of $\cA$) and $\cR'_{*y}$ (for $u-|y| = v$) with $\cR'_{*y} \sqcup r'_y$, we may assume that the chain $(a, r)$ is zero, i.e. that each $\cS_{*x}$ and each $\cT_{*y}$ is nullbordant.\par 
            We now define extensions of $\cA'$ and $\cR'$: explicitly, we define compact manifolds with faces $\cA'_{*x}$ for $x$ in $\cF$ with $u-|x| = v+2$ and $\cR'_{*y}$ for $y$ in $\cG$ with $u-|y| = v+1$, with systems of faces as described earlier.\par 
            For $x$ in $\cF$ with $u-|x|=v+2$, let $\cA''_{*x} = f_{*x}^{-1}(-\infty, 1]$. Since $\cS_{*x} = f_{*x}^{-1}\{1\}$ is closed and nullbordant, we glue such a nullbordism onto this boundary face to obtain $\cA'_{*x}$.\par 
            For $y$ in $\cG$ with $u-|y| = v+1$, we let $\cR''_{*y} = f_{*y}^{-1}(-\infty, 1]$. Since $\cT_{*y} = f_{*y}^{-1}\{1\}$ is closed and nullbordant, we glue such a nullbordism onto this boundary face to obtain $\cR'_{*y}$, completing the inductive step.
        \end{proof}
        \begin{proof}[Incorporating framings]
            Similarly to the proof of Theorem \ref{thm:onto}, we coherently frame all manifolds arising in the unframed case to prove the framed case. We assume $\cA', \cB',\cR'$ are all framed, and we recall that each $\cB'_{*y}$ is an \emph{outgoing} boundary face of each $\cR'_{*y}$. We first equip each $\tilde{\cD}_{*x}$, $\tilde{\cK}_{*y}$, $\tilde{\cJ}_{*y}$, $\tilde{\cL}_{*y}$ with compatible systems of framings in the same way as we did with $\tilde{\cQ}(\ldots)$ in the proof of Lemma \ref{lem:fram Q}; these then induce compatible systems of framings on each $\cS_{*x}$ and $\tilde{\cL}^r_{*y}$ in the usual way. \par 
            We next frame each $\tilde{\cE}_{*y}$. The only unbroken boundary face of $\tilde{\cE}_{*y}$ is $\cB_{*y} \times \{0\}$; we declare this to be outgoing. \par 
            We first choose, inductively in $i-|y|$, stabilising bundles $\bT^{\tilde \cE}_{*y}$ over each $\tilde \cE_{*y}$, along with embeddings of the stabilising bundles of the stable framing of each boundary face of each $\tilde \cE_{*y}$ into $\bT^{\tilde \cE}_{*y}$, which agree on their overlaps.\par 
            
            We now define $st_{*y}^{\tilde{\cE}}: V_{*[u]} \oplus I^{\tilde{\cE}}_{*y} \to \bR^2 \oplus V_y$ as follows. Over each $\cB'_{*y} \times [0,-\eps)$, we define $st^{\tilde{\cE}}$ to be the composition
            $$V_{*[u]} \oplus I^{\tilde{\cE}}_{*y} \xrightarrow{\psi^{-1}} V_{*[u]} \oplus I^{\cB'}_{*y} \oplus \bR \sigma \xrightarrow{st^{\cB'}} \bR \oplus \bR\sigma \oplus V_y$$
            extended to have stabilising bundle $\bT^{\tilde \cE}_{*y}$, where the first copy of $\bR$ comes from $st^{\cB}$. On $\cR'_{*y'} \times \cG_{y'y} \times [0, \eps)$, we define $st^{\tilde{\cE}}$ to be the composition
            $$V_{*[u]} \oplus I^{\tilde{\cE}} \xrightarrow{\psi^{-1}} V_{*[u]} \oplus I^{\cR'}_{*y'} \oplus I^\cG_{y'y} \xrightarrow{st^\cG \circ st^{\cR'}} \bR^2 \oplus V_y$$
            and on $\cA'_{*x} \times \cW_{xy} \times [0,\eps)$, we define $st^{\tilde{\cE}}$ to be the composition
            $$V_{*[u]} \oplus I^{\tilde{\cE}}_{*y} \xrightarrow{\psi^{-1}} V_{*[u]} \oplus I^{\cA'}_{*x} \oplus I^\cW_{xy} \xrightarrow{st^\cW \circ st^{\cA'}} \bR \oplus \bR \oplus V_y$$
            again extended to have stabilising bundle $\bT^{\tilde \cE}_{*y}$, where the first copy of $\bR$ in the final term comes from $st^\cW$ and the second from $st^{\cA'}$.\par 
            We need these to agree on the regions where they overlap. Over all regions except those of the form $\cB_{*y'} \times \cG_{y'y} \times [0,\eps)_u \times [0,-\eps)_v$, this is identical to checking the gluing on overlaps in Lemma \ref{lem:fram Q}; on the region $\cB_{*y'} \times \cG_{y'y} \times [0,\eps)_u \times [0,-\eps)_v$, it follows from a similar argument, using the following diagram.
            \[
            \centerline{\xymatrix{
                V_{*[u]} \oplus T \cB'_{*y} \oplus \bR \partial_s \oplus \bR \tau_y \ar[d]_{\bar{\psi}} &
                V_{*[u]} \oplus \bR \partial_u \oplus \bR \partial_v \oplus T\cB'_{*y'} \oplus T\cG_{y'y} \oplus \bR\tau_y \ar[l]_{d\alpha} \ar[r]_{d\beta} \ar[d]_{=} & 
                V_{*[u]} \oplus T\cR'_{*y'} \oplus \partial_s \oplus T\cG_{y'y} \oplus \bR \tau_y \ar[d]_{\bar{\psi}}\\
                V_{*[u]} \oplus I^{\cB'}_{*y} \oplus \bR \sigma \ar[r]_\psi &
                V_{*[u]} \oplus I^{\tilde{\cE}}_{*y} &
                V_{*[u]} \oplus I^{\cR'}_{*y'} \oplus I^\cG_{y'y} \ar[l]_\psi
            }
            }
            \]
            Here the $\bar{\psi}$ on the left is defined to send $\partial_s$ to $\sigma$ and is the identity on all other factors, and the $\bar{\psi}$ on the right is defined to send $\partial_s$ to $\tau_{y'}$ and is the identity on all other factors. Then we can see directly that this diagram commutes since all three ways of going from the top middle entry to the bottom middle entry send $\tau_y$ to $\tau_y$, $\partial_u$ to $\nu_{y'}$ and $\partial_v$ to $-\nu^\cB$ (and are the identity on all other factors).\par 
            The $st^{\tilde{\cE}}$ (along with the other framings) assemble to form systems of framings for all $\tilde{\cE}_{*y}$, and similarly we may frame all $\tilde{\cM}_{*y}$ in exactly the same way. These framings then induce framings on each $\cT_{*y}$ and each $\tilde{\cM}^r_{*y}$ in the usual way.\par 
            The framings on $\cS_{*x}$ and $\cT_{*y}$ gives us a well-defined element 
            $$(a^{fr}, r^{fr}) \in CM_{u-v-2}\left(\cF; \Omega_{v+1}^{fr}\right) \oplus CM_{u-v-1}\left(\cG; \Omega^{fr}_{v+1} \right)$$
            We have already seen that $da^{fr}=0$. We would like to show that the pair $(a^{fr}, r^{fr})$ is closed; in this case, an identical argument to before allows us to complete the induction step.\par 
            It therefore remains to show that $\cW_*(a^{fr}) - d(r^{fr}) = 0$. To do this we next equip $\cV_{*y}$ with appropriate framings. Note that both $\tilde{\cL}^r_{*y}$ and $\tilde{\cM}^r_{*y}$ have $h^{-1}_{*y}\{1\}$ as an incoming face, so we can frame $\cV_{*y}$ by equipping it with the framing of $\tilde{\cL}^r_{*y}$ over $\tilde{\cL}^r_{*y}$ and with the opposite framing of $\tilde{\cM}^r_{*y}$ over $\tilde{\cM}^r_{*y}$. Then $g^{-1}$ is a framed nullbordism of
            $$\bigsqcup\limits_{x \in \cF} \cS_{*x} \times \cW_{xy} \bigsqcup_{y' \in \cG} \overline{\cT}_{*;y'y}$$
            from which the required closedness follows.
        \end{proof}
        The analogous algebraic result is the following corollary of Hurewicz' and Whitehead's theorems:
        \begin{prop}
            Let $R$ be a $(-1)$-connected dga or ring spectrum, and $f: M \to N$ a morphism in $R\modules$. If the induced map
            $$f: M \otimes_R HR_0 \to N \otimes_R HR_0$$
            is an equivalence of $HR_0$-modules, then $f$ is an equivalence of $R$-modules.
        \end{prop}

\section{Baas-Sullivan singularities}\label{Sec:BS}
    For each $k>0$, we fix a set $\bM^{k}$ of closed $k$-manifolds $P$, each equipped with a stable framing $st^P: TP \to \bR^k$ with stabilising bundle $\bT^P$, containing a representative of each isomorphism (i.e. diffeomorphism respecting stable framings) class of stably framed $k$-manifolds. 
    \subsection{Singularities}
        \begin{defn}\label{kBS sing def}
            A \emph{manifold with faces with $k$-Baas-Sullivan singularities} (abbreviated to \emph{with $\kBS$-singularities}, or just \emph{$\kBS$-manifold (with faces)}) $(M, N_P, \iota)$ (often just abbreviated to $M$) consists of a manifold with faces $M$, along with manifolds with faces $N_P$ for each $P \in \bM^k$, such that only finitely many $N_P$ are non-empty, and $\iota$ is a proper embedding of a boundary face
            $$\iota: \bigsqcup\limits_{P \in \bM^k} N_P \times P \hookrightarrow M$$
            We call faces of $M$ which have no components lying in the image of $\iota$ \emph{essential faces} and faces of $M$ which lie in the image of $\iota$ \emph{singular faces}. We write $D_M$ for the set of essential faces of $M$, and if $F \subseteq M$ is an essential face, we write $D_{F \subseteq M}$ for the set of essential faces containing $F$. Note that any essential face $F$ of $M$ naturally inherits the structure of a $\kBS$-manifold with faces. In particular, if $(M',N'_P,\iota')$ is another $\kBS$-manifold with faces, for $M' \hookrightarrow M$ to be an inclusion of a face, there are inclusions of faces of non-singular manifolds $N_P'\hookrightarrow N_P$ such that the following diagram commutes:
            \[\xymatrix{
                \bigsqcup\limits_P N_P'\times P \ar[r] \ar[d]_{\iota'} &
                \bigsqcup\limits_P N_P \times P \ar[d]_\iota \\
                M' \ar[r] &
                M
            }
            \]
            Note that in particular, even if $P,P' \in \bM^k$ are diffeomorphic, we do not treat them as the same singularity type.\par
            A \emph{closed $\kBS$-manifold} is a $\kBS$-manifold with faces which has no essential faces.\par
            A \emph{non-singular} $\kBS$-manifold is one where all $N_P$ are empty; this is just an ordinary manifold.\par 
            A \emph{system of essential faces} for $M$ is a collection of faces $F_i^j$ is a collection of essential faces $F_i^j$ of $M$ such that each $F_i^j$ has codimension $j$, all $F_i^j$ have disjoint smooth loci, and any (positive codimension) essential face of $M$ is a disjoint union of components of the $F_i^j$.\par 
            Let $X$ be some other space. We say a continuous function $f: M \to X$ is \emph{compatible with the singularities} if, on each $\iota(N_P \times P)$, $f$ is constant in the $P$ direction, i.e. it only depends on the $N_P$-co-ordinate.
        \end{defn}
        \begin{rmk}
            One should think of such a manifold as a singular manifold with singular loci locally modelled on cones over $k$-manifolds. Indeed taking $(M, N_P, \iota)$ a $\kBS$-manifold with faces and attaching cells $N_P \times C_{P}$ to each $N_P \times P$ (where $C_{P}$ is the cone over $P$) gives such a singular manifold. Definition \ref{kBS sing def} can be thought of as removing neighbourhoods of the singular loci but keeping track of how to glue them back in; this definition is easier to work with in practice.\par 
            In particular, we think of the singular faces as part of a ``singular locus'' rather than actual faces of the manifold, whose role the essential faces play.
        \end{rmk}
        \begin{lem}
            Let $M$ be a $\kBS$-manifold with faces, $G$ a face of $M$, and $F$ a face of $G$. Then if $G$ is a singular face, $F$ is too. \par 
            It follows from this that if $F$ is an essential face, $G$ is too.
        \end{lem}
        \begin{proof}
            If $G$ lies in the image of $\iota$ and $F \subset G$ then $F$ lies in the image of $\iota$.
        \end{proof}
        \begin{rmk}\label{dim nonsing}
            A $\kBS$-manifold with faces of dimension $\leq k$ is the same as a manifold with faces.
        \end{rmk}
        \begin{defn}
            A framed function $(f, s)$ on a $\kBS$-manifold $(M, N_P, \iota)$ is a framed function on the manifold with corners $M$ such that over each $N_P \times P$, $f$ is constant in the $P$-direction (so $f$ is compatible with the singularities as a continuous function), and $s$ is also constant in the $P$-direction (noting that since $f$ is constant in the $P$-direction, $f^{-1}\{1\} \cap (N_P \times P)$ naturally splits into a product of something with $P$).
        \end{defn}
        The same proofs show that Lemma \ref{prop: res of fram fn is fram} (respectively Lemma \ref{prop: ext fram fn}) holds for framed functions on $\kBS$-manifolds, where we restrict to (respectively extend from) only essential faces.\par 
        From the preimage theorem for manifolds with faces, we see that
        \begin{prop}\label{preimage thm sing}
            Let $(M, N_P, \iota)$ be a $\kBS$-manifold with face with a system of essential faces $\{F_i\}$, and $(f, s)$ a framed function on $M$. Then $f^{-1}\{1\}$ and $f^{-1}(-\infty, 1]$ are naturally $\kBS$-manifolds with faces.\par 
            There is a system of essential faces for $f^{-1}\{1\}$ given by the $F_i \cap f^{-1}\{1\}$, and one for the $f^{-1}(-\infty,1]$ given by the $F_i \cap f^{-1}\{1\}$ and the $F_i \cap f^{-1}(-\infty, 1]$.
        \end{prop}
        
        \begin{lem} \label{prod sing}
            Let $M$ be a $\kBS$-manifold with faces, and $N$ a manifold with faces. Then $M \times N$ and $N \times M$ each canonically admit the structure of a manifold with faces with $\kBS$ singularities.
        \end{lem}
        \begin{proof}
            $M \times N$ is a manifold with faces, and admits a canonical proper embedding of a codimension 1 face
            $$\iota \times \id_N: \bigsqcup\limits_P (N_P\times N) \times P \hookrightarrow M \times N$$
            where the $N_P$, $\iota$ are as in Definition \ref{kBS sing def}. The case of $N \times M$ is the same.
        \end{proof}
        \begin{cor} \label{prod sing2}
            Let $M$ and $N$ be $\kBS$-manifolds with faces, of dimensions $m$ and $n$ respectively. If either $m \leq k$ or $n \leq k$, in particular if $2k > m + n$, then $M \times N$ is naturally a $\kBS$-manifold with faces.
        \end{cor}
        \begin{proof}
            By Remark \ref{dim nonsing} one of $M$ and $N$ is non-singular, and so by Lemma \ref{prod sing}, their product is a $\kBS$-manifold with faces. 
        \end{proof}
        We have the following abstract gluing lemma, similarly to Lemma \ref{lem:abstract gluing} (but now with essential faces playing the role of faces).\par 
        Let $(M_i, Q_{i,P}, \iota_i)$ be $\kBS$-manifolds (so each $\iota_i: \sqcup_P Q_{i,P} \times P \hookrightarrow M_i$ is the inclusion of a boundary face). Let $\alpha_w: N_w \hookrightarrow M_{i_w}$ and $\beta_w: N_w \hookrightarrow M_{j_w}$ be inclusions of essential boundary faces; each of these induces the structure of a manifold with $\kBS$-singularities on $N_w$, and we assume these structures agree. We write $(N_w, R_{w,P}, \iota_w)$ for this structure.\par
        In particular, we have diffeomorphisms $R_{w,P} \cong \partial Q_{i_w,P} \cong \partial Q_{j_w,P}$.\par  
        We now have inclusions of boundary faces $\alpha_w: R_{w,P} \hookrightarrow Q_{i,P}$ and $\beta_w: R_{w,P} \hookrightarrow Q_{j,P}$ as in the hypothesis of Lemma \ref{lem:abstract gluing}; for each $P$ we let $Y_P$ be the coequaliser of the diagram
        $$\bigsqcup\limits_w R_{w,P} \times [0, \eps)^2 \rightrightarrows \bigsqcup\limits_i Q_{i,P} \times [0, \eps)$$
        which is a manifold with corners by Lemma \ref{lem:abstract gluing}.\par 
        Let $Z$ be constructed as in Lemma \ref{lem:abstract gluing}. There are natural disjoint embeddings $\iota_j: Y_P \times P \hookrightarrow Z$ given by gluing together the $\iota_i$.
        \begin{lem}
            $(Z, Y_P, \iota:= \sqcup_P \iota_P)$ is a manifold with corners with $\kBS$-singularities.
        \end{lem}
        \begin{proof}
            Lemma \ref{lem:abstract gluing} shows that $Z$ is a manifold with corners. Equipping this with the map $\iota$ defines a system of singular faces which makes this a manifold with corners with $\kBS$-singularities. 
        \end{proof}
        Furthermore if we can index the $M_i$ and $N_w$ as in Lemma \ref{lem:abstract gluing with faces}, the statement of Lemma \ref{lem:abstract gluing with faces} still holds for a system of essential faces, with the same proof.
        \begin{defn}\label{morphism with kBS sing}
            Let $\cF$ and $\cG$ be flow categories. A \emph{$\kBS$-pre-$\tau_{\leq n}$-morphism} $\cW$ from $\cF$ to $\cG$ is the same as a $pre$-$\tau_{\leq n}$-morphism except all manifolds are allowed to have $\kBS$-singularities, and all systems of faces are essential systems of faces.\par 
            More explicitly, this consists of compact $\kBS$-manifolds with faces, say  $\cW_{xy}=(\cW_{xy},N_{xy,P},\iota_{xy})$ (of dimension $|x|-|y|$) for all $x$ in $\cF$ and $y$ in $\cG$ such that $|x|-|y| \leq n$, along with maps
            $$c = c^\cW: \cF_{xz} \times \cW_{zy} \rightarrow \cW_{xy}$$
            and
            $$c = c^\cW: \cW_{xz} \times \cG_{zy} \rightarrow \cW_{xy}$$
            which are compatible with $c^\cF$ and $c^\cG$, which define a system of essential faces
            $$\cW_{x_0 \ldots x_i; y_0 \ldots y_j}$$
            for $x_0=x, \ldots, x_i$ in $\cF$ and $y_0, \ldots, y_j = y$ in $\cG$.\par 
            We define \emph{$\kBS$-$\tau_{\leq n}$-pre-bordisms} and \emph{$\kBS$-$\tau_{\leq n}$-bordisms} analogously.\par 
            There are similar definitions of flow categories with $\kBS$-singularities, bilinear maps, associators etc., but we will not use them.
        \end{defn}
        \begin{rmk}\label{rmk:dimension_constraint_2}
            Note that since the flow categories $\cF$ and $\cG$ don't have any singularities, the products $\cW_{x_0 \ldots x_i;y_0 \ldots y_j}$ admit the structure of manifolds with $\kBS$ singularities, by Lemma \ref{prod sing2}. In practice we will only use these definitions when $2k-1 > n$, so that we can compose these: if we did not have this inequality, then when one attempts to compose such things, eventually one will have to take a product of two manifolds with non-empty singular faces, which does not have $\kBS$ singularities in the sense of the definition we've given. Compare to Remark \ref{rmk:dimension_constraint}.
        \end{rmk}
        Composition of pre-$\tau_{\leq n}$-morphisms with $\kBS$-singularities can be defined the same as in the non-singular case, with exactly the same proofs of associativity and well-definedness, except all framed functions which are chosen must be chosen to be compatible with the singularities. Composing pre-$\tau_{\leq n}$-morphisms with $\kBS$-singularities with bilinear maps is entirely analogous.
        \begin{rmk}
            In order to avoid the above truncation issue, we could amend our definition to allow boundary faces of the form $N \times P_1 \times \ldots \times P_k$ where each $P_i \in \bM^k$, in which case products of manifolds with this structure still have this structure; this however is not necessary for the current application.\par 
            One can make an analogous construction for Baas-Sullivan singularities modelled after any set of closed manifolds $\bM$; we conjecture that the resulting category is equivalent to (the homotopy category of) the category of perfect modules over $R$, where $R$ is the $\bE_1$-quotient of $MO$ by each $[M] \in \Omega_* \cong \pi_*MO$ for $M \in \bM$.
        \end{rmk}
    \subsection{Bordism with singularities}
        Let $X$ be a space.
        \begin{defn}
            Let $M$ and $N$ be two closed manifolds with $\kBS$-singularities. A \emph{$\kBS$-bordism} from $M$ to $N$ is a compact manifold $W$ with faces with $\kBS$-singularities, with a system of essential faces given by $\{M, N\}$.\par 
            The \emph{$i$th $\kBS$-bordism group of $X$}, $\Omega^{\kBS}_i(X)$, is the set of pairs $(M, f)$, where $M$ is a closed manifold with $\kBS$-singularities and $f:M \rightarrow X$ is a continuous map which is compatible with the singularities, modulo the equivalence relation $\sim$, where we say $(M, f) \sim (M', f')$ if there is a $\kBS$-bordism $W$ from $M$ to $M'$, along with a map $F: W \rightarrow X$ extending $f$ and $f'$ which is again compatible with the singularities.
        \end{defn}
        We write $\Omega_i^{\kBS}$ and $\Omega_i$ for $\Omega_i^{\kBS}(*)$ and $\Omega_i(*)$ respectively. Recall that $\Omega_*(X)$ is naturally an $\Omega_*$-module, though $\Omega^{\kBS}_*$ is not necessarily a ring.
        \begin{prop}\label{prop:kBS k}
            Let $X$ be any space. The group $\Omega_k^{\kBS}(X)$ is zero.
        \end{prop}
        \begin{proof}
            Let $f:M \to X$ represent a class in $\Omega_k^{\kBS}(X)$. $M$ is a closed $k$-dimensional $\kBS$-manifold, and so is non-singular.\par 
            Let $W = M \times [0,1]$, viewed as a $\kBS$-manifold with singular boundary face $M \times \{1\}$, and $g: W \to X$ given by $f$ on the first factor. This is a bounding $\kBS$-manifold witnessing the vanishing of $[(f,M)]$.
        \end{proof}
        \begin{prop}\label{prop:kBS-bordism}
            Let $X$ be any space. The natural map $\theta: \Omega_{k+1}(X) \rightarrow \Omega_{k+1}^{\kBS}(X)$ which sends $(M, f: M \to X)$ to itself, has kernel $\Omega_k \cdot \Omega_1(X)$.
        \end{prop}
        \begin{proof}
            Choose a class in $\ker \theta$, represented by a pair $(M, f: M \to X)$, where $M$ is a closed $k$-manifold and $f$ is a continuous map. Then there is some $\kBS$-manifold $(W, N_P, \iota)$ with faces, such that its only essential face is $M$ (and itself), and a map $F: W \to X$ which is compatible with the singularities. This means that, viewed as a smooth manifold (i.e. without $\kBS$-singularities), $(W, F)$ is a bordism from $(M, f)$ to 
            $$\left(\bigsqcup\limits_P N_P \times P, F\right)$$
            and so we must have that 
            $$[(M, f)] = \sum\limits_P [(N_P, F)] \cdot [P]$$
            in $\Omega_{k+1}(X)$, using the fact that on each $N_P \times P$, $F$ is constant in the $P$ direction. Clearly the right hand side lives in $\Omega_k \cdot \Omega_1(X)$.\par 
            Conversely, any product $N \times Q$ of a 1-manifold $N$ and some $Q \in \bM^k$, along with a map $f: N \times Q \to X$ which is constant in the $Q$-direction, represents a class in $\ker \theta$, since it bounds the $\kBS$-manifold with faces $N \times Q \times [0,1]$ where $\iota$ is the inclusion $N \times Q \times \{1\} \hookrightarrow N \times Q \times [0,1]$, along with the map to $X$ given by $f$.
        \end{proof}
    \subsection{Framings and singularities}\label{sec:BS fram}
        We would like to equip our manifolds with singularities with framings which are suitably compatible with the singularities. Roughly speaking, a stable framing on a manifold with $\kBS$-singularities consists of a stable framing on the underlying manifold with faces, which respects the product structure over each singular face. Note however that ``respecting the product structure'' will in fact consist of more data.
        \begin{defn}\label{def:kBS-framed}
            Let $M=(M,N_P, \iota)$ be a $\kBS$-manifold with faces. We treat each singular boundary face $N_P \times P$ as an unbroken face.\par 
            Given a stable isomorphism $E \oplus TM \to F$ with stabilising bundle $\bT$, we say it's \emph{compatible with the singularities} if we are also given embeddings $\bT^P \to \bT$ over each $N_P \times P$, along with stable isomorphisms $st^N_P: E \oplus TN_P \oplus \bR \sigma_P \to F-\bR^k$ with stabilising bundle the orthogonal complement of $\bT^P$,  such that the following diagram commutes over each $N_P \times P$:
            \[\xymatrix{
                E \oplus TN_P \oplus \bR\sigma_P \oplus TP \ar[d]_\psi \ar[r]_-{st^N_P \oplus st^P} & 
                F-\bR^k\oplus \bR^k \ar[d]\\
                E \oplus TM \ar[r] & 
                F
            }
            \]
            where the right vertical arrow uses the natural stable isomorphism $-\bR^k \oplus \bR^k \to 0$ with stabilising bundle 0. Note that this definition requires the extra data of the $st^N_P$.
        \end{defn}
        From this, we can define the framed bordism groups with $\kBS$-singularities, which we denote by $\Omega^{fr; \kBS}_*(X)$ for any space $X$. As before, there is a natural map $\Omega^{fr}_*(X) \to \Omega^{fr;\kBS}_*(X)$, though $\Omega^{fr; \kBS}_*:=\Omega^{fr;\kBS}_*(*)$ is not a ring.

        Similarly to Proposition \ref{prop:kBS k} and by the same proof with framings incorporated, we have
        \begin{prop}\label{prop:fram kBS k}
            Let $X$ be any space. The group $\Omega_k^{fr;\kBS}(X)$ is zero.
        \end{prop}
         \begin{prop}\label{prop:kBS-framed-bordism}
            The natural map $\Omega_{k+1}^{fr}(X) \to \Omega_{k+1}^{fr; \kBS}(X)$ has kernel $\Omega_k^{fr}(\ast)\cdot\Omega_1^{fr}(X)$.
        \end{prop}
        \begin{proof}
            The proof is a minor variation on that of Proposition \ref{prop:kBS-bordism}. If $M$ is a framed closed $\kBS$-manifold and a map $f: M \to X$, such that $(M,f)$ vanishes in $\Omega_{k+1}^{fr; \kBS}(X)$, then there is some framed $\kBS$-manifold $(W, N_P, \iota)$ mapping to $X$, with $M$ as its only essential face, and such that the framing on $W$ and the map $W \to X$ are both compatible with the singularities. Thus, since the framing is product-like on each $N_P\times P$ and the map to $X$ is constant on $P$, it follows that $(M,f)$ lies in $\Omega_k^{fr}(\ast)\cdot\Omega_1^{fr}(X)$. The other direction is exactly as in \emph{op.cit.}  
        \end{proof} 
        \begin{defn}
            Let $\cS: \cF \to \cG$ be a $\kBS$-$\tau_{\leq n}$-morphism (respectively bordism). A \emph{framing} on $\cS$ consists of the same data as a framing on a non-singular morphism (resp. bordism), along with the extra data that each stable framing $V_x \oplus I^\cS_{xy} \to \bR^i \oplus V_y$, where $i=1$ (resp. $i=2$),  is compatible with the singularities, and that the embeddings $\bT^P \to \bT^\cS_{xy}$ over each singular face are compatible with the embeddings $\bT^\cS_{xx'y},\bT^\cS_{xy'y} \to \bT^\cS_{xy}$.\par 
            For concreteness, we spell out what this definition says explicitly.\par 
            Write $\cS_{xy} = (\cS_{xy}, N^\cS_{P;xy}, \iota^\cS_{xy})$ for the structure of a $\kBS$-manifold on each $\cS_{xy}$ and let $I^{N^\cS}_{P;xy}:= TN^\cS_{P;xy} \oplus \bR \tau_y$.\par 
            Then the new extra data consists of embeddings $\bT^P \to \bT^\cS_{xy}$ along each $N_{P;xy} \times P$ such that the following diagrams commute over appropriate singular faces:
            \begin{equation*}
                \xymatrix{
                    \bT^P \ar[r] \ar[dr] & 
                    \bT^\cS_{xx'y} \ar[d] \\
                    & \bT^\cS_{xy}
                }
                \,\,\,\,
                \xymatrix{
                    \bT^P \ar[r] \ar[dr] &
                    \bT^{\cS}_{xy'y} \ar[d]\\
                    & \bT^{\cS}_{xy}
                }
            \end{equation*}
            along with stable isomorphisms
            $$st^{N^\cS}_{P;xy}: V_x \oplus I^{N^{\cS}}_{P;xy} \oplus \bR\sigma_P \to \bR^i \oplus V_y - \bR^k$$
            with stabilising bundles the orthogonal complements of $\bT^P$ in $\bT^\cS_{xy}$ (where $i=1$ (resp. $i=2$)), such that the following diagram commutes over each $N^\cS_{P;xy} \times P$:
            \[\xymatrix{
                V_x \oplus I^{N^\cS}_{P;xy} \oplus \bR \sigma_P \oplus TP \ar[r]_{st^{N^\cS}_{P;xy} \oplus st^P}\ar[d]_\psi &
                \bR^i \oplus V_y - \bR^k \oplus \bR^k \ar[d]\\
                V_x \oplus I^\cS_{xy} \ar[r]_{st^\cS}&
                \bR^i \oplus V_y
            }
            \]
            where $i=1$ (resp. $i=2$) and the right vertical arrow uses the natural stable isomorphism $-\bR^k \oplus \bR^k \to 0$.\par 
            Note that this definition implies that the following diagrams commute. Over a broken face $\cS_{xx';y} \subseteq \cS_{xy}$:
            \[\xymatrix{
                V_x \oplus I^\cF_{xx'} \oplus I^{N^\cS}_{P;x'y} \oplus \bR\sigma_P \ar[r]_{st^\cF} \ar[d]_\psi &
                V_{x'} \oplus I^{N^\cS}_{P;x'y} \oplus \bR \sigma_P \ar[d]_{st^{N^\cS}}\\
                V_x \oplus I^{N^\cS}_{P;xy} \ar[r]_{st^{N^\cS}} &
                \bR^i \oplus V_y - \bR^k
            }
            \]
            after extending along along the relevant embedding of stabilising bundle (where $i=1$ (resp. $i=2$)) as well as a similar diagram over each $\cS_{x;y'y}$.\par 
            If $\cS$ is a bordism, then over an unbroken face $\cW_{xy} \subseteq \cS_{xy}$:
            \[\xymatrix{
                V_x \oplus I^{N^\cS}_{xy} \oplus \bR \sigma^\cW \oplus \bR \sigma_P \ar[d]_\psi \ar[r]_{st^{N^\cW}} &
                \bR \oplus \bR \sigma^\cW \oplus V_y \ar[d]_{=}\\
                V_x \oplus I^{N^\cS}_{xy} \oplus \bR \sigma_P \ar[r]_{st^{N^\cS}} &
                \bR^2 \oplus V_y
            }
            \]
        
        \end{defn}
        Note that if $W$ is a framed manifold with $\kBS$ singularities, and $(f,s)$ a framed function on $W$, then since $f$ and $s$ are constant in the $P$-direction of any singular face, the diagram in Definition \ref{def:kBS-framed} restricts directly to a corresponding diagram involving $T(f^{-1}(1))$ and $T(f|_{N\times P})^{-1}(1) = T((f|_N)^{-1}(1) \times P)$; thus we needn't impose any further condition for $(f,s)$ to be compatible with the framing.\par 
        Now all of the results and definitions in Sections \ref{sec:flow cat} and \ref{sec: flow alg} hold for morphisms and bordisms with $\kBS$-singularities (possibly with framings), as long as one truncates appropriately, i.e. works with $\tau_n$-morphisms everywhere where $2k-1>n$, with identical proofs. We illustrate this in the following case.\par 
        Let $\cW: \cF \to \cG$ and $\cV: \cG \to \cH$ be framed $\kBS$-pre-$\tau_{\leq n}$ morphisms, with $2k>n$. We define $\tilde{\cQ}_{xz}=\tilde{\cQ}_{xz}(\cW,\cV)$ as in Equation (\ref{tilde Q}). We first analyse the structure of $\tilde{\cQ}_{xz}$ as a $\kBS$-manifold with faces more closely.\par 
        For each $x \in \cF$, $y \in \cG$ and $z \in \cH$ with $|x|-|z| \leq n$, for dimension reasons at most one of $N^\cW_{P;xy}$ and $N^\cV_{P;yz}$ can be non-empty. If $N^\cW_{P;xy}$ is non-empty, we let $N^{\cW \times \cV}_{P;xyz} = N^\cW_{P;xy} \times \cV_{yz}$, if $N^\cV_{P;yz}$ is non-empty we let $N^{\cW \times \cV}_{P;xyz} = \cW_{xy} \times N^\cV_{P;yz}$, and if neither is non-empty we set $N^{\cW \times \cV}_{P;xyz}$ to also be empty. Similarly we let $N^{\cW \times \cV}_{P;xyy'z}$ be $N^\cW_{P;xy} \times \cV_{yy';z}$ or $\cW_{x;yy'} \times N^\cV_{P;y'z}$ (or empty) as appropriate. Then each $\cW_{xy} \times \cV_{yz}$ is a $\kBS$-manifold with faces, with singular boundary faces given by the $N^{\cW \times \cV}_{P;xyz} \times P$.\par 
        Let $N^{\tilde{\cQ}}_{P;xz}$ be the coequaliser of the following diagram
        $$\bigsqcup\limits_{y,y' \in \cG} N^{\cW\times \cV}_{P;xyy'z} \times [0,\eps)^2 \rightrightarrows \bigsqcup\limits_{y \in \cG} N^{\cW \times \cV}_{P;xyz}\times [0,\eps)$$
        where the maps are the restrictions of the maps in (\ref{tilde Q}). Then $\tilde{\cQ}_{xz}$ is a $\kBS$-manifold with faces, with singular boundary faces given by the $N^{\tilde{\cQ}}_{P;xz}$.\par 
        Choosing framed functions compatible on overlaps as in Section \ref{sec:unor flow cat} lets us define the composition of $\cW$ and $\cV$ as unoriented $\kBS$-pre-$\tau_{\leq n}$-morphisms. To equip this composition with framings, we frame each $N^{\tilde{\cQ}}_{P;xz}$ by defining stable isomorphisms
        $$st^{N^{\tilde{\cQ}}}: V_x \oplus I^{N^{\tilde{\cQ}}}_{P;xy} \oplus \bR \sigma_P \to \bR^2 \oplus V_y - \bR^k$$
        as follows. We first choose, inductively in $|x|-|z|$, stabilising bundles $\bT^{\tilde \cQ}_{xz}$ over each $\tilde \cQ_{xz}$, along with embeddings of the stabilising bundle over each face (this is $\bT^P$ for some $P$ over each singular face, and as in Lemma \ref{lem:fram Q} over each essential face) into $\bT^{\tilde \cQ}_{xz}$ which are compatible on overlaps.\par 
        Over each $N^\cW_{P;xy} \times \cV_{yz} \times [0,\eps)$, we define $st^{N^{\tilde{\cQ}}}$ to be the following composition:
        $$V_x \oplus I^{N^{\tilde{\cQ}}}_{P;xz} \oplus \bR \sigma_P \xrightarrow{\psi^{-1}} V_x \oplus I^{N^\cW}_{P;xy} \oplus I^\cV_{yz} \oplus \bR \sigma_P \xrightarrow{st^\cV \circ st^{N^\cW}} \bR \oplus \bR \oplus V_z - \bR^k$$
        extended to have stabilising bundle $\bT^{\tilde \cQ}_{xz}$, where the first copy of $\bR$ on the target comes from $st^{N^\cW}$ and the second from $st^\cV$. \par 
        These glue together to compatibly frame each $\tilde \cQ_{xz}$ compatibly with the singularities, just as in Lemma \ref{lem:fram Q}; we then frame the composed morphism in the usual way.\par 
        In particular, this shows there is a category $\Flow^{fr}_{\tau_{\leq n}, \kBS}$ whenever $2k-1>n$ with objects framed flow categories; we write morphisms in this category as $[\cdot, \cdot]^{fr}_{\tau_{\leq n}, \kBS}$. We may also compose morphisms with singularities with non-singular bilinear maps in exactly the same way. \par 
        From Propositions \ref{prop:fram kBS k} and \ref{prop:kBS-framed-bordism}, it follows that
        $$[*[i], *]^{fr}_{\tau_{\leq n},{(n-1)}\mathrm{-BS}} \cong \begin{cases}
            \Omega^{fr}_i & \textrm{ if } i \leq n-2\\
            0 & \textrm{ if } i = n-1\\
            \Omega^{fr,(n-1)\textrm{-BS}}_n\textrm{, which }\Omega^{fr}_{n-1} \cdot \Omega^{fr}_1 \textrm{ dies in} & \textrm{ if } i=n\\
            0 &\textrm{ if } i > n
        \end{cases}$$ 
    \section{Spectral Donaldson-Fukaya category}\label{sec: Sp Don Fuk}
        In this section we prove Theorem \ref{thm:exists}, by defining a unital associative category $\scrF(X;\bS)$ enriched in graded abelian groups, which we call a \emph{spectral Donaldson-Fukaya category}, and  furthermore  give the proof of Theorem \ref{thm:main}. The morphism spaces of $\scrF(X;\bS)$ are modules over the  framed bordism ring $\Omega_*^{fr}(\ast)$ of a point. Note however that any direct interpretation in terms of the category of  spectra is at this point conjectural.\par 
       
       The category $\scrF(X;\bS)$ is defined when $X$ admits a  stable trivialisation of $TX$, as a complex vector bundle, i.e. for some $k$ we have:
        \begin{equation} \label{eqn:polarisation}
        \Psi: TX \oplus \bC^k \xrightarrow{\cong} \bC^{n+k}
        \end{equation}
        The construction of $\scrF(X;\bS)$ appeals to the existence  of smooth structures on the relevant moduli spaces of holomorphic curves, and of stable framings of those moduli spaces, due to Fukaya, Oh, Ohta and Ono and Large. This material is subsequently reviewed in Section \ref{Sec:technical}. 
        
        \subsection{Objects}
            Suppose $X$ is stably framed, with a fixed trivialisation \eqref{eqn:polarisation}. 
                Consider a closed exact Lagrangian $L \subseteq X$, equipped with a nullhomotopy $h$ of the stable Gauss map determined by $\Psi$:
                $$\Psi_L: L \rightarrow U(n+k)/O(n+k) \rightarrow U/O$$
                where the first arrow sends $p$ in $L$ to the totally real subspace $\Psi(T_pL \oplus \bR^k) \subseteq \bC^{n+k}$ and the second map is the stabilisation map.

                 \begin{lem}
                Given $L$ as above,  the nullhomotopy $h$ of the stable trivialisation $\Psi_L$ defines a brane structure on $L$,  in the sense of \cite{Seidel:book}, canonical up to an even shift.
            \end{lem}
            \begin{proof}
            The isomorphism $\psi: TX \oplus \bC^k \to \bC^{n+k}$ gives a preferred trivialisation of $\Lambda^{top}_{\bC}(TX)$ and hence a preferred volume form $\eta$ with associated phase function
            \[
            \alpha: \mathrm{Gr}_{\mathrm{Lag}}(TX) \to S^1.
            \]
            Since a stabilisation of $\psi_L$ is null, 
            there is some $K>k$ for which the map $x \mapsto \psi_L(T_pL \oplus \bR^K)$ is homotopically trivial $L \to U(n+K)/O(n+K)$, which implies that the induced map $\alpha: L \to S^1$ is homotopic to a constant map. It therefore admits a lift to $\bR$; since the phase function is defined by a volume form and not a quadratic volume form, this is well-defined up to translation by an even integer (equivalently, the stable framing canonically orients the Lagrangian). 
          Finally, the stable framing $\psi_L$ induces a spin structure on $L$.  This completes the brane data in the sense of \cite{Seidel:book}.
            \end{proof}

                \begin{defn}\label{defn:spectral_brane}
                    A \emph{spectral Lagrangian brane} is a triple $(L,h,\eta)$ where $L$ is a Lagrangian submanifold with stable framing $\Psi_L$, $h$ is a nullhomotopy of $\Psi_L$ and $\eta$ is an associated grading.
                \end{defn}
                
            The objects of $\scrF(X; \bS)$ are defined to be spectral Lagrangian branes. Given a spectral Lagrangian brane $(L, h, \eta)$, we will often drop $\{h,\eta\}$ from the notation.
            
        \subsection{Morphisms}\label{sec:morphisms}
            For every pair $L$, $K$ of spectral Lagrangian branes, we choose regular Floer data $(H_t, J_t)$, for $t$ in $[0, 1]$, as in \cite[Section 5.2]{Large}. We define a framed flow category $\overline{\cM}^{LK}$ as follows.\par 
            The objects of $\overline{\cM}^{LK}$ are Hamiltonian chords from $L$ to $K$, meaning maps $x: [0, 1] \to X$, such that
            \begin{enumerate}
                \item $x(0)$ lies in $L$.
                \item $x(1)$ lies in $K$.
                \item $dH_t = \omega(\cdot, \dot{x}(t))$
            \end{enumerate}
            The grading $\overline{\cM}^{LK} \to \bZ$ is given by the Conley-Zehnder index, as in \cite{Seidel:book}.\par 
            For $x$, $y$ in $\overline{\cM}^{LK}$, we define the (uncompactified) moduli space $\cM^{LK}_{xy}$ to be the space of maps $u: \bR_s \times [0,1]_t \to X$, such that
            \begin{enumerate}
                \item $u(\bR \times \{0\})$ lies in $L$.
                \item $u(\bR \times \{1\})$ lies in $K$.
                \item $u(-\infty, \cdot) \to x$ and $u(+\infty, \cdot) \to y$ exponentially fast in any $C^k$ norm.
                \item $\partial_s u + J_t(\partial_t u - X_t(u)) = 0$
            \end{enumerate}
            quotiented by the free $\bR$-action by translation. For $x_0, \ldots, x_i$ in $\overline{\cM}^{LK}$, we write $\cM^{LK}_{x_0 \ldots x_i}$ for $\cM^{LK}_{x_0 x_1} \times \ldots \times \cM^{LK}_{x_{i-1} x_i}$.\par 
            We define 
            $$\overline{\cM}^{LK}_{xy} :=
            \bigsqcup\limits_{\substack{i \geq 1 \\
            x_0 = x, \ldots, x_i = y \in \overline{\cM}^{LK}}} \cM^{LK}_{x_0 \ldots x_i}$$
            equipped with the Gromov topology. The concatenation maps are given by inclusions of the compactification strata. \par
For each Hamiltonian chord $x$, we have a stable vector space $V(x)$ as constructed in Section \ref{Sec:technical} coming as a fibre of the index bundle over a space of abstract Floer caps for the chord.  The linearisations of the $\overline{\partial}$-operators on Floer solutions assemble to define an index bundle $\Ind_{xy} \to \overline{\cM}^{LK}_{xy}$, which glue associatively over boundary faces.
            
            \begin{thm} 
                The $\overline{\cM}^{LK}_{xy}$ admit the structures of compact smooth manifolds with corners, such that $\overline{\cM}^{LK}$ forms a framed flow category.  Furthermore, the stable vector spaces $V_x$ satisfy $I(x,y) \oplus V(y) \cong V(x)$ compatibly with gluing stable  isomorphisms, hence define a framing on this flow category.
            \end{thm}

            \begin{proof}
                The first statement follows from Theorem \ref{thm:smooth_structure}. The second result follows from Propositions \ref{prop: two types of index} and \ref{prop:framing_data}.
            \end{proof}
            \begin{defn}
                For each $L,K$, we define the morphisms from $L$ to $K$ of degree $i$, written $\scrF_i(L, K)$, to be
                $[*[i], \overline{\cM}^{LK}]^{fr}$. 
            \end{defn}

\begin{lem}
    The morphism  groups of $\scrF(X;\bS)$ are naturally graded modules over $\Omega^{fr}_*(\ast)$, so there is a homomorphism
    \[
    \Omega^{fr}_i \otimes \scrF_j(L,K) \longrightarrow \scrF_{i+j}(L,K).
    \]
\end{lem}

\begin{proof}
    Given a framed flow module $\cW$, i.e. a framed morphism of flow categories $\ast$ to $\overline{\cM}^{LK}$, and a closed framed manifold $P$, we obtain another framed flow module by multiplying all the manifolds with faces in $\cW$ by $P$. It is straightforward that this operation descends to framed bordism classes. 
\end{proof}

        \subsection{Composition}\label{sec:composition}
            For each triple $L$, $K$, $M$ of spectral Lagrangian branes, we define a bilinear map 
            $$\overline{\cM}^{LKM}: \overline{\cM}^{KM} \times \overline{\cM}^{LK} \to \overline{\cM}^{LM}$$
            as follows.\par 
            We choose three boundary marked points $\zeta_0^+, \zeta_1^-, \zeta_2^-$ on the disc $D^2$, labelled anticlockwise. We label the components of $\partial D^2 \setminus\{\zeta_i^\pm\}_i$ by $L, K, M$ respectively, ordered clockwise, so that $\zeta^+_0$ is between the components labelled by $L$ and $M$.\par 
            We choose strip-like ends $\eps^+_0, \eps_1^-, \eps_2^-$ near $\zeta^+_0, \zeta^-_1, \zeta^-_2$ respectively, with domains depending on the sign. This means $\eps^\pm_i : \bR_\pm \times [0,1] \hookrightarrow D^2$ parametrises a holomorphic punctured neighbourhood of $\zeta^\pm_i$. We choose these neighbourhoods to be disjoint.\par 
            We choose regular perturbation data $(K, J_z)$, where $J_z$ is an appropriate almost complex structure for $z$ in $D^2 \setminus \{\zeta^\pm_i\}_i$, and $K$ is a 1-form in $\Omega^1(D^2 \setminus \{\zeta^\pm_i\}_i, C^\infty(X, \bR))$, as in \cite[Page 64]{Large}. These are in particular chosen to be appropriately compatible with the Floer data chosen for each pair of Lagrangians in Section \ref{sec:morphisms}. $K$ determines a 1-form $Y$ in $\Omega^1(D^2 \setminus \{\zeta^\pm_i\}_i, C^\infty(X, TX))$, as in \cite[Page 105]{Seidel:book}.\par 
            For $x$ in $\overline{\cM}^{LK}$, $y$ in $\overline{\cM}^{KM}$ and $z$ in $\overline{\cM}^{LM}$, we define the (uncompactified) moduli space ${\cM}^{LKM}_{yx;z}$ to be the space of maps $u: D^2 \setminus \{\zeta_i^\pm\}_i \to X$, such that
            \begin{enumerate}
                \item $u$ sends each component of $\partial D^2 \setminus \{\zeta^\pm_i\}_i$ to the Lagrangian it is labelled by.
                \item $u(\eps^\pm_i(\pm \infty, \cdot)) \to z, y \,\mathrm\,x$ exponentially fast in any $C^k$ norm, for $i = 0, 1 \,\mathrm{or}\, 2$, respectively.
                \item $(du-Y)^{0,1} = 0$
            \end{enumerate}
            We define
            $$\overline{\cM}^{LKM}_{yx;z} := \bigsqcup\limits_{\substack{ i,j,k \geq 1\\
            x_0 = x, \ldots, x_i \in \overline{\cM}^{LK}\\
            y_0 = y, \ldots, y_j \in \overline{\cM}^{KM}\\
            z_0, \ldots, z_k = z \in \overline{\cM}^{LM}}} \cM^{LK}_{x_0 \ldots x_i} \times \cM^{KM}_{y_0 \ldots y_j} \times \cM^{LKM}_{y_j x_i; z_0} \times \cM^{LM}_{z_0 \ldots z_k}$$
            to be the stable map compactification, 
            equipped with the Gromov topology.
            \begin{thm}\label{thm:floer_product}
                The $\overline{\cM}^{LKM}_{yx;z}$ admit the structures of compact smooth manifolds with corners, for which there are gluing-compatible stable isomorphisms 
            \[
            V_x \oplus V_y \oplus I^{\overline{\cM}^{LKM}}_{yx;z} \to V_z 
            \]
            such that $\overline{\cM}^{LKM}$ yields a framed  bilinear map $\overline{\cM}^{KM} \times \overline{\cM}^{LK} \to \overline{\cM}^{LM}$.
            \end{thm}

            \begin{proof}
                This follows from  Propositions \ref{prop:smooth_structure_general} and (the bilinear maps analogue of) Lemma \ref{lem:two types of index for morphisms} and  Proposition \ref{prop:framing_data}.
            \end{proof}
            \begin{defn}
                For each $L,K,M$, the composition map in $\scrF(X; \bS)$ is given by
                $$\overline{\cM}^{LKM}_*: \scrF_i(K,M) \otimes \scrF_j(L,K) \to \scrF_{i+j}(L,M)$$
            \end{defn}
        \subsection{Associativity}\label{sec: assoc}
            For each quadruple $L, K, M, N$ of spectral Lagrangian branes, we define an associator in order to show that composition for these objects is associative.\par 
            Let $\cR^\circ = (-\infty, \infty)$. This has a family of Riemann surfaces with 4 boundary marked points living over it 
            $\cS \to \cR$, defined as follows.\par 
            We define the fibre $\cR_r$ over $r$ in $(-\infty, \infty)$ to be the strip $\bR \times [0,1]$, with marked points at $\pm \infty$, $0$ and $(r, 1)$.\par 
            We extend this to a family over $\cR := [-\infty, \infty]$ by setting the fibres over $\pm \infty$ to be appropriate nodal discs. Choosing charts $[-\infty, -1) \cong (-1, 0]$ and $(1, \infty] \cong [0, 1)$ via $r \mapsto \frac1r$ in both cases endows $\cR$ with a smooth structure, and gluing at the nodes of the fibres over $\pm \infty$ allows us to endow this family with the structure of a degenerating family of Riemann surfaces with marked points, as in \cite[Page 68]{Large}.\par 
            We name the boundary marked points of each $\cS_r$ by $\zeta^+_0, \zeta^-_1, \zeta^-_2, \zeta^-_3$, ordered anticlockwise. We label the components of each $\partial \cS_r \setminus \{\zeta^\pm_i\}_i$ by $L,K,M,N$ respectively, ordered clockwise, so that $\zeta^+_0$ is between the components labelled by $L$ and $M$.\par 
            We choose disjoint strip-like ends $\eps^\pm_i$ near $\zeta^\pm_i$ on each $\cS_r$, along with strip-like ends near the nodes in each nodal fibre $\cS_{\pm \infty}$, and we choose consistent perturbation data $(K_r, Z_{z,r})$ for each $r$ in $\cR$ extending the Floer data chosen in Section \ref{sec:morphisms}, and agreeing with the perturbation data chosen in Section \ref{sec:composition} over the nodal fibres (see \cite[Page 68]{Large} for further details).\par 
            For $x$ in $\overline{\cM}^{LK}$, $y$ in $\overline{\cM}^{KM}$, $z$ in $\overline{\cM}^{MN}$ and $w$ in $\overline{\cM}^{LN}$, we define the (uncompactified) moduli space
            ${\cM}^{LKMN}_{zyx;w}$ to be the space of pairs $(u, r)$, where $u:\cS_r \setminus \{\zeta^\pm_i\}_i \to X$, such that
            \begin{enumerate}
                \item $u$ sends each component of $\partial \cS_r\setminus \{\zeta^\pm_i\}_i$ to the Lagrangian it is labelled by.
                \item $u(\eps^\pm_i(\pm, \cdot) \to z, y, x, w)$ exponentially fast in any $C^k$ norm, for $i=0,1,2\,\mathrm{or}\,3$ respectively.
                \item $(du-Y)^{0,1}=0$
            \end{enumerate}
            We define the compactification
            $$\overline{\cM}^{LKMN}_{zyx;w} := \bigsqcup\limits_{\substack{i,j,k,l \geq 1\\
            x_0=x,\ldots,x_i \in \overline{\cM}^{LK}\\
            y_0=y,\ldots,y_j \in \overline{\cM}^{KM}\\
            z_0=z,\ldots,z_k \in \overline{\cM}^{MN}\\
            w_0,\ldots,w_l=w \in \overline{\cM}^{LN}}} 
            \cM^{LK}_{x_0 \ldots x_i} \times \cM^{KM}_{y_0 \ldots y_j} \times
            \cM^{MN}_{z_0 \ldots z_k} \times
            \cM^{LKMN}_{z_k y_j x_i; w_0} \times 
            \cM^{LN}_{w_0 \ldots w_l}$$
            equipped with the Gromov topology.
            \begin{thm}\label{sm:associators}
                The $\overline{\cM}^{LKMN}_{zyx;w}$ admit the structures of compact smooth manifolds with corners, such that $\overline{\cM}^{LKMN}$ forms an associator for the tuple
                \[
                (\overline{\cM}^{MN}, \overline{\cM}^{KM}, \overline{\cM}^{LK}, \overline{\cM}^{KN}, \overline{\cM}^{LM}, \overline{\cM}^{LN}, \overline{\cM}^{KMN}, \overline{\cM}^{LKM}, \overline{\cM}^{LKN}, \overline{\cM}^{LKN}, \overline{\cM}^{LMN})
                \]
                which admits the structure of a framed associator.
                \end{thm}
            From Theorem \ref{sm:associators} and Lemma \ref{associators}, it follows that
            \begin{cor}
                The category $\scrF(X; \bS)$ is associative.
            \end{cor}
        \subsection{Units}
            Let $L$ be a spectral Lagrangian brane. Choosing appropriate Floer data on a disc with a single marked point, Theorem \ref{thm:smooth_structure}  gives us a morphism of framed flow categories $\overline\cM^L: * \rightarrow \overline\cM^{LL}$, which determines a class $[\overline\cM^L]$ in $\scrF_0(L, L)$.
            \begin{lem}\label{idempotent}
                $\overline\cM^L$ is idempotent.
            \end{lem}
            \begin{proof}
                Gluing the disc with one marked point on the boundary to a disc with three marked points on the boundary (and using Theorem \ref{thm:smooth_structure} and Proposition \ref{prop:smooth_structure_general}) shows that the morphism of flow categories $\overline\cM^{LLL} \circ \overline\cM^L$ is framed bordant to the morphism of flow categories obtained by a continuation map from $\overline\cM^{LL}$ to itself. The same gluing argument as in \cite{Seidel:book}, but also appealing to  the smooth structures mentioned previously, implies that this morphism of flow categories is idempotent up to framed  bordism.
            \end{proof}
            \begin{lem}\label{unit is iso}
                Post-composition with $\overline\cM^L$ is surjective.\par 
                More precisely, for any other brane $K$, the linear map
                $$\overline\cM^{KLL}_*(\cdot, [\overline\cM^L]): \scrF_*(K,L) \rightarrow \scrF_*(K,L)$$
                is surjective. 
            \end{lem}
            \begin{proof}
                A gluing argument using Proposition \ref{prop:smooth_structure_general} shows that $[\overline\cM^{KLL} \circ \overline\cM^L]: \overline\cM^{KL} \rightarrow \overline\cM^{KL}$ can be represented by a continuation morphism coming from a homotopy of Floer data. By \cite{Seidel:book} these induce isomorphisms in the Fukaya category over any field and hence in the integral Fukaya category, so $[\tau_{\leq 0}(\overline\cM^{KLL} \circ \overline\cM^L)]$ is an isomorphism in $\Flow_{\tau_{\leq 0}}$. The result follows by Theorem \ref{thm:onto}. 
            \end{proof}
            \begin{cor}
                $\scrF(X; \bS)$ is a unital category, where the unit for each $L$ is $[\overline\cM^L]$.
            \end{cor}
            \begin{proof}
                Choose a class $a$ in $\scrF_*(L, K)$. By the previous result, there is some class $b$ with $a = b \circ [\overline\cM^L]$. By Lemma \ref{idempotent}, 
                \[
                a = b \circ [\overline\cM^L] = b \circ [\overline\cM^L] \circ [\overline\cM^L] = a \circ [\overline\cM^L]
                \] A similar argument shows $[\overline\cM^K] \circ a = a$.
            \end{proof}

            \begin{cor} \label{cor:projecting_units}
                An element $x\in \scrF(L,L)$ is invertible exactly when $\tau_{\leq 0}(x) \in CF(L,L;\bZ)$ is invertible.
            \end{cor}

            \begin{proof}
                One direction is straightforward. The other follows from Theorem \ref{thm:onto}, which implies that if $a\in \scrF_i(L,L) = [\ast[i],\overline\cM^{LL}]^{fr}$ and composition with $\tau_{\leq 0}a$  induces a quasi-isomorphism on $CF(L,L;\bZ)$ then composition with $a$ is surjective. In particular the identity $[\overline\cM^L]$ then lies in the image of composition with $a$ so $a$ is invertible.
            \end{proof}

\subsection{Open-closed maps}\label{sec: OC}

Let $\cF$ and $\cG$ be flow categories and $\cW: \cF \to \cG$ be a morphism. Fix a background manifold $X$. We say that $\cW$ \emph{lives over $X$} (or is `equipped with an $X$-valued local system') if there are smooth evaluation maps 
\[
\ev_{xy}^X : \cW_{xy} \to X
\]
which are compatible with breaking, meaning that there are commutative diagrams
\[
\xymatrix{
\cF_{xy} \times \cW_{yz} \ar[r] \ar[d]_{\ev_{yz}^X \circ (\mathrm{proj})} & \cW_{xz} \ar[d]^{\ev_{xz}^X} \\ X \ar@{=}[r] & X
}
\]
and 
\[
\xymatrix{
\cW_{xy} \times \cG_{yz} \ar[r] \ar[d]_{\ev_{xy}^X \circ (\mathrm{proj})} & \cW_{xz} \ar[d]^{\ev_{xz}^X} \\ X \ar@{=}[r] & X
}
\]
where $(\mathrm{proj})$ denotes projection to the $\cW$-factor of the product.   
There is an analogous notion of a bilinear map $\cW: \cF\times \cG \to \cH$ of flow categories living over $X$ (i.e. equipped with evaluation maps to $X$ compatibly with all breaking).

A \emph{homotopy} $\cR$ between bilinear maps $\cW^i: \cF \times \cG \to \cH$, with $i=0,1$, comprises manifolds with faces $\cR_{xy;z}$ of dimension $|x|+|y|-|z|+1$ and living over $[0,1]$, together with identifications
\begin{equation} \label{eqn:homotopy}
\ev^{-1}(0) = \cW^0_{xy;z} \quad \ev^{-1}(1) = \cW^1_{xy;z}.    
\end{equation}
If $\cW^0$ and $\cW^1$ live over $X$ then the homotopy lives over $X$ if there are evaluation maps $\cR_{xy;z} \to X$ compatible with breaking and also with the identifications of \eqref{eqn:homotopy}.

The above definitions carry over unchanged to the case of framed morphisms of framed flow categories, and framed homotopies of framed bilinear maps (the evaluation maps do not need to be compatible with any particular data on the target manifold $X$).

Recall that if $L$ is a spectral Lagrangian brane then the unit $[\overline\cM^L]$ is defined via a morphism $e: * \to \overline\cM^{LL}$ which counts holomorphic discs with one boundary puncture (viewed as an output, at which we turn on a Hamiltonian term so the output is a chord). Choosing a negative rather than positive strip-like end at the puncture, we also obtain a morphism 
\[
\Phi^L: \overline\cM^{LL} \to \ast[-n].
\] 
Now fix an interior marked point in the disc. Evaluation at that point shows that the morphism $\Phi$ lives over $X$.  The composition
\[
\cC: \Phi^L\circ e: * \to *[-n]
\]
comprises a single closed framed manifold $\cC_{**}$ of dimension $n$, which inherits a map to $X$.  

\begin{lem}\label{lem:bordant_to_inclusion}
    The manifold $\cC_{**}$ is framed bordant over $X$ to the inclusion $L \hookrightarrow X$.
\end{lem}

\begin{proof}
    Gluing shows that the composition $\Phi_L \circ e$ is given by a framed morphism $* \to *[-n]$ of flow categories for which the unique moduli space is the space of  perturbed holomorphic discs with no boundary punctures and with boundary on $L$. Turning off the Hamiltonian term in the Floer equation, this morphism is framed bordant to one in which all such discs are constant by exactness.
\end{proof}

Now consider spectral branes $L, K$. We consider Floer data on the infinite strip $\bR \times [0,1]$ with Lagrangian boundary conditions $L,K$ and we fix an additional boundary marked point. Depending on which boundary  component the marked point lies on, we obtain two bilinear maps
\begin{equation} \label{eqn:two_bilinear_maps}
\cW^L: \overline\cM^{LK} \times \overline\cM^{KL} \to \ast \qquad \textrm{and} \qquad \cW^K: \overline\cM^{LK} \times \overline\cM^{KL} \to \ast.
\end{equation}
As usual these admit the structure of framed bilinear maps.
\begin{lem}
    The bilinear maps from \eqref{eqn:two_bilinear_maps} are homotopic over $X$.
\end{lem}

\begin{proof}
    The bilinear maps live over $X$ by evaluation at the boundary marked point.  A homotopy of the given bilinear maps over $X$ is obtained by considering the Gromov-compactified moduli spaces of strips which have a single marked point which lies somewhere on a fixed path $\gamma(t)$ between the two boundary components (say $\{0\} \times [0,1] \subset \bR \times [0,1]$); cf. Figure \ref{Fig:CO_homology_relation}.  The moduli spaces defining the homotopy live over $[0,1]$ by recording the point $t$ on the path specifying the conformal structure of the domain, and the faces lying over $0,1$ are clearly those associated to $\cW^K$ and $\cW^L$.
\end{proof}

To relate the preceding two constructions, recall that we have a bilinear composition
\[
\mu^2_{KLK}: \scrF_i(L,K) \otimes \scrF_j(K,L) \to \scrF_{i+j}(K,K).
\]

\begin{lem}\label{lem:identity_shows_up}
   $\Phi^K \circ \mu^2_{KLK}(x,y) = \cW^K(x,y)$, and similarly for $L$.
\end{lem}

\begin{proof}
    This follows from the gluing argument illustrated in Figure \ref{Fig:bilinear_gluing}.
\end{proof}

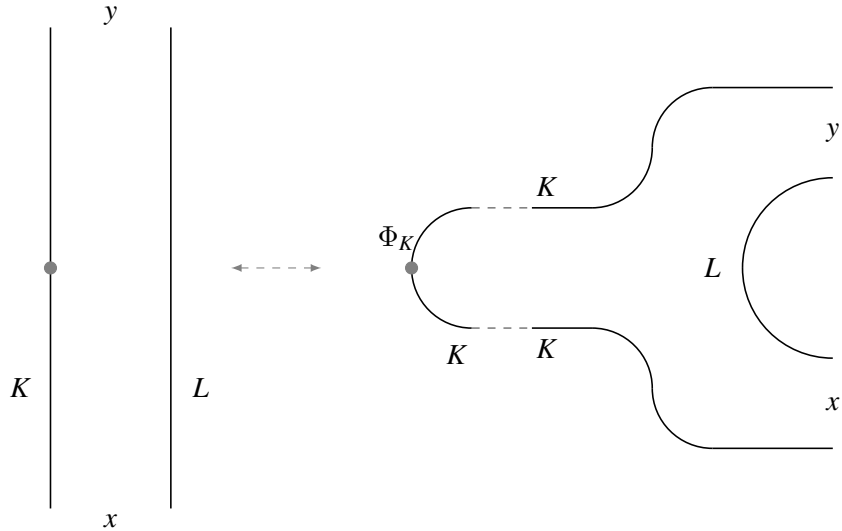
\begin{figure}[ht]
\begin{center} 
\begin{tikzpicture}[scale=0.8]

\draw[semithick] (-3,-4) -- (-3,4);
\draw[semithick] (-1,-4) -- (-1,4);
\draw[fill,gray] (-3,0) circle (0.1cm);
\draw (-2,-4.2) node {$x$};
\draw (-2,4.2) node {$y$};
\draw (-3.5,-2) node {$K$};
\draw (-0.5,-2) node {$L$};

\draw[dashed, color=gray][<->] (0,0) -- (1.5,0);

\draw[semithick] (4,1) arc (90:270:1);
\draw[fill,gray] (3,0) circle (0.1);
\draw (3.75,-1.45) node {$K$};
\draw (2.75,0.5) node {$\Phi_K$};
\draw[semithick] (5,1)--(6,1);
\draw[semithick] (5,-1)--(6,-1);
\draw[semithick,gray,dashed] (4,1)--(5,1);
\draw[semithick,gray,dashed] (4,-1)--(5,-1);

\draw[semithick] (6,1) arc (-90:0:1);
\draw[semithick] (7,2) arc (180:90:1);
\draw[semithick] (8,3) -- (10,3);
\draw[semithick] (6,-1) arc (90:0:1);
\draw[semithick] (8,-3) arc (270:180:1);
\draw[semithick] (8,-3) -- (10,-3);
\draw[semithick] (10,1.5) arc (90:180:1.5);
\draw[semithick] (8.5,0) arc (180:270:1.5);
\draw (5.25,-1.35) node {$K$};
\draw (5.25,1.35) node {$K$};
\draw (8,0) node {$L$};
\draw (10,-2.25) node {$x$};
\draw (10, 2.25) node {$y$};

\end{tikzpicture}
\end{center}
\caption{Gluing relates a bilinear map to the product composed with the unit{\label{Fig:bilinear_gluing}}}

\end{figure}

\begin{cor}\label{cor:fundamental class}
    If  spectral Lagrangian branes $L$ and $K$ are isomorphic in $\scrF(X; \bS)$ then they are framed bordant over $X$. In particular they define the same class in $\Omega_n^{fr}(X)$.
\end{cor}

\begin{proof}
    This follows from combining Lemma \ref{lem:bordant_to_inclusion} with Lemma \ref{lem:identity_shows_up}.
\end{proof}

\begin{rmk}
    A choice of Morse function $f: X \to \bR$ defines a framed flow category $\overline\cM^f$ as in Example \ref{ex:mor} (see \cite{Wehrheim} for a proof of smoothness and \cite{CJS} for the framings). Defining $CM_i(X; \Omega^{fr}) := [*[i], \overline\cM^f]^{fr}$, there are then natural morphisms $\mathcal{OC}_L: \scrF_i(L,L)\to CM_i(X; \Omega^{fr})$ obtained from taking fibre products of the moduli spaces defining $[\overline\cM^L]$ with unstable manifolds of critical points of $f$, where we evaluate the punctured discs at a fixed interior point. One can reformulate the previous argument in terms of these more classical open-closed maps (and their length two cousin), but the chosen formulation avoids the need to relate Morse theory on $X$ to $X$ itself.
\end{rmk}

        \subsection{Floer theory with Baas-Sullivan singularities}

Assume $2k>n+1$. We define a category $\tau_{\leq n}\scrF(X;\bS)_{\kBS}$ whose objects are spectral Lagrangian branes, but whose morphism groups are given by the groups
\[
[\ast[i], \overline{\cM}^{LK}]_{\tau \leq n, \kBS}^{fr}
\]
An associative composition then arises from Section \ref{sec:BS fram}, using that the spaces of holomorphic triangles and discs with four marked points gave a framed bilinear map and corresponding framed associator. Analogous to Corollary \ref{cor:fundamental class} we have

\begin{cor}\label{cor:class-in-kBS-framed}
    If $L$ and $K$ are spectral Lagrangian branes which are quasi-isomorphic in $\tau_{\leq n}\scrF(X;\bS)_{\kBS}$ then their fundamental classes agree in $\Omega_n^{fr;\kBS}(X)$.
\end{cor}

\subsection{Comparisons}

We assume as usual that we have fixed a stable trivialisation of $TX$, and that $L$ and $K$ are spectral Lagrangian branes.  Since these carry gradings, a choice of regular Floer data $(H_t,J_t)$ for $(L,K)$ gives rise to a `classical' Floer cochain complex $CF^*(L,K)$ as constructed in \cite{Seidel:book}, as well as to the Floer flow category $\overline\cM^{LK}$.

\begin{lem}\label{lem:floer_is_floer}
    There is an identification of complexes 
    \[
    CM_*(\overline\cM^{LK}) = CF^{-*}(L,K).
    \]
\end{lem}

\begin{proof}
    Basically tautological. Note that our grading conventions are homological but still with a degree zero unit and graded product, which ensures that $CM_i = CF^{-i}$.
\end{proof}

\begin{lem}\label{lem:forgetful_respects_products}
    Under the equivalence of Lemma \ref{lem:floer_is_floer}, the product induced by the framed bilinear map $\overline\cM^{LKM}$ in Theorem \ref{thm:floer_product} agrees with the Floer triangle product.
\end{lem}

\begin{proof}
    Clear.
\end{proof}

\begin{lem}
    Suppose $2k>n+1$. There are forgetful functors 
    \[
    \scrF(X;\bS) \to \tau_{\leq n}\scrF(X;\bS) \to \tau_{\leq n}\scrF(X;\bS)_{\kBS}. 
    \]
    Moreover, when $n=0$, there is an embedding 
    \[
    \tau_{\leq 0}\scrF(X;\bS) \hookrightarrow \scrF(X;\bZ)
    \]
    with image the full subcategory of exact Lagrangians which admit some spectral brane structure.
\end{lem}

\begin{proof}
    The forgetful functors are clear. That morphisms in the  $\tau\leq 0$-Fukaya category agree with morphisms in the integral Fukaya category follows from Lemma \ref{lem:floer_is_floer}, and that that identification is multiplicative follows from Lemma \ref{lem:forgetful_respects_products}. 
\end{proof}

\subsection{Proof of Theorem \ref{thm:main}}

We recall the statement:

\begin{thm} 
    Let $X$ be a Liouville manifold with $TX$ stably trivial as a complex vector bundle, and fix a stable trivialisation $\phi$. Let $L, K\subset X$ be closed exact Lagrangian integer homology spheres  whose stable Gauss maps are stably trivial compatibly with $\phi$. If $L$ and $K$ admit gradings such that they represent quasi-isomorphic objects of $\scrF(X;\bZ)$ then they represent the same class in the quotient $\Omega_n^{fr}(X)/(\Omega_1^{fr}(X) \cdot \Omega_{n-1}) = \Omega_n^{fr; (n-1)\mathrm{-BS}}(X)$. In particular, by Proposition \ref{prop:kBS-bordism}, $[L]$ and $[K]$ agree in $\Omega_n^{fr}(\ast) / \mathrm{image}(\eta)$.
\end{thm}

The hypotheses imply that there is a spectral Donaldson-Fukaya category $\scrF(X;\bS)$ which is unital and associative, and which contains objects $L$ and $K$ whose brane structures lift classical gradings.  We assume those gradings are chosen so that $L$ and $K$ are quasi-isomorphic in $\scrF(X;\bZ) = \tau_{\leq 0}\scrF(X;\bS)$.

Let $\alpha \in \tau_{\leq 0}\scrF_n(L,K)$ and $\beta \in \tau_{\leq 0}\scrF_n(K,L)$ be inverse quasi-isomorphisms (with homological grading, so the product maps $CF_i \otimes CF_j \to CF_{i+j-n}$). 

\begin{lem}
    $\alpha$ and $\beta$ admit lifts $\tilde{\alpha} \in \scrF_n(L,K;\Omega^{fr;(n-1)-BS})$ and $\tilde{\beta}
     \in \scrF_n(K,L;\Omega^{fr;(n-1)-BS})$ respectively.
\end{lem}

\begin{proof}
According to (the framed version of) Theorem \ref{Obstruction Theorem} the obstruction to lifting $\alpha$ to a $\tau_{\leq 1}$-morphism lies in
\[
\Hom (CM_{*+2}(\ast[n]), CM_*(\overline\cM^{LK}; \Omega_{1}^{fr}))
\] 
The domain vanishes unless $* = n-2$, in which case the target vanishes using that $HF_{n-2}(L,K)=0$. Iteratively, for degree reasons (using that $L$ and $K$ are integer homology spheres), the unique obstruction comes when lifting from a $\tau_{\leq n-2}$ to a $\tau_{\leq n-1}$ morphism, and is valued in $HM_0(\overline\cM^{LK}, \Omega_{n-1}^{fr})$. This vanishes in $\Omega^{fr;(n-1)-BS}_{n-1}$.
\end{proof}

\begin{lem}
    The compositions $\tilde{\alpha} \circ \tilde{\beta}$ and $\tilde{\beta} \circ \tilde{\alpha}$ are both units.
\end{lem}

\begin{proof}
    This follows from Corollary \ref{cor:projecting_units}.
\end{proof}

\begin{lem}
    There are inverse quasi-equivalences 
    \[
    \tilde{\alpha}' \in \scrF_n(L,K;\Omega^{fr;(n-1)-BS})\] and \[\tilde{\beta}' \in \scrF_n(K,L;\Omega^{fr;(n-1)-BS}).
    \]
\end{lem}

\begin{proof}
    Let
\[
\tilde{\beta}' = \tilde{\beta} (\tilde{\alpha}\tilde{\beta})^{-1}.
\]
Then $\tilde{\alpha} \tilde{\beta}' = \id$, but also if $\gamma := \tilde{\beta}' \tilde{\alpha}$ then $\gamma$ is idempotent. But
\[
\gamma(1-\gamma) = 0 \Rightarrow \gamma = 1
\]
since $\gamma$ is by construction a unit. Thus, we have found a quasi-inverse for $\tilde{\alpha}$, as required.
\end{proof}

The conclusion of Theorem \ref{thm:main} now follows from Corollary \ref{cor:class-in-kBS-framed}.
 
\subsection{Endomorphisms}
Though not required for the proof of Theorem \ref{thm:main}, in this section we identify the endomorphism groups in $\scrF$.
\begin{prop}\label{prop: pss}
    Let $L$ be a spectral Lagrangian brane. \par 
    Then there are isomorphisms of abelian groups 
    $$\scrF_{i-n}(L,L) \cong \Omega^{fr}_i(L)$$
\end{prop}
We provide a sketch proof, noting that to make the final step rigorous requires some Morse-Bott gluing.
\begin{proof}[Sketch proof]
    We consider the morphisms $\overline \cM^L:* \to \overline \cM^{LL}$ and $\Phi^L: \overline \cM^{LL} \to *[-n]$, which both live over $X$ by evaluating at a choice of boundary marked points. By choosing the regular Floer data and the boundary marked point for $\overline \cM^{L}$ appropriately, i.e. choosing a boundary point away from the support of the Hamiltonian term, we can assume that the evaluation map lands in $L$ (rather than some small Hamiltonian translate). \par 
    We define maps $\alpha: \Omega_i^{fr}(L) \to \scrF_{i-n}(L,L)$ and $\beta: \scrF_{i-n}(L,L) \to \Omega_i^{fr}(L)$ as follows.\par 
    Let $b=(s: M^i \to L, TM \simeq \bR^i)$ represent a class in $\Omega_i^{fr}$; generically perturbing $s$ if necessary lets us assume that it is transverse to all $\overline \cM^L_{*x}$.\par 
    We define $\alpha(b)$ to be represented by the morphism $*[i-n] \to \overline \cM^{LL}$ determined by the manifolds given by the (transverse) fibre products $M \times_L \overline \cM^L_{*x}$; this acquires a natural framing from those on $M$ and $\overline \cM^L$. A similar fibre product construction applied to bordisms shows that this descends to a well-defined map $\alpha$.\par 
    Let $(\cA: *[i-n] \to \cM^{LL})$ represent a class in $\scrF_{i-n}(L,L)$. Then we define $\beta(\cA)$ to be represented by $(\Phi^L \circ \cA)_**$, which lives over $L$ via the evaluation maps of $\Phi^L$.\par 
    Then $\beta \circ \alpha$ sends $b$ as above to $M \times_L \cC_{**}$, where $\cC = \Phi^L \circ \overline \cM^L$ as in Section \ref{sec: OC}, therefore by Lemma \ref{lem:bordant_to_inclusion} $\beta \circ \alpha$ is the identity.\par 
    To see that $\alpha \circ \beta$ is the identity, we observe that it sends $\cA$ to $\cA \circ \cW$, for the morphism $\cW: \overline\cM^{LL} \to \overline \cM^{LL}$ with $\cW_{xy} = \Phi^L_{x*} \times_L \overline \cM^L_{*y}$, generically perturbing so this is always a transverse fibre product; then a Morse-Bott argument shows that $\cW$ is bordant to a continuation map which is bordant to multiplication by the unit.
\end{proof}
An identical argument, with $\overline\cM^{LL}$ replaced by the framed flow category $\overline\cM^f$ of a Morse function $f:M\to \bR$ on a closed manifold, and further replacing $\overline\cM^L$ and $\Phi^L$ by the morphisms given by the compactified ascending and descending manifolds, shows that the groups $[*[i], \overline \cM^f]^{fr}$ agree with the framed bordism groups of $M$.

   \section{Floer flow category technicalities}\label{Sec:technical}

        \subsection{Smooth structures on spaces of strips} We summarise some results of Large \cite[Sections 6,7]{Large}, which build on work of Fukaya-Oh-Ohta-Ono \cite{FO3:smoothness}, on smooth structures on compactified moduli spaces of Floer solutions.
        
        Let $L,K$ be exact Lagrangians and fix Floer data $(H_t,J_t)$ for which $L$ and $\phi_{H_t}^1(K)$ meet transversely, and moduli spaces of Floer solutions are cut out transversely. We write $\cW_{xy}$ for the moduli space of Floer solutions, $\cM_{xy}$ for its quotient by the translation reparametrization action and $\overline{\cM}_{xy}$ for the Gromov compactification of $\cM_{xy}$. 
        
        Let $x_0,\ldots,x_n$ be Hamiltonian chords from $L$ to the time-one image of $K$, and 
        let $u \in \cW_{x_0x_n}$ be a Floer strip. Recall that a point $z=(s,t) \in \bR\times [0,1]$ is `regular' for $u$ if 
        \begin{equation} \label{eqn:regular_point}
        u^{-1}(u(\bR \times \{t\}) = z; \ \partial_s u|_z \neq 0; \ u(z) \neq x_i(t), \, i \in \{0,n\}.
        \end{equation}
        \cite{FHS} shows that the set of regular points is open and dense. Now suppose $u$ has Maslov index $d+1$, so that $\cW_{x_0x_n}$ has dimension $d+1$ and $\cM_{x_0,x_n}$ has dimension $d$. Pick $z_0, z_1,\ldots, z_d$ in $\bR\times [0,1]$ with the properties that they are all regular, and if $z_i = (s_i,t_i)$ then $s_0 < s_1 < \ldots < s_d$.

        \begin{defn}\label{defn:complete_system}
            A collection $\textbf{H} = \{H_0,\ldots,H_d\}$ of smooth codimension one open hypersurfaces of $X$ form a \emph{complete system} (of hypersurfaces) for $u$ if the following two conditions hold:
            \begin{enumerate}
                \item the 
            evaluation map $\cW_{x_0,x_n} \to X^{d+1}$ sending $v \mapsto (v(z_0), v(z_1),\ldots,v(z_d))$ is transverse to $H_0 \times H_1 \times \cdots \times H_d$ at $u$;
            \item  for each $0 \leq i \leq d$, $u^{-1}(H_i)$ and $\bR\times \{t_i\}$ intersect transversely at $z_i$ in $\bR\times [0,1]$, and moreover $z_i$ is the only intersection of $u^{-1}(\overline{H}_i)$ and $\bR\times \{t_i\}$.
    \end{enumerate}
        \end{defn}

        \begin{lem} \label{lem:complete_systems_exist}
        Every $u$ admits a complete system $\textbf{H}$.
        \end{lem}
\begin{proof}
    Pick points $z_i = (s_i,t_i)$ in $\bR\times [0,1]$ so that the map $\ev_z$ given by evaluation at $z = (z_0,\ldots, z_d)$ satisfies
    \[
\Gamma(u^*TX) \supset T_u \cW_{x_0,x_n} \stackrel{D(ev_z)}{\longrightarrow} \oplus_{i=0}^d T_{u(z_i)}X
    \]
    is injective. This is an open condition in the $\{z_i\}$, so we can perturb and re-index to ensure that the points are also regular and ordered (in particular all the co-ordinates $s_i$ are distinct). We now pick real codimension one linear subspaces
    \[
    \tilde{H}_i \subset T_{u(z_i)}X
    \]
    so that $\oplus_i \tilde{H}_i \subset \oplus_i T_{u(z_i)}X $ is transverse to $\ev_z$.  We now pick arbitrary open hypersurfaces $H_i$ with tangent spaces $\tilde{H}_i$ at $u(z_i)$. This means that the first condition of Definition \ref{defn:complete_system} holds by construction.  Shrinking the $H_i$ to smaller open subsets if necessary, we can ensure the second condition also holds. \end{proof}

    A choice of complete system for $u$ gives rise to an open set $u \in V_{\textbf{H}} \subset \cW_{x_0,x_n}$ with the property that for any $v \in V_{\textbf{H}}$ there is a unique point $z_i(v) = (s_i(v),t_i) \in \bR\times [0,1]$ with $v(z_i(v)) \in H_i$, and so that the map 
    \[
    V_\textbf{H} \to \bR^{d+1}, \ v \mapsto \, (s_0(v),\ldots,s_d(v))
    \]
    is a diffeomorphism to its image. Taking $V_\textbf{H}$ to be translation invariant, one gets a a chart $U_{\textbf{H}}$ for $\cW(x_0,x_n)$ near $u$ with local diffeomorphism
    \begin{equation} \label{eqn:transit_times}
    U_{\textbf{H}} \to \bR^d, \ v \mapsto (s_0(v)-s_1(v),s_1(v)-s_2(v),\ldots, s_{d-1}(v)-s_d(v)).
    \end{equation}
The previous lemma implies that such charts cover.

    \begin{rmk}
        These co-ordinates are similar to the `transit time' co-ordinates used by Barraud and Cornea in \cite{Barraud-Cornea}, obtained from recording the $s$-co-ordinates of intersections with a collection of level sets of the action functional.
    \end{rmk}

Now let $(u^1,\ldots, u^n) \in \cM_{x_0,x_1}\times\cdots\times \cM_{x_{n-1},x_n}$ be a broken solution, and for each $1 \leq \ell \leq n$ fix a complete system $\textbf{H}^\ell = \{H^{\ell}_0,\ldots, H^{\ell}_{\mu(u_{\ell})-1}\}$ for $u^{\ell}$. We obtain an open set 
\[
U = U_{\textbf{H}^1} \times \cdots \times U_{\textbf{H}^n} \subset \cM_{x_0,x_1}\times\cdots\times \cM_{x_{n-1},x_n}
\]
A convenient feature of the charts defined by the open sets $U_\textbf{H}$ associated to complete systems is that they yield `associative gluing maps' in the following sense:

\begin{prop}[{\cite[Corollary 6.8 and Proposition 6.9]{Large}} ]
    For sufficiently large $T_0 \gg 0$ there is a gluing map 
    \[
    G: U \times [T_0,\infty)^{n-1} \to \cM_{x_0,x_n}
    \]
    with the following properties:
    \begin{enumerate}
        \item
    it is smooth onto a neighbourhood of the point (at infinity) $(u^1,\ldots,u^n)$;
    \item  $G(u,T)$ is the unique trajectory which admits a parametrization with $v(z_i^{\ell}(u,T)) \in H_i^{\ell}$ for every $i,\ell$;
    \item gluing is strictly associative in the sense that if we fix $T_m \in [T_0,\infty)$ for $m \not \in \{\ell_1,\ldots, \ell_r\}$ and further fix
    \[
    u^{\alpha} = (u^{l_{\alpha-1}+1}, \ldots, u^{l_{\alpha}}) \in \cM_{x_{l_{\alpha-1}}, x_{l_{\alpha}}}; \quad T^\alpha = (T_{l_{\alpha-1}+1},\ldots, T_{l_{\alpha}})
    \]
    then $G(u,T) = G(G(u^\alpha,T^\alpha)_{\alpha = 1,\ldots,r}, (T_{l_1},\ldots, T_{l_r}))$.
        \end{enumerate}
\end{prop}

It is classical that if $w$ is obtained from pregluing $(u^1,\ldots,u^n)$ with gluing lengths $T_1,\ldots, T_{n-1}$ then 
\[
\| \overline{\partial}(w)\|_{2,k} \leq C\cdot e^{-\delta \min(\textbf{T})}
\]
for some $C>0$ and $\delta>0$. The gluing, obtained from a Newton-Picard iteration, expresses the actual neraby Floer solution as the exponential of a vector field $\xi$ along $w$ and 
\begin{equation} \label{eqn:classical_estimate}
\overline{\partial} \exp_w(\xi_{\textbf{u,T}}) = 0, \qquad \|\xi\|_{2,k} \leq  C\cdot e^{-\delta \min(\textbf{T})}
\end{equation}

Let $\psi: (0,\varepsilon] \to [T_0,\infty)$ be any diffeomorphism, in this context called a `gluing profile'. We apply the map $\psi$ co-ordinatewise to vector inputs. The associativity of gluing and usual properties of the gluing map implies that the map
\begin{equation}\label{eqn:continuous_gluing}
\prod_{\ell} U_{\textbf{H}^{\ell}} \times (0,\varepsilon)^{n-1} \to \cM_{x_0,x_n}, \qquad (u,\textbf{r}) \mapsto G(u, \textbf{T}:=\psi(\textbf{r}))
\end{equation}
extends to a well-defined continuous map
\begin{equation} \label{eqn:smooth_gluing}
\prod_{\ell} U_{\textbf{H}^{\ell}} \times [0,\varepsilon)^{n-1} \to \overline{\cM}_{x_0,x_n}
\end{equation}
Building on exponential decay estimates for Floer solutions, \cite{FO3:smoothness} proved a smoothness theorem for compactified moduli spaces of holomorphic curves. 
In the simple case where the Lagrangians are transverse (so there is no Morse-Bott gluing and there are no weighted Sobolev spaces), \cite{Large} gives a simplified  proof of the following of their results, by directly analysing the transition functions associated to changing the hypersurfaces in a complete system:
\begin{thm}[{Fukaya-Oh-Ohta-Ono \cite{FO3:smoothness}, Large \cite[Theorem 6.12]{Large}}] \label{thm:smooth_structure}
    If we use the gluing profile $\psi(r) = 1/r$ then the maps from \eqref{eqn:smooth_gluing} (associated to charts defined by complete systems of hypersurfaces) give $\overline{\cM}_{x_0,x_n}$ the structure of a smooth manifold with corners.
\end{thm}

\begin{proof}[Sketch]
 The crucial decay estimate refines \eqref{eqn:classical_estimate} by showing that
\begin{equation} \label{eqn:FO3}
\| (\nabla_{\textbf{u}})^p \frac{\partial^m}{\partial \textbf{T}^m} (\xi_{\textbf{u,T}})|_{[-S,S] \times [0,1]} \|_{2,k-p-|m|} \leq C \cdot e^{-\delta \min(\textbf{T})} \leq C e^{-\delta' / \max(\textbf{r})}, \qquad p+|m| \leq k-1
\end{equation}
where the last inequality uses the choice of gluing profile and $\delta' < \delta$ is chosen appropriately. 

Given two complete systems of hypersurfaces $\textbf{H}$ and $\textbf{H'}$, the transition map between associated charts has the shape, cf. \eqref{eqn:transit_times}
\[
(u,\textbf{r}) \mapsto ((\tilde{s}_i(u,\textbf{r}) - \tilde{s}_{i-1}(u,\textbf{r})), \tilde{\textbf{r}})
\]
for the `new' transit lengths $\tilde{s}_j$ and functions schematically of the form $\tilde{r}_j = r_j / (1+\phi(r, \tilde{s}_i - s_i)$. Smoothness of gluing amounts to the fact that the functions $\tilde{s}_i$ extend smoothly over $\textbf{r} = (0,\ldots,0)$, which is inferred from \eqref{eqn:FO3}.
\end{proof}

\begin{rmk}
    The map $\psi(r) = -\ln(r)$ does not yield a smooth manifold-with-corners structure, even though this is a `natural' profile from the viewpoint of algebraic geometry (essentially because the modulus of the annulus $\{1<|z|<r\}$ is $-\ln(r)/2\pi$ and that modulus is closely related to the smoothing parameter $\{xy = \varepsilon\}$ for a node).
\end{rmk}

The discussion up to this point adapts directly to prove that compactified regular moduli spaces of Floer cylinders are smooth manifolds with corners, using local charts which again record differences of the $s$-co-ordinates of intersections with real hypersurfaces $H_i$ where the co-ordinates lie in the cylinder rather than the strip, in direct analogy with \eqref{eqn:transit_times}.

\subsection{Smooth structures on other moduli spaces}\label{sec: smoth ext}

For moduli spaces of holomorphic triangles, where the domain $\Sigma$ is rigid and all breaking is unstable, the previous analysis applies but one must first choose a parametrisation $\bR \times (a,b) \to \Sigma $ of an open subset of the domain and work with the corresponding local co-ordinates, which allow one to define a  regular point as in \eqref{eqn:regular_point}. 

The remaining step is to deal with moduli spaces where the underlying domain moves in a non-trivial moduli space $\cR$ of holomorphic curves. The only cases relevant to this paper are where $\cR = [0,1]$, namely a moduli space of four-marked discs, used in constructing the `associator' which proves associativity of the composition operation on bilinear flow morphisms (i.e. of the Floer product), and a moduli space of discs with two boundary marked points and an interior marked point, arising in the discussion of open-closed maps.  

Let $\cF(x_{\alpha})$ denote a compactified moduli space of solutions $(p,u)$ where $u: C_p \to X$ is defined on a Riemann surface $C_p$ depending on a point $p\in\cR \cong [0,1]$ (and the $x_{\alpha}$ denote the fixed asymptotics, i.e. corners of holomorphic polygons and/or an  interior Hamiltonian orbit). Recall that regularity of the moduli space (with varying $p$) then concerns surjectivity of an extended linearised operator 
\[
D_{C_p,u}^{\ext}: \, T_p\cR \oplus W^{1,2}(C_p,u^*TX) \to L^2(\Omega^{0,1}\otimes_J u^*TX)
\]
which takes account of variation of the domain parameter, and there is a natural map
\begin{equation} \label{eqn:forget}
\ker(D_{C_p,u}^{\ext}) = T_{C_p,u}\cF(x_{\alpha}) \to T_p\cR
\end{equation}
coming from the forgetful map $\forget: \cF(x_{\alpha}) \to \cR$. There are now two different possibilities: either \eqref{eqn:forget} is trivial, or it is surjective.  Accordingly, we introduce two kinds of charts on the moduli space. Suppose $\dim \cF(x_{\alpha}) = d$:
\begin{enumerate}
    \item Suppose \eqref{eqn:forget} vanishes. This means we are in a situation much as before, since the tangent space to $\cF(x_{\alpha})$ at $(p,u)$ is spanned by variations with fixed domain. Choose $d$ points $z_1,\ldots, z_d \subset C_p$ lying in a chart $\bR \times (a,b) \subset C_p$ (for instance in a the strip-like ends), regular in the sense of \eqref{eqn:regular_point} for the given co-ordinates, and open hypersurfaces $H_1,\ldots, H_d \subset X$ with $u(z_i) \in H_i$ so that 
     \[
    \partial_s u|_{z_i} \neq 0, \ u^{-1}(H_i) \pitchfork \partial_s \ \textrm{at} \ z_i.
    \]
    and 
    \[
    \ev_\textbf{z}: \cF(x_{\alpha}) \to X^d 
    \]
    is transverse to $H_1\times \cdots \times H_d$ at $(p,u)$. We obtain an open set $U_{\textbf{H}} \subset \cF(x_{\alpha})$ with co-ordinates so that for $v\in U_{\textbf{H}}$ there are points $z_i(v) = (s_i(v),t_i)$ in the strip-like-end co-ordinates with $v(z_i(v)) \in H_i$ and $v \mapsto (s_1(v),\ldots, s_d(v)) \in \bR^d$ defines a local chart.
    \item Suppose instead that \eqref{eqn:forget} is onto at $(p,u)$. Again fix a chart $\bR \times (a,b) \subset C_p$. We now pick points $z_1,\ldots, z_{d-1} \subset C_p$ in the chart,   pick codimension one open hypersurfaces $H_1,\ldots, H_{d-1}$ of $X$ containing the $u(z_i)$,  regular in the previous sense,  and finally pick a codimension one hypersurface $X \times \{p\} = H_{\cR} \subset X\times \cR$, and ask that
    \[
    \ev_\textbf{z} \times (\forget): \cF(x_{\alpha}) \to X^{d-1} \times \cR
    \]
    is transverse to $H_1 \times \cdots \times H_{d-1} \times H_{\cR}$ at $(p,u)$. This yields an open set $U_{(\beta,\textbf{H})} \subset \cF(x_{\alpha})$ where nearby solutions $v$ are uniquely determined by the $s$-co-ordinates $s_i(v)$ of the local intersections with the $H_i$ together with their intersection with $H_{\cR}$, which in particular records the conformal parameter of the domain at $v$. 
    \end{enumerate} 

\begin{rmk}
Both cases involve picking $(d-1)$ hypersurfaces in $X$ and one in $X\times \cR$ and asking for transversality of an evaluation map to $X^d \times \cR$;  the final hypersurface is $H_d \times \cR$ in the first case and $X \times \{p\}$ in the second. 
    For spaces $\cR$ of domains which are higher-dimensional manifolds with corners, the crucial information is the rank $r$ of \eqref{eqn:forget}, and one should construct charts which involve picking $r$ codimension one hypersurfaces in the product $X \times \cR$, but we will stick to the simpler case which suffices for our purposes.
\end{rmk}

   \begin{lem}
       Charts of the previous two kinds cover the moduli space $\cF(x_{\alpha})$.
   \end{lem}

    \begin{proof}
        We have $T_{(p,u)}\cF(x_{\alpha}) \subset \Gamma(u^*TX)\oplus T_p\cR$. For linear hyperplanes $\tilde{H}_i \subset T_{u(z_i)}(X)$, let $W^2(\underline{\tilde{H}})$ denote those vector fields $\xi \in u^*TX$ with $\xi(z_i) \in \tilde{H}_i$, where $1\leq i\leq \{d-1,d\}$ so this has codimension $d$ respectively  $d-1$  depending on whether the chart has the first or second form. For suitably chosen hyperplanes, as in Lemma \ref{lem:complete_systems_exist}, a complement to the kernel $\ker(D^{ext}$ is given by $T_p\cR \oplus W^2(\underline{\tilde{H}})$  in the first case and $ W^2(\underline{\tilde{H}})$ in the second. The proof then follows as before, by taking generic hyperplanes $H_i$ in $X$ tangent to the $\tilde{H}_i$. 
    \end{proof}
     \begin{ex}
        Suppose $p$ lies in a collar neighbourhood of the boundary of $\cR$ covered by the image of a gluing map $\beta$. The proof of linear gluing shows that \eqref{eqn:forget} is onto, so a neighbourhood of the Gromov-Floer boundary is covered by charts of the second kind. The domain $C_p$ contains a distinguished `neck', coming from the gluing, and it is natural to take the $z_i$ in that neck. Nearby solutions $v$ are uniquely determined by the $s$-co-ordinates $s_i(v)$ together with the unique value $\beta(v)$ of the gluing parameter (i.e. collar co-ordinate) at $v$. \end{ex}

We use the same gluing profile $\psi(r) = 1/r$ to map the parameter space $[L_0,\infty)$ for neck length (equivalently conformal parameter near the boundary in $\cR \simeq [0,1]$) to an interval $(0,\varepsilon]$.  With this choice:

        \begin{prop} \label{prop:smooth_structure_general}
            The previous charts endow the compactified moduli space $\cF(x_{\alpha})$ with the structure of a smooth manifold with corners.
        \end{prop}

\begin{proof}[Sketch]
    The proof follows the lines of Theorem \ref{thm:smooth_structure}. The interior of the moduli space is canonically smooth, so we need only consider smoothness near the boundary. For unstable breaking the situation is exactly as in \emph{op. cit}. Also for stable breaking, the transition functions associated to nested strata in the boundary, with charts defined by a fixed choice of hyperplanes (and perhaps a neck length), are smooth because of the strict associativity of gluing, which follows from the geometric characterisation of the co-ordinates in the two kinds of chart. The new phenomenon arises when considering charts near a boundary stratum corresponding to stable breaking, when one either compares a chart involving only hypersurfaces $\{H_i\}$ and one with a neck length $\{H_i', \beta\}$, or compares two charts involving neck lengths $\beta,\beta'$.  In the first case, the estimate \eqref{eqn:FO3} again shows that the length $\beta$ depends smoothly on the co-ordinates associated to the $\{H_i\}$. In the second case, the transition function is smooth in the neck length parameter since that is \emph{intrinsic} to $\cF(x_{\alpha})$, being determined by the conformal structure of the domain.
\end{proof}

\subsection{Index bundles}

At a regular Floer solution $u: \bR \times [0,1] \to X$ with transverse Lagrangian boundary conditions $L,K$ there is a linearised operator $D_u = D(\overline{\partial}_J|_u)$, whose kernel $\ker(D_u)$ is the fibre of the tangent  bundle $T\cW_{xy} \to \cW_{xy}$ to the space of Floer solutions (where we have not quotiented by the re-parametrisation action of $\bR$).  Over the open locus of unbroken solutions, the map $\cW_{xy} \to \cM_{xy} = \cW_{xy} / \bR$ admits a unique section up to contractible choice, which induces a bundle $\Ind(x,y) \to \cM_{xy}$ called the \emph{index bundle}.

\begin{lem}
    The index bundle $\Ind(x,y) \to \cM_{xy}$ extends as a topological vector bundle to the compactification $\overline{\cM}_{xy}$. Moreover, one can construct this extension in such a way that there are  gluing isomorphisms $\psi = \psi_{xyz}: \Ind(x, y) \oplus \Ind(y, z) \to \Ind(x, z)$ over $\overline\cM_{xyz}$ which are associative, meaning the following diagram commutes over $\overline\cM_{xyzw}$.
    \[\xymatrix{
        \Ind(x, y) \oplus \Ind(y, z) \oplus \Ind(z, w) \ar[r] \ar[d] &
        \Ind(x, y) \oplus \Ind(y, w) \ar[d] \\
        \Ind(x, z) \oplus \Ind(z, w) \ar[r] &
        \Ind(x, w)
    }
    \]
\end{lem}

\begin{proof}
    See \cite[Proposition 7.3]{Large}.
\end{proof}

Suppose we now consider a space of holomorphic curves which are maps from a family $\cR$ of domains; then the tangent bundle to the moduli space at $(C_p,u)$ with $p\in \cR$ and $u: C_p \to X$ is given by the kernel of the extended operator $\ker(D_{(C_p,u)}^{\mathrm{ext}})$ introduced previously.

\subsection{Index bundles and abstract index bundles}\label{sec:2ind}
    The index bundles $\Ind(x, y)$ do not canonically contain $T\overline\cM_{xy}$. What we instead have are short exact sequences %((7.17) in [Tim])
    $$0 \rightarrow \bR\tau_{xy} \xrightarrow{\sigma_{xy}} \Ind(x, y) \rightarrow T\overline\cM_{xy} \to 0$$
    where the first arrow comes from the $\bR$-symmetry on the set of Cauchy-Riemann sections.\par 
    
    At a regular point $u$ of a moduli space of Floer solutions with asymptotics $x,y$, the index bundle $\Ind(x,y)$ is the tangent space to the space of Floer solutions, $\overline\cM_{xy}$ is the compactification of the space of solutions modulo reparametrization, and the translation vector field $\partial / \partial s$ generates the $\bR$-factor of $\Ind(x,y)$.  If $u_1 \in \overline\cM_{xy}$  and $u_2 \in \overline\cM_{yz}$ are Floer solutions, then (after suitable translation and cutting-off) there is a parameter-dependent gluing $u_R = u_1 \#_R u_2$ ($R$ being the gluing parameter). Linear gluing (cf. \cite[Chapter II.12.f]{Seidel:book}) shows that under the gluing map, 
 \[
 ((\partial/\partial s)_{u_1}, (\partial/\partial s)_{u_2}) \mapsto (\partial / \partial s)_{u_R}, \quad (-(\partial/\partial s)_{u_1}, (\partial/\partial s)_{u_2}) \mapsto v
 \]
 where $v \in \Ind(x,z)/\bR$ has non-trivial inward-pointing component to the boundary face (i.e. in the collar direction of the embedding $\cC: \overline\cM_{xy} \times \overline\cM_{yz} \times  [0,\varepsilon) \to \overline\cM_{xz}$). 
    It follows that the index bundles further satisfy that the following diagram (cf. \cite[(7.17)]{Large}) commutes
    \[
    \xymatrix{
        0 \ar[r] &
        \bR \tau_{xz} \ar[r]_{\sigma_{xz}} \ar[d]_\Delta &
        \Ind(x, z) \ar[r] \ar[d]_\psi &
        T\overline\cM_{xz} \ar[r] \ar[d] &
        0 \\
        0 \ar[r] &
        \bR \tau_{xy} \oplus \bR \tau_{yz} \ar[r]_{(\sigma_{xy}, \sigma_{yz})} &
        \Ind(x, y) \oplus \Ind(y, z) \ar[r] & 
        T\overline\cM_{xy} \oplus T\overline\cM_{yz} \ar[r] &
        0
    }
    \]
    where $\Delta$ is the diagonal map sending $\tau_{xz}$ to $\tau_{xy}+\tau_{yz}$, $\psi$ is the (analytically defined) linear gluing map and the rightmost vertical map is a choice of splitting for the natural map 
    $$dc: T\overline\cM_{xy} \oplus T\overline\cM_{yz} \rightarrow T\overline\cM_{xz}$$
    The composition $\bR\tau_{xy} \oplus \bR \tau_{yz} \to T\overline\cM_{xz}$ also sends $\tau_{xyz}$ to an inwards-pointing normal vector in $T\overline\cM_{xz}$ for the boundary component $\overline\cM_{xy} \times \overline\cM_{yz}$.\par 
    By splitting the short exact sequence above, there are isomorphisms $\Ind(x, y) \cong T\overline\cM_{xy} \oplus \bR\tau_{xy}$. We would like to choose such isomorphisms which are compatible with the abstract gluing isomorphisms for the right hand side.

    \begin{prop}\label{prop: two types of index}
        Let $\overline\cM=\overline\cM^{LK}$ be the flow category associated to a pair of spectral Lagrangian branes.  Then there are isomorphisms of vector bundles
        $$\phi = \phi_{xy}: \Ind(x, y) \to I_{xy}$$
        where $I_{xy} = I_{xy}^\cM$ is the abstract index bundle of $\overline\cM$, and the $\phi_{xy}$ respect the abstract and analytic gluing isomorphisms, in the sense that the following diagram commutes:
        \[ \xymatrix{
            \Ind(x, y) \oplus \Ind(y, z) \ar[r]_\psi \ar[d]_\phi&
            \Ind(x, z) \ar[d]_\phi \\
            I_{xy} \oplus I_{yz} \ar[r]_\psi &
            I_{xz}
        }
        \]
        and are compatible with projection to tangent spaces, in the sense that the following diagram commutes
        \[\xymatrix{
            \Ind(x,y) \ar[d]_\phi \ar[dr] & \\
            I_{xy} = T\overline\cM_{xy} \oplus \bR \tau_{xy} \ar[r] & T\overline\cM_{xy}
        }
        \]
        where the horizontal arrow is projection onto the first factor.\par 
        Furthermore, this $\phi$ is canonical up to contractible choice.
    \end{prop}
    \begin{proof}
        We construct the $\phi_{xy}$ inductively in the dimension of the $\overline\cM_{xy}$. Assume we have constructed suitably associative $\phi_{xy}$ when $|x|-|y| - 1 < k$ and pick $x, z$ such that $|x|-|z| - 1 = k$. Then we want to choose $\phi_z$ such that the following diagram commutes over $\overline\cM_{xyz}$ for all $y$.
        \[ \xymatrix{
            \Ind(x, y) \oplus \Ind(y, z) \ar[r]_\psi &
            \Ind(x, z) \\
            I_{xy} \oplus I_{yz} = \bR\tau_{xy} \oplus T\overline\cM_{xy} \oplus \bR \tau_{yz} \oplus T\overline\cM_{yz} \ar[u]_{(\phi^{-1}_{xy}, \phi^{-1}_{yz})} \ar[r]_\psi &
            I_{xz} = \bR \tau_{xz} \oplus T\overline\cM_{xz} \ar[u]_{\phi^{-1}_{xz}}
        }
        \]
        This determines $\phi_{xz}$ over all $\overline\cM_{xyz}$ (and these are compatible over all overlaps $\overline\cM_{xyy'z}$). Note that this diagram is compatible with the projection maps to $T\overline\cM_{xy} \oplus T\overline\cM_{yz}$ and $T\overline\cM_{xz}$, by \cite[(7.17)]{Large}. We then choose a splitting $T\overline\cM_{xz} \to \Ind_{xz}$ extending the one on the boundary to define $\phi_{xz}^{-1}$ over $T\overline\cM_{xz} \subseteq I_{xz}$ (which we can do since the space of such splittings is convex), and choose $\phi_{xz}^{-1}(\tau_{xz})$ to be transverse to the subspace $\phi_{xz}^{-1}T\overline\cM_{xz}$ and in the same component of such vectors as the image of the analytically-defined map $\bR \to \Ind_{xz}$ (again, we can do this since the space of such transverse vectors is convex; note the reason we don't take $\phi^{-1}(\tau_{xz})$ to be equal to the analytically-defined map $\bR \to \Ind_{xz}$ is because this would be incompatible with commutativity of the required diagrams). \par 
        The choices we made above are canonical up to contractible choice so it follows that $\phi$ is too.
    \end{proof}
    \begin{rmk}
        Proposition \ref{prop: two types of index} would be false if we further required that the following diagram commutes:
        \[ \xymatrix{
            \bR \tau_{xy} \ar[r]_{\sigma_{xy}} \ar[dr] &
            \Ind(x, y) \ar[d]_\phi \\
            & 
            I_{xy}
        } 
        \]
        where the map $\bR \tau_{xy} \to I_{xy} = \bR \tau_{xy} \oplus T\overline\cM_{xy}$ is the inclusion into the first factor.
    \end{rmk}

We have an analogue of the previous result for continuation maps rather than morphisms. Let $\cF$, $\cG$ be two flow categories from Floer theory. We already have isomorphisms $\phi: \Ind^\cF_{xx'} \cong I^\cF_{xx'}$ compatible with gluing on both sides; there is an analogous version for morphisms of flow categories from Floer theory:
\begin{lem} \label{lem:two types of index for morphisms}
    Let $\cW: \cF \to \cG$ be a morphism of Floer flow categories $\cF = \overline\cM^{LK}$, $\cG = \overline\cM^{LK'}$ associated to a continuation map of spectral Lagrangian branes from $K$ to $K'$. Then there are isomorphisms $\phi: I^\cW_{xy} \to \Ind(x,y) \oplus \bR$ such that the following diagrams commute:
    \[\xymatrix{I^\cW_{xy} \ar[r]^-\phi \ar[dr] &
    \Ind^\cW(x,y) \oplus \bR \ar[d] \\
    & T\cW_{xy}} \]
    \[\xymatrix{I_{xx'}^\cF \oplus I^\cW_{x'y} \ar[r]_\psi \ar[d]_{\phi\oplus \phi} &
    I^\cW_{xy} \ar[d]_\phi \\
    \Ind^\cF(x,x') \oplus \Ind(x',y) \oplus \bR \sigma_{x'y} \ar[r]_-\psi &
    \Ind(x,y) \oplus \bR \sigma_{xy}}\]
    and 
    similar for breaking on the right, where the $\sigma$s are generators for the extra copies of $\bR$ and $\psi$ acts as the identity on them.
\end{lem}
\begin{proof}
    Assume we've chosen $\phi$ whenever $|x|-|y|<k$. Pick $x,y$ such that $|x|-|y|=k$. We want to choose $\phi_{xy}$ such that the diagrams commute on appropriate boundary faces. (Note that by associativity, where this compatibility is enforced twice (i.e. over the codimension 2 boundary faces), these are compatible with each other). Therefore $\phi_{xy}$ is already specified over the boundary; these come from some splittings $\bar{\phi}:T\cW_{xy} \to \bR\sigma_{xy}$. Choose $\bar{\phi}$ extending this over the boundary; choose some map $\bR \tau_{xy} \to \Ind^\cW(x,y) \oplus \bR$ which is transverse to $\phi_{xy}(T\cW_{xy})$ and in the same component of such things as the identity $\bR \to \bR \sigma_{xy}$ extending that on the boundary, and take $\phi_{xy}$ to be given by $\bar{\phi}$ on the subspace $\bR \tau_{xy}$.
\end{proof}
Let $\cZ=\cM^{LKMN}$ be an associator from Section \ref{sec: assoc}. The moduli space of domains is $\cR \cong [0,1]$, and recall there is an analytic index bundle
\begin{equation}\label{eqn:extended}
  \Ind^{\cZ}(x,y,z;w) 
\end{equation}
over each $\cZ_{xyz;w}$, whose fibre at a map $u: C_p \to X$ is given by the kernel of (a virtual perturbation of) the (pointwise) Cauchy-Riemann operator over $C_p$, as in \cite[Section 7.3]{Large}. \par 
In this case, there is also an extended Cauchy-Riemann operator $D^{\mathrm{Ext}}$ (see Section \ref{sec: smoth ext} and \cite[(5.3)]{Large}); which is surjective (by the assumption that the moduli spaces are transversely cut out) with kernel $T\cZ_{xyz;w}$.\par  
The same proof as in the cases of Floer flow categories and continuation maps shows that there are isomorphisms of vector bundles:
\[
\phi:I^\cZ_{xyz;w} \to \ker(D^{\mathrm{Ext}}) \oplus \bR
\]
which are compatible with the abstract and analytic gluing isomorphisms on each side over each boundary face, such that the following diagram commutes:
\[ 
\xymatrix{
    I^\cZ_{xyz;w} \ar[r]^-\phi \ar[dr] &
    \ker(D^{\mathrm{Ext}}) \oplus \bR \ar[d]\\
    & T\cZ_{xyz;w}
}
\]
where both downwards arrows are projection maps.\par 
Using consistent surjective perturbations $D_u + f$ of the pointwise operators (adding finite-dimensional spaces to the domains of the operators), Large \cite[Section 8.5]{Large} constructs stable isomorphisms
\[
\ker(D_{(C_p,u)}^{\mathrm{Ext}}) \cong \ker(D_u + f) \oplus \bR
\]
which are compatible with breaking, and where the $\bR$ comes from a choice of trivialisation of $T\cR$. 

\subsection{Stable framings from spaces of caps\label{Sec:framings_from_caps}}

We outline Large's approach to constructing the data used to frame a Floer flow category. 
Let $X$ be an exact symplectic manifold with a stable trivialisation 
\[
\Psi: TX \oplus \bC^{N-n} \stackrel{\sim}{\longrightarrow} \bC^N
\]
and let $L$ and $K$ be transverse graded Lagrangian branes which are more explicitly equipped with Lagrangian subbundles 
\[
\Gamma_L \subset L \times [0,1] \times \bC^N
\]
with $(\Gamma_L)|_{L \times \{0\}} = \Psi(TL \oplus \bR^{N-n})$ and $(\Gamma_L)|_{L \times \{1\}} = \bR^N$, and similarly for $K$. Let $x\in \cM^{LK}$ be an intersection point of $L$ and $K$. We will associate a stable vector space $V(x)$ to $x$ as follows.  In the case where $L,K$ are not transverse, one can incorporate a Hamiltonian $H_t$ into the Floer data so that $L\pitchfork \phi_{H_t}^1(K)$ and take $x(t)$ a Hamiltonian chord; since the analogous case of Hamiltonian orbits involves this $t$-dependence we will incorporate it into the notation.

Let $Z = \bR\times[0,1]$ with co-ordinates $(s,t)$. 
The linearisation of the $\overline{\partial}$-operator at the constant Floer solution $u(s,t) = x=x(t)$ has the form
\begin{equation}\label{eqn:linearise_at_constant}
\xi \mapsto \overline{\partial}_{J_t}(\xi) + Y_t(\xi)
\end{equation}
for a translation-invariant one-form $Y_t \in \Omega^{0,1}(Z)\otimes_{\bC} x^*TX$ (this is usually written explicitly with respect to a choice of symplectic connection on $TX$). By taking the direct sum with a trivial $\overline{\partial}$-operator on $\bC^{N-n}$ (not changing the kernel or cokernel of the linearisation), we get an operator
\begin{equation} 
D_x : W^{1,2}(Z,\bC^N, T_xL \oplus \bR^{N-n}, T_xK \oplus \bR^{N-n}) \longrightarrow L^2(Z, \Omega^{0,1}\otimes_J \bC^N).
\end{equation}
Let $D_+ = D^2\backslash \{-1\}$. We now fix the following data: 
\begin{enumerate}
\item a strip-like end $\varepsilon: (-\infty,0] \times [0,1] \hookrightarrow D_+$ at the puncture; 
\item a real $L>0$ and an almost complex structure $J$ on $\bC^N \times D_+$ (taming the standard symplectic structure) and one-form $Y\in \Omega^{0,1}(D_+)\otimes_{J} \mathrm{End}(\bC^N)$ which agree with $J_t, Y_t$ from \eqref{eqn:linearise_at_constant} on the image of $\varepsilon|_{(-\infty,-L)\times[0,1])}$;
    \item a Lagrangian subbundle $\Gamma \subset \partial D_+ \times [0,1] \times \bC^N$ which interpolates between the pairs $(T_xL \oplus \bR^{N-n}, T_xK \oplus \bR^{N-n})$ on $\varepsilon(-\infty,-L)\times\{0,1\})$ at $\{0\}$ and the constant $\bR^N$ at $\{1\}$, and which agrees with $\Gamma_L$ respectively $\Gamma_K$ on $\varepsilon((-\infty,-L)\times\{0\}) \times [0,1]$ respectively $\varepsilon((-\infty,-L)\times\{1\}) \times [0,1]$.
\end{enumerate}
The data specifies an abstract real Cauchy-Riemann operator
\begin{equation} \label{eqn:abstract_operator}
D: W^{1,2}(D_+, \bC^N, \Lambda|_{D_+ \times \{0\}}) \to L^2(D_+,\Omega^{0,1}(D_+) \otimes_J \bC^N)
\end{equation}
We can fix a $q\geq 0$ and a linear map $f:\bR^q \to C^{\infty}(Z,\Omega^{0,1}(D_+) \otimes_J \bC^N)$ (with image in the space of sections vanishing over the image of $\varepsilon$) so that $D+f$ is surjective, and then define
\[
V_q(x) := \ker(D+f)-\bR^q.
\]
Clearly stabilising by increasing $q$ and replacing $f$ by $\bR^{q+1} \to \bR^q \stackrel{f}{\longrightarrow}C^{\infty}(Z,\Omega^{0,1}(D_+) \otimes_J \bC^N)$  yields an isomorphic stable vector space.

Since the flow category $\cM^{LK}$ is finite, we can pick a sufficiently large $q$ to make the operator $D=D_x$ surjective for all $x\in L\pitchfork \phi_{H_t}^1(K)$ simultaneously. 

\begin{prop}[{\cite[Section 8.4]{Large}}] \label{prop:framing_data}
    The association $x \mapsto V_q(x)$ has the following features:
    \begin{enumerate}
        \item Up to homotopy, the space of choices made in constructing $V_q(x)$ is given by the data of a sufficiently large $q$;
        \item for $q \gg 0$ there are stable isomorphisms $V_q(x) \oplus \Ind(x,y)  \to V_q(y)$ which are compatible with gluing, in the sense that the following diagram of vector bundles over $\cM_{xy} \times \cM_{yz}$ commutes:
        \[
        \xymatrix{
            V_q(x) \oplus \Ind(x, y) \oplus \Ind(y, z)  \ar[r] \ar[d] &
            V_q(y) \oplus \Ind(y,z)  \ar[d] \\
            V_q(x) \oplus \Ind(x, z) \ar[r] &
            V_q(z) 
        }
        \]
    \end{enumerate}
\end{prop}

Assuming we have taken a sufficiently large such $q$, we will suppress the $q$ and just write $V_x$ for the stable framing space associated to $x$.
\begin{rmk}
    Since in the proof of the theorem we only need to consider moduli spaces of holomorphic curves of dimension $\leq n+1$, for our purposes it would suffice to choose a sufficiently large $q$ such that all the spaces of choices involved are sufficiently highly connected, but not necessarily contractible.
\end{rmk}
The choices enumerated previously yield a contractible space $\cU(x)$ of `abstract half-planes' carrying a (trivial) stable bundle given by the index of \eqref{eqn:abstract_operator}. The proof of the Proposition relies on a gluing operation $\cU(x) \times \cM^{LK}_{xy}  \to \cU(y)$ (strictly defined on a relatively compact open subset of the domain determined by the collar structure) which is associative up to coherent homotopy.

If $\cW: \cF \to \cG$ is a continuation morphism of Floer flow categories associated to a Hamiltonian isotopy, there is a  parallel story, except the gluing theory that underlies the isomorphism $\Ind(x,y) \oplus \cV(y) \simeq \cU(x)$ with $\Ind(x,y) = T\cW(x,y)$ relates spaces of caps $\cU(x)$ and $\cV(y)$ for $\cF$ respectively $\cG$ with a space $\cW(x,y)$ of Floer continuation solutions.  

We briefly comment on the analogous results for other moduli spaces:
\begin{enumerate}
    \item For a framed bilinear map $\cR: \cF \times \cG \to \cH$ there is a (suitably associative) gluing map
    \[
    \cU(x) \times \cU(y) \times \cR_{xy;z} \longrightarrow \cU(z)
    \]
    which covers an isomorphism of index bundles
    \[
    V_x \oplus V_y \oplus I^{\cR}_{xy;z} \longrightarrow V_z \oplus \bR;
    \]

    \item for a space of discs $\cM_x$ with one incoming puncture (as used in constructing units), there is an inclusion $\cM_x \hookrightarrow \cU(x)$ which identifies the index bundle $I_x$ over $\cM_x$ with the restriction of $V_x$;

    \item for a space of discs $\cM_{\bar{x}}$ with a unique \emph{outgoing} puncture (as arise in constructing co-units), there is a gluing map $\cU(x) \times \cM_{\bar{x}} \to \cD$ where $\cD$ denotes a space of perturbed holomorphic discs on $L$ with no boundary punctures. By turning off all Hamiltonian perturbation data, and using exactness, the inclusion of constant discs $L \to \cD$ is a deformation retraction. The stable framing on $L$ then gives a stable equivalence between $I_{\bar{x}}$ and $-V_x$.

    \item For an associator $\cZ$ defined by a space $\scrR \cong [0,1]$ of holomorphic discs with four boundary punctures, there are gluing maps
\[
\cU(x) \times \cU(y) \times \cU(z) \times \cZ_{xyz;w} \longrightarrow \cU(w).
\]
\end{enumerate}
Then the stable isomorphisms from Proposition \ref{prop:framing_data}, combined with the comparisons between the analytic and abstract index bundles from Section \ref{sec:2ind}, provide the data required for the framings in Section \ref{sec: Sp Don Fuk}.

\begin{rmk}
In Proposition \ref{prop:framing_data}, the stable isomorphisms exist with trivial stabilising bundle, coming from the choice of a sufficiently large $q$ when constructing the stable vector space $V(x) = \ker(D+f)-\bR^q$. That is, once all operators have been stabilised to be surjective, the gluing maps $\cU(x) \times \cM_{xy} \to \cU(y)$ are covered by isomorphisms of actual vector bundles which on fibres are given by $L^2$-orthogonal projections on pregluings, schematically $\ker(D_u+f) + \ker(D_v) \to \ker(D_{u\#v} + g)$ for appropriate stabilisations $f,g$. In general, in all the situations we encounter, by making the original stabilisations sufficiently large (for all $x$ and over all relatively compact subsets of spaces of operators $\cU(x)$ and $\cD$ encountered in homotopies from gluing, or between compositions of continuation maps and the identity, etc), we can ensure that no stabilising bundles appear in the analytic construction. However, the general abstract framework we developed in Section \ref{Sec:framings} does require the flexibility of allowing non-trivial stabilising bundles, for instance in constructing inverses in $\Flow$.
\end{rmk}

\bibliographystyle{amsalpha}
\bibliography{Refs.bib}{}

\end{document}